\documentclass[a4paper]{amsart}

\usepackage{amscd,amsthm,amsfonts,amssymb,amsmath,esint}
\usepackage{amsmath,tikz-cd}
\usepackage[colorlinks=true]{hyperref}
\usepackage{csquotes}
\usepackage{fullpage}
\usepackage[all]{xy}
\usepackage{abstract}
\usepackage{mathtools}

\newcommand{\BA}{{\mathbb {A}}} 
\newcommand{\BC}{{\mathbb {C}}} 
\newcommand{\BE}{{\mathbb {E}}} \newcommand{\BF}{{\mathbb {F}}}
\newcommand{\BG}{{\mathbb {G}}} \newcommand{\BH}{{\mathbb {H}}}

 \newcommand{\BP}{{\mathbb {P}}}
\newcommand{\BQ}{{\mathbb {Q}}} \newcommand{\BR}{{\mathbb {R}}}

 \newcommand{\BZ}{{\mathbb {Z}}}

\newcommand{\CA}{{\mathcal {A}}} 
 \renewcommand{\CD}{{\mathcal {D}}}
 
 \newcommand{\CH}{{\mathcal {H}}}
 
 \newcommand{\CL}{{\mathcal {L}}}
\newcommand{\CM}{{\mathcal {M}}} 
\newcommand{\CO}{{\mathcal {O}}} \newcommand{\CP}{{\mathcal {P}}}
 \newcommand{\CR}{{\mathcal {R}}}
\newcommand{\CS}{{\mathcal {S}}} 
 \newcommand{\CV}{{\mathcal {V}}}
\newcommand{\CW}{{\mathcal {W}}} \newcommand{\CX}{{\mathcal {X}}}

\newcommand{\fa}{{\mathfrak{a}}} \newcommand{\fb}{{\mathfrak{b}}}
\newcommand{\fc}{{\mathfrak{c}}} 
 \newcommand{\ff}{{\mathfrak{f}}}
\newcommand{\fg}{{\mathfrak{g}}} 
 
 \newcommand{\fl}{{\mathfrak{l}}}
\newcommand{\fm}{{\mathfrak{m}}} 
 \newcommand{\fp}{{\mathfrak{p}}}
\newcommand{\fq}{{\mathfrak{q}}} \newcommand{\fr}{{\mathfrak{r}}}
\newcommand{\fs}{{\mathfrak{s}}}

\newcommand{\alg}{{\mathrm{alg}}}
\newcommand{\Aut}{{\mathrm{Aut}}}

\newcommand{\cond}{\mathrm{cond^r}}

\newcommand{\Cl}{{\mathrm{Cl}}}

\newcommand{\End}{{\mathrm{End}}}

\newcommand{\Frob}{{\mathrm{Frob}}} 

\newcommand{\N}{{\mathrm{N}}}

\newcommand{\Gal}{{\mathrm{Gal}}} \newcommand{\GL}{{\mathrm{GL}}}

\newcommand{\Hom}{{\mathrm{Hom}}}

\renewcommand{\Im}{{\mathrm{Im}}}

\newcommand{\Nm}{{\mathrm{Nm}}}

\newcommand{\Lie}{{\mathrm{Lie}}}

    \newcommand{\ord}{{\mathrm{ord}}} \newcommand{\rank}{{\mathrm{rank}}}
     \newcommand{\Pic}{\mathrm{Pic}}

    \renewcommand{\mod}{\ \mathrm{mod}\ }\renewcommand{\Re}{{\mathrm{Re}}}

    \newcommand{\Res}{{\mathrm{Res}}}

    \newcommand{\Spec}{{\mathrm{Spec}}}

    \newcommand{\Stab}{{\mathrm{Stab}}}\newcommand{\MS}{{\mathrm{MS}}}
    \newcommand{\Tate}{{\mathrm{Tate}}}
      
    \newcommand{\tr}{{\mathrm{tr}}}

    \newcommand{\vol}{{\mathrm{vol}}}

    \font\cyr=wncyr10

    \newcommand{\Sha}{\hbox{\cyr X}}
    \newcommand{\wh}{\widehat}
    
    \newcommand{\pair}[1]{\langle {#1} \rangle}

    \newcommand{\ov}{\overline}

    \newcommand{\ra}{\rightarrow} 
    \newcommand{\bs}{\backslash}
    \newcommand{\nequiv}{\equiv\hspace{-10pt}/\ }

    \theoremstyle{plain}
    \newtheorem{thm}{Theorem}[section] \newtheorem{cor}[thm]{Corollary}
    \newtheorem{lem}[thm]{Lemma}
    \newtheorem{fact}[thm]{Fact}
         \newtheorem{prop}[thm]{Proposition}
    \newtheorem {conj}[thm]{Conjecture} \newtheorem{defn}[thm]{Definition}
     \newtheorem{lem-defn}[thm]{Lemma-Definition}

\theoremstyle{remark} \newtheorem{remark}[thm]{Remark}
\theoremstyle{remark} 
\theoremstyle{remark} 

    \newcommand{\Neron}{N\'{e}ron~}
    \newcommand{\adeles}{ad\'{e}les~}
    
    \newcommand{\etale}{\'{e}tale~}

    \numberwithin{equation}{section}

\title{Stability of $p$-adic valuations of Hecke L-values}
\author{Wei He}
\address{School of Mathematical Sciences, University of Science and Technology of China, Hefei, 230026, People's Republic of China}
\email{hewei0714@ustc.edu.cn}
\begin{document}

\maketitle
\begin{abstract}
In this paper, based on Hida's methods, we establish $p$-stability results for the critical L-values of algebraic Hecke characters over CM fields in $\ell$-adic anticyclotomic twist family with $\ell\neq p$. 
\end{abstract}
\tableofcontents

\section{Introduction}
Let $K$ be a quadratic CM extension of a totally real field $F$.
For $\lambda$ a Hecke character over $K$, let $L_f(s,\lambda)$ be the finite part of its Hecke L-function. Let $p,\ell$ be two distinct rational primes. Fix a prime $\fl$ of $F$ above $\ell$ and let $\Gamma$ be Galois group of the maximal $\BZ_\ell$-free quotient of the anticyclotomic extension of $K$ with conductor $\fl^\infty$. Fix embeddings $\iota_\infty:\ov{\BQ}\ra \BC$, $\iota_p:\ov{\BQ}\ra \BC_p$. Suppose that $\lambda$ is algebraic and is critical at $0$. The $p$-stability question is whether the $p$-adic valuations of algebraic part of $L_f(0,\lambda\varepsilon)$ are equal to a constant for almost all finite order characters $\varepsilon$ of $\Gamma$. When the constant is zero, it is often called the modulo $p$ non-vanishing problem. It's an analog of Washington theorem \cite{wa} for critical CM Hecke L-values.

This question was first studied by Hida \cite{Hi1} and Finis \cite{Fi}. This paper aims to treat missing cases, see Theorem \ref{maini}. 
\subsection{Main result} \label{mainre}
We begin by formalizing the question more precisely.

Suppose $p$ is ordinary, which means every prime of $F$ above $p$ splits in $K$. Let $\Sigma$ be a CM type of $K$ that is $p$-ordinary in the sense that the finite places of $F$, $\Sigma_p:=\{v\ |\ \text{$v$ is induced by $\iota _p \sigma$}, \sigma\in \Sigma\}$ is disjoint with $\Sigma_p^c=\{v\ |\ \text{$v$ is induced by $\iota _p \sigma c$}, \sigma\in \Sigma\}$, where $c$ is the nontrivial elements in $\Gal(K/F)$. For the remainder of this paper, we will always assume the ordinary condition.

Let $\fl$ and $\Gamma$ be as given before. View the group of finite order characters on $\Gamma$, denoted by $\wh{\Gamma}$, as Hecke characters over $K$ via class field theory. For an algebraic Hecke character $\lambda$ over $K$ with $\lambda_\infty|_{K^\times}(k)= \prod_{\sigma\in \Sigma}\sigma(k)^{a_\sigma}(c\sigma(k))^{b_\sigma}$, its infinite type is defined by $w=\sum_{\sigma\in \Sigma}a_\sigma\sigma+b_\sigma c\sigma$. Consider all Hecke characters over $K$ of infinity type $k\Sigma+(1-c)\kappa$, $k\Sigma+\kappa\in \BZ_{>0}[\Sigma]$, $\kappa\in \BZ_{\geq 0}[\Sigma]$. The group $\wh{\Gamma}$ acts on this set by multiplication.

Fix a $\wh{\Gamma}$-orbit denoted by $\CX$, then it is known (cf.~Theorem~\ref{measure}) that for almost all $\lambda\in \CX$, \[\CL(\lambda)=\prod_{w\in\Sigma_p}G(\lambda_w)\frac{\pi^\kappa\Gamma_\Sigma(k\Sigma+\kappa)L_f(0,\lambda)}{\sqrt{|D_F|_\BR}\Im(\vartheta)^{\kappa}\Omega_\infty^{k\Sigma+2\kappa}}\in \ov{\BZ}_{p} \cap \ov{\BQ},\]
where (i) $\Omega_\infty\in (\BC^\times)^\Sigma$ is the CM period of a \Neron differential on an abelian scheme over $\ov{\BZ}_{p} \cap \ov{\BQ}$ of CM type $(K,\Sigma)$; (ii) $\vartheta$ is any pure imaginary element in $K$ such that $\Im(\sigma(\vartheta))>0$ for each $\sigma\in \Sigma$ and $2\vartheta\CD_{F}^{-1}$ is prime to $p$, here $\CD_F$ is the different ideal of $F$, $\Im(\vartheta)^{\sum_{\sigma\in \Sigma} n_\sigma \sigma}=\prod_{\sigma\in \Sigma}\Im(\sigma(k)^{n_\sigma})$; (iii) $G(\lambda_w)$ is the Gauss sum \[|\varpi|_w^{e}\cdot \sum_{u\in\varpi^{-e}\cdot(\CO_{K_w}/\varpi^{e}\CO_{K_w})^\times}\lambda_w(u)\psi_w(u),\] here $\varpi$ is a uniformizer of $K_w$, $|\cdot|_w$ is the normalized valuation on $K_w^\times$, $e$ is the conductor of $\lambda_w$, and $\psi_w$ is a nontrivial additive character of $K_w$ with trivial conductor; (iv) $\Gamma_{\Sigma}(\sum_{\sigma\in \Sigma} n_\sigma \sigma)=\prod_{\sigma\in \Sigma}\Gamma(n_\sigma)$.

Identify $\CX$ with $\ell^\infty$ torsion points of a split torus (cf.~\S\ref{rg}), here \enquote{almost all} means \enquote{except for finitely many} if $\rank_{\BZ_\ell}\Gamma=1$, and \enquote{Zariski dense in torus} if $\rank_{\BZ_\ell}\Gamma>1$.

\begin{defn}Let $\lambda$ be an algebraic Hecke character over $K$.
\begin{itemize}\item [(i)] The character $\lambda$ is called self-dual if \[\lambda^*|_{\BA_F^\times}=\tau_{K/F},\] where $\lambda^*:=\lambda\cdot |\cdot|_{\BA_K^\times}^{-1/2}$, and $\tau_{K/F}$ is the quadratic character over $F$ associated to $K/F$.
\item [(ii)] The character $\lambda$ is called residually self-dual if \[\wh{\lambda}|_{\BA_F^\times}\omega^{-1}\equiv\tau_{K/F}\pmod{ \fm_p},\] where $\wh{\lambda}$ is the $p$-adic avatar of $\lambda$, $\omega: \Gal(\ov{F}/F)\ra \BZ_p^\times$ is the Teichmuller character over $F$, and $\fm_p$ is the maximal ideal of $\ov{\BZ}_p$.
\end{itemize}
For a self-dual Hecke character $\lambda$, its root number is defined by $\epsilon(1/2,\lambda^*,\psi_K)$, where $\psi$ is a nontrivial character of $F\bs \BA_F$, $\psi_K=\psi\tr_{K/F}$, and $\epsilon(1/2,\lambda^*,\psi_K)=\prod_v\epsilon(1/2,\lambda_v^*,\psi_{K_v})$ is the epsilon factor defined in Tate's thesis. 
\end{defn}
For $\lambda$ residually self-dual and $p>2$, we introduce the following residual root number:
\begin{defn}\label{resr} Suppose that $p>2$ is ordinary. For $\lambda$ residually self-dual, the residual root number $\ov{\epsilon}(\lambda)\equiv \pm 1\pmod{\fm_p}$ is defined by \[\prod_{\substack{\text{$v$ is finite and}\\\text{ $K_v$ is a field}}}\frac{\varepsilon(1/2,\lambda_v^*,\psi_{K_v})}{\lambda_v^*(2\vartheta)}\pmod{\fm_p},\]where $\vartheta\in K$ is any purely imaginary element with $\Im(\sigma(\vartheta))>0$ for each $\sigma\in \Sigma$.\end{defn}
We have for each $v<\infty$ such that $K_v$ is a field, $\frac{\varepsilon(1/2,\lambda_v^*,\psi_{K_v})}{\lambda_v^*(2\vartheta)}\equiv \pm1\pmod{\fm_p}$ (cf.~\S\ref{mtd}). The residual root number coincides with the root number in self-dual case.

We have following conjecture on the $p$-stability of Hecke L-values:
\begin{conj}\
\begin{itemize}
\item[(1)] If $\CX$ is not residually self-dual, then there exists $\mu_p(\CX)\in\BQ_{\geq 0}$ such that for almost all $\lambda\in\CX$, \[\ord_p(\CL(\lambda))=\mu_p(\CX).\]
\item[(2)] If either (i) $\CX$ is self-dual or (ii) $\CX$ is residually self-dual and $p>2$, write $\CX=\CX^+\sqcup\CX^-$, where $\CX^\pm$ consists of those characters $\lambda\in \CX$ with root number (resp. residual root number) equal to $\pm1$ (resp. $\pm1$ modulo $\fm_p$) in the self-dual (resp. residually self-dual) case. Then for $\epsilon=\pm$, there exists $\mu_p(\CX^\epsilon)\in\BQ_{\geq 0}\sqcup \{\infty\}$ such that for almost all $\lambda\in\CX^\epsilon$, \[\ord_p(\CL(\lambda))=\mu_p(\CX^\epsilon).\] 
\end{itemize}
\end{conj}
Here we convention that $\ord_p(0)=+\infty$, and $\mu_p(\CX^\epsilon)=+\infty$ if $\CX^\epsilon=\emptyset$. 

\begin{remark} In the self-dual case (resp. residually self-dual case and $p>2$), if $\fl$ is inert, then the parity of the root number (resp. residual root number) depends on the parity of the conductor at $\fl$. If $\fl$ is split, then the root number (resp. residual root number) is constant. In the case where $\CX$ is self-dual, we have $\CL(\lambda)=0$ for all $\lambda\in \CX^-$, and in this case, one should ask about the $p$-divisibility of derivative L-values, which we will not touch in this paper.
\end{remark} 

Our main result is as follows:
\begin{thm}\label{maini}
Suppose that
\begin{itemize}
\item[(i)] $p$ is ordinary, $(p,2 D_F)=1$,
\item[(ii)] $(\fl,p)=1$ and $\lambda|_{F_\fl^\times}$ is unramified for $\lambda\in \CX$,
\item[(iii)] $p$ is odd if $\CX$ is residually self-dual but not self-dual,
\item[(iv)] if $\lambda$ is residually self-dual, $\fl$ inert, then $\rank_{\BZ_\ell}\Gamma=1$.
\end{itemize} Then the following hold:
\item[(1)] If $\CX$ is not residually self-dual, then for Zariski dense $\lambda\in\CX$, \[\ord_p(\CL(\lambda))=\mu_p(\CX).\]
\item[(2)] If $\CX$ is residually self-dual, then for \[\mu_p(\CX^m):=\min\{\mu_p(\CX^+),\mu_{p}(\CX^-)\},\] we have \[\ord_p(\CL(\lambda))=\mu_p(\CX^m)\]
for Zariski dense $\lambda\in \CX^m$.
\end{thm}
The invariants $\mu_p(\CX)$ and $\mu_p(\CX^\pm)$ in separated cases are defined in \S\ref{pprop}.

By definition, one has the following vanishing criteria: 
\begin{itemize}
\item [(i)] If $\lambda$ is not residually self-dual, then $\mu(\CX)=0$ if for each $\fq|\cond(\lambda)$ that non-split in $K/F$ and prime to $\fl$, either $\fq^2|\cond(\lambda_\fq)$ or $p\nmid q-1$, where $q$ is the cardinality of residue field of $\CO_{K,\fq}$.
\item [(ii)] If $\lambda$ is self-dual, $p>2$, then $\mu_p(\CX^m)=0$ if for each $\fq|\cond(\lambda)$ that is inert in $K/F$ and prime to $\fl$, either $\fq^2|\cond(\lambda_\fq)$ or $p\nmid q+1$, where $q$ is the cardinality of residue field of $\CO_{F,\fq}$. 
\end{itemize}

Compare with previous works \cite{Fi}, \cite{Hi2}, \cite{Hs1}, \cite{Oh}, our results include the following new cases:
\begin{itemize}
\item $p$ is not coprime to the conductor of Hecke characters.
\item Residually self-dual but not self-dual.
\item $\fl$ is inert in $K$ in the residually self-dual case.
\end{itemize} 
We remark that the second part $(p,2D_F)=1$ of the assumption (i) in Theorem \ref{maini} is only needed in Hida's result on density of CM points.

\subsection{Strategy of proof}
Let's briefly recall Hida's approach in \cite{Hi1}, \cite{Hi2}. The anticyclotomic twisted $L$-values can be interpreted as periods of Eisenstein series on CM points, which dates back to Damerell's formula. Katz's theory of $p$-adic modular forms \cite{Ka} provides a $p$-adic interpretation of critical L-values whenever $p$ is ordinary. Then, via density of ordinary CM points on product of Hilbert modular varieties modulo $p$ established by Hida \cite{Hi1}, the non-vanishing modulo $p$ of Hecke L-values could be reduced to the non-vanishing modulo $p$ of Eisenstein series. See \cite{Hf} for a recent update of Hida's result on density of CM points.

To prove the $p$-stability result, it is crucial to construct Eisenstein series related to Hecke characters on a suitable integral model of Hilbert modular scheme. Using Kottwitz model, which has maximal level at $p$, Hida proved the non-vanishing results for Zariski dense subset of $\wh{\Gamma}$ in the case where the conductor of $\lambda$ is only divided by split primes that are prime to $p\fl$ with $(p, 2D_F)=1$. 

Following Hida's approach, there are several related works \cite{Hs1} and \cite{Oh}. Let's briefly recall: Still using Kottwitz model, but with more representation theory, Hsieh \cite{Hs1} allowed the conductor to have non-split primes, still prime to $p$. He dealt with both non residually self-dual and
self-dual case with global root number equal to $+1$ and $\fl$ split in $K$. Note that when $\lambda$ is residually self-dual, every ramified prime in $K/F$ must divide the conductor. In the self-dual case, to get the $p$ divisibility of the Eisenstein series, Hsieh crucially used theta dichotomy for $(U(1), U(1))$ in \cite{hks}, also see \S\ref{td}. In the case $F=\BQ$, Ohta \cite{Oh} uses the $\Gamma(n)^{\text{arith}}$ model of modular curve to get non-vanishing result whenever the conductor consists of odd split primes that need not be coprime to $p$ and $p$ is odd. 

Our approach also follows Hida. To handle the case $p$ not being coprime to the conductor, one needs to use a good integral model of Hilbert modular scheme at $p$, whose level at $p$ is not maximal but still has good geometric properties so that the $q$-expansion principle holds (cf.~Proposition~\ref{q-exp}). To deal with the general totally real field case, we use a different $p$-integral model (cf.~Proposition~\ref{mode}) from Ohta. To construct test vectors on the Hilbert modular scheme with distinguished level at $p$, one needs to make full use of the machine of Waldspurger formulae for Eisenstein series.

In the residually self-dual but not self-dual case, we establish a relation between local Fourier coefficient modulo $p$ and local residual root number based on theta dichotomy, this may be viewed as modulo $p$ theta dichotomy (cf.~\S\ref{S51}--\S\ref{mtd}). Similar to the self-dual case handled by Hsieh \cite{Hs1}, the residual root number in the residually self-dual case determines which Fourier coefficients of Eisenstein series get the minimal $p$-adic valuation (cf.~Proposition~\ref{mod}). Based on this, we could deal with an interesting case in which $\fl$ is inert and Hecke characters are (residually) self-dual. In this case, half of Hecke characters in the family have (residual) root number $-1$, and half of the Fourier coefficients of the Eisenstein series vanish modulo $p$. To handle the case $p=2$ for reducing the $p$-stability question to density of CM points modulo $p$, one meets the question on $p$-indivisibility of Hecke twist sum of Eisenstein series over genus class subgroup. We give a construction of the Eisenstein series such that the $p$-divisibility of the twist sum is equal to the $p$-divisibility of a single one.

Let's give a few explanations regarding the main results: (1) The assumption on $\lambda|_{F_\fl^\times}$ is necessary to construct Eisenstein series, which is $U(\fl)$-eigen with level $U_0(\fl)$, which seems necessary for general totally real field (cf.~Remark~\ref{ul}). (2) For the case $\rank_{\BZ_\ell}\Gamma=1$, if strong density holds (cf.~Remark~\ref{sden}), one could replace the condition \enquote{Zariski dense} (equivalently, \enquote{infinitely many}) by \enquote{except finitely many}, and the main results still hold.

\subsection{Related works}
We recall some related results in addition to \cite{Hi1}, \cite{Hi2}, \cite{Hs1}, and \cite{Oh}.

There are some previous results in the case $F=\BQ$ using a different approach based on the Mumford theory of theta functions. It was proved by Finis \cite{Fi} under the condition $\lambda$ is self-dual with root number $+1$, $\ell$ split, and $p$ odd. His method could also handle the case where $p$ is non-ordinary.

Also see recent joint work of the author (cf.~\cite{BHKO}) in the self-dual case with $F=\BQ$, using a different method from Hida and Finis, which also allows $\ell$ to be inert, and the result is on \enquote{except for finitely many}.

There is also related work of Burungale \cite{As}, based on the method of this paper, one may extend the $p$-stability result to the $p$-adic family containing $\CX$.
\subsection{Applications and further questions}

The $p$-stability results have several potential applications.
\begin{itemize}
\item One may use our result to remove certain ramification assumptions in the Eisenstein congruence approach to the Main conjecture for CM fields (see Remark after Theorem 2 in \cite{Hs2}).
\item Together with the rank $0$ result on $p$-adic BSD in the CM case \cite{AF}, one may get constancy of $p$-indivisibility of $\Sha$ of CM abelian varieties in the $\ell$-adic anticyclotomic twist family.
\item The strong $p$-stability result on \enquote{except for finitely many} is still open in general, which has a consequence on the stability of $p$-part of ideal class groups in $\ell$-adic anticyclotomic extensions (cf.~\cite{Oh}). 
\item As pointed out by the referee, the modulo $p$ theta dichotomy may have applications to the Iwasawa $\mu$-invariant of Katz $p$-adic L-function in the residually self-dual but not self-dual case. The modulo $p$ theta dichotomy may also have some potential applications in other situations.
\end{itemize} 
\subsection{Organization}
In \S\ref{S3}, we recall a good integral model of Hilbert modular schemes at $p$, Katz's theory of $p$-adic modular forms, and Hida's density of CM points. In \S\ref{S4}, we introduce Waldspurger formulae for Eisenstein and also consider its explicit and family version to construct the measure that interprets the $\ell$-adic twist family of Hecke L-values. Further $p$-adic properties of the Eisenstein series are studied in \S\ref{S5}. In the last section, we complete the proof of the main result.

\subsection{Notations}
Let $F$ be a totally real field and $\CO_F$ its ring of integers. Let $\BA_F$ be its ring of \adeles and $\BA_{F,f}$ its ring of finite \adeles. 
For $v$ a place of $F$, denoted by $F_v$ be the completion of $F$ at $v$, $\CO_{F,v}$ be the ring of integers of $F_v$, $\fp_v$ the maximal ideal of $\CO_{F,v}$, $\varpi_v$ be a uniformizer of $\CO_{F,v}$, $|\cdot|_v$ the normalized absolutely value on $F_v^\times$. For $K/F$ a quadratic field extension, we have similar notations for $K$. Let $\Nm:K^\times \ra F$ be the norm from $K$ to $F$.

Denoted by $\Sigma_F$ be the set of infinite places of $F$. We may identify $\Sigma_F$ with the restriction of a CM type $\Sigma$ of $K$ to $F$. Let $\psi$ be the standard nontrivial character over $F\bs \BA_F$. Let $\delta_v\in \BZ_{\leq 0}$ such that $\varpi_v^{-\delta_v}$ generates the different ideal of $F_v$. 

Fix a Haar measure $d t_v$ on $F_v$ such that for all finite place $v$, $\vol(\CO_{F,v},dt_v)=1$ and the measure $d t:=\prod_{v}d t_v$ on $\BA_F$ induces the Tamagawa measure. Fix a Haar measure $d^\times t_v$ on $K_v^\times/F_v^\times$ such that for all finite place $v$, $\vol(\CO_{K,v}^\times/\CO_{F,v}^\times,d^\times t_v)=1$ and the measure $d^\times t:=\prod_{v}d^\times t_v$ on $\BA_{K}^\times/\BA_{F}^\times$ induces the Tamagawa measure.

Let $G=\GL_{2}$ over $F$.

\subsection*{Acknowledgement}
The author is grateful to Prof. Ye Tian for guidance and continuous encouragement. The author would like to thank Prof. Burungale Ashay, Prof. Haruzo Hida, Prof. Li Cai, Prof. Ming-Lun Hsieh, and Xiaojun Yan for many helpful communications. We are grateful to the referee for helpful comments and constructive suggestions. W. He is partially supported by Innovation Program for Quantum Science and Technology Grant no. 2021ZD0302902.

\section{Geometry of Hilbert Modular Variety}\label{S3}
In this section, we introduce the $p$-adic theory of modular forms and CM points on the integral model of Hilbert modular schemes. In \S\ref{3.1}, we recall Hilbert modular forms over $\BC$, and in \S\ref{3.2}, we introduce integral model of the Hilbert modular scheme and the $q$-expansion principle of Katz $p$-adic Hilbert modular forms. In \S\ref{3.3}, we introduce Hida's result on density of CM points modulo $p$. 
\subsection{Complex theory of modular forms}\label{3.1}
Denoted by $G(F)_+$ be the subgroup of $G(F)$ consisting of elements with totally positive determinants. Denoted by $\Sigma_F$ be the set of infinite places of $F$. For each $\sigma\in\Sigma_F$, $G(F)_+$ acts on the upper half plane $\CH^\sigma$ by embedding $\sigma:G(F)\ra G(F_{\sigma})$ and fractional linear transformation. Let $U$ be an open compact subgroup of $G(\BA_{F,f})$, we have a complex analytic space \[X_U(\BC):=G(F)_+\bs\CH^{\Sigma_F}\times G(\BA_{F,f})/U.\]
For $g_\infty=\prod_{\sigma\in\Sigma_F}\begin{pmatrix}
a_\sigma & b_\sigma \\
c_\sigma& d_\sigma
\end{pmatrix}\in G(F_\infty)$, $\rho=\sum_{\sigma\in\Sigma_F}n_\sigma\sigma\in \BZ[\Sigma_F]$ and $\tau\in \CH^{\Sigma_F}=\prod_{\sigma\in \Sigma_F}\CH^\sigma$, let $\displaystyle J(g_\infty,\tau)^\rho=\prod_{\sigma\in\Sigma_F}(c_\sigma\tau_\sigma+d_\sigma)^{n_\sigma}$.
Recall that
\begin{enumerate}
\item [(i)] A smooth Hilbert modular form $\mathbf{f}$ of weight $\rho$, level $U$ is a function on $\CH^{\Sigma_F}\times G(\BA_{F,f})$ such that for each $g_f$, $\mathbf{f}(\tau,g_f)$ is a smooth function on $\CH^{\Sigma_F}$ and \[\mathbf{f}(\alpha(\tau,g_f)u)=J(\alpha,\tau)^{\rho}\mathbf{f}(\tau,g_f)\] for all $\alpha\in G(F)_+$ and $u\in U$. The space of all smooth forms of weight $\rho$, level $U$ is denoted by $M_\rho(U)(C^\infty)$.
\item [(ii)] A holomorphic Hilbert modular form $\mathbf{f}$ is a smooth Hilbert modular form that $\mathbf{f}(\tau,g_f)$ is holomorphic on $\CH^{\Sigma_F}$ for each $g_f$ and holomorphic at all the cusps. Denoted by $M_\rho(U,\BC)$ be the space of all holomorphic Hilbert modular forms of weight $\rho$, level $U$. When $\rho=k\Sigma_F$, we simply write the space as $M_k(U,\BC)$.
\end{enumerate}
The set of cusps is given by $G(F)_+\bs \BP^1(F)\times G(\BA_{F,f})/U$. Note that each cusp is represented by $(\infty, g_f)$. A modular form $\mathbf{f}\in M_\rho(U,\BC)$ has Fourier expansion at cusp $(\infty, g_f)$:
\[\begin{aligned}\mathbf{f}(\tau,g_f)
&=\sum_{\beta\in F_+\sqcup\{0\}}\alpha_\beta(\mathbf{f},g_f)e^{2\pi i\tr(\beta\tau)},\end{aligned}\] here $\tr(\tau\beta)=\sum_{\sigma\in \Sigma_F}\sigma(\beta)\tau_\sigma$ for $\tau=(\tau_{\sigma})_{\sigma\in \Sigma_F}$, $F_+\subset F$ is the subset of the totally positive elements. We call $\alpha_\beta(\mathbf{f},g_f)$ the $\beta$-th Fourier coefficient of $\mathbf{f}$ at the cusp $(\infty,g_f)$.

For each $g\in G(\BA_{F,f})$, the natural map \[X_{gUg^{-1}}(\BC)\ra X_U(\BC)\quad [\tau,g']\mapsto [\tau,g'g]\]induces an isomorphism \[|[g]:M_\rho(U)(C^\infty)\simeq M_{\rho}(gUg^{-1})(C^\infty),\quad f\mapsto f|[g]:(\tau,g')\mapsto f(\tau,g'g)\] and $|[g]:M_\rho(U,\BC)\simeq M_{\rho}(gUg^{-1},\BC)$.

For each $\sigma\in\Sigma_F$, there exists a weight-raising differential operator called the Maass-Shimura operator (cf.~\cite[\S~2.1]{Ka}) $\delta_{\sigma,n_\sigma}=\frac{1}{2\pi i}(\frac{\partial}{\partial \tau_\sigma}+\frac{n_\sigma}{2iy_\sigma})$, which maps $M_\rho(U)(C^\infty)$ to $M_{\rho+2\sigma}(U)(C^\infty)$ for $\rho=\sum_{\sigma\in\Sigma_F}n_\sigma\sigma$. For different $\sigma,\sigma'\in \Sigma_F$, $\delta_{\sigma,n_{\sigma}}$ and $\delta_{\sigma',n_{\sigma'}}$ are commute with each other. The Maass-Shimura operators commute with Hecke action $|[g]$.
\subsection{Integral theory of modular forms}\label{3.2}
\subsubsection{Geometric modular forms and $p$-adic modular forms.}

Let $V=F^2$ be the space endowed with the standard $F$-bilinear alternating pairing $\pair{\ ,\ }$ given by $\pair{(1,0),(0,1)}=1$. Then $\pair{\ ,\ }$ induces a non-degenerate $\BQ$-valued alternating pairing $\pair{\ ,\ }_\BQ$ defined as $\tr_{F/\BQ}\circ\pair{\ ,\ }$. Denote $\CO_{F}^*$ as the inverse of different, i.e. $\CO_F^*=\{x\in F|\ \tr_{F/\BQ}(x\CO_{F})\subset \BZ\}$. Let $\CL=\CO_F\oplus \CO_F^*$ be the self-dual lattice with respect to $\pair{\ ,\ }_\BQ$. View $V$ as row vectors, and $G$ acts left on $V$ by $g(a,b):=(a,b)\cdot g'$, where $'$ is the involution given by $\begin{pmatrix}a & b \\c & d\end{pmatrix}':=\begin{pmatrix}d & -b \\-c & a\end{pmatrix}$.

 Let $\square$ be the set consisting of a rational prime $p$ or $\emptyset$. Let $U=U_{\square}\cdot U^{(\square)}\subset G(\BA_{F,\square})\cdot G(\BA_{F,f}^{(\square)}$) be an open compact subgroup such that $U_{\square}$ is maximal in the sense that it is the stabilizer of $\CL_p=\CL\otimes_{\BZ}\BZ_p$. (If $\square=\emptyset$, then this is an empty condition.)
For a connected $\BZ_{(\square)}$-scheme $S$, let $\ov{s}$ be a geometric point of $S$. (When $\square=\emptyset$, the localization means localizing at the zero ideal.) A $S$-quadruple of level $U$ consists of the following data:
\begin{enumerate}
\item [(i)] $A$ is an abelian scheme of relative dimension $g$ over $S$.
\item [(ii)] $\iota:\CO_F\ra \End_{S}(A)\otimes_{\BZ}\BZ_{(\square)}$.
\item [(iii)] $\ov{\lambda}$: $\CO_{F,(\square),+}$-orbit of a prime-to-$\square$-polarization $\lambda$ (in $\Hom_{S}(A,A^t)\otimes_{\BZ}\BZ_{(\square)}$); Here $a\in\CO_{F,(\square),+}$ acts on $\lambda$ by $\lambda a:=\lambda\circ \iota(a)$.
\item [(iv)] $\ov{\eta}^{(\square)}$: $U^{(\square)}$ orbit of isomorphism of $\BA_{F,f}^{(\square)}$-modules $\eta^{(\square)}:V\otimes_{\BZ}\BA_{F,f}^{\square}\simeq H_1(A_{\ov{s}}, \BA_{F,f}^{(\square)})$ that defined over $S$, i.e. $\ov{\eta}^{(\square)}$ is $\pi_1(S,\ov{s})$ invariant. Here $u\in U^{(\square)}$ acts on $\eta^{(\square)}$ by $\eta^{(\square)} u:=\eta^{\square}\circ u$.
\end{enumerate}
Here further need that (1) $\Lie_S(A)$ is locally free rank one as $\CO_{S}\otimes _{\BZ}\CO_{F}$ module in Zariski topology; (2) $\iota$ is invariant under the Rosati involution induced by $\lambda$; (3) Fix an isomorphism $\wh{\BZ}\simeq \wh{\BZ}(1)$ by mapping $1/N$ to $e^{2\pi i/N}$ then take inverse limit. The Weil pairing induced by $\lambda$ is equal to $\tr_{F/\BQ}\circ e_{\lambda}$, where $e_{\lambda}$ is a $\BA_{F,f}^{(\square)}$ linear alternating form $e_{\lambda}: H_1(A_{\ov{s}}, \BA_{F,f}^{(\square)})^2\ra \BA_{F,f}^{(\square)}$. Assume pullback of $e_{\lambda}$ is a $\BA_{F,f}^{(\square)}$ multiple of $\pair{\ ,\ }$.

Two $S$-quadruples $(A,\iota,\ov{\lambda},\ov{\eta}^{(\square)})$ and $(A',\iota',\ov{\lambda'},\ov{\eta'}^{(\square)})$ are said to be equivalent if there exists a $S$-quadruple homomorphism induced by a prime-to-$\square$-isogeny $\phi$ in $ Hom_{S}(A,A')\otimes_{\BZ}\BZ_{(\square)}$.

Let's recall properties of the Kottwitz model (cf.~\cite[\S~4.2.1]{Hi1}):
\begin{prop}\label{rat}
\item If $U$ is sufficiently small so that $\det(U)\cap \CO_{F,+}^\times\subset (U\cap \CO_{F})^2$, the moduli problem on $\BZ_{(\square)}$-schemes which parameterize equivalent classes $(A,\iota,\ov{\lambda},\ov{\eta}^{(\square)})$ is represented by a normal quasi-projective $\BZ_{(\square)}$-scheme $X^{(\square)}_U$. The geometric irreducible component of $X^{(\square)}_U$ is one-to-one correspondence with $F_+^\times\bs\BA_{F,f}^{\times}/\det U$. The scheme $X^{(\square)}_U$ is smooth over $\BZ_{(\square)}$ if $\square$ is prime to $D_F$.
\end{prop}\label{unif}When $\square=\emptyset$, we will denote $X^{(\square)}_U$ by $X_U$.
We have the complex uniformization \[X_U(\BC)= G(F)_+\bs\CH^{\Sigma_F}\times G(\BA_{F,f})/U\] given by the following: Given $[\tau,g]\in G(F)_+\bs\CH^{\Sigma_F}\times G(\BA_{F,f})/U$, we have period map $p_\tau:V\otimes_{\BQ}\BC\simeq \BC^{\Sigma_F}, (a,b)\mapsto (a,b)\cdot\begin{pmatrix}\tau \\1\end{pmatrix}$.
Denote $\CL_{\tau,g}=p_\tau(g\wh{\CL}\cap V)$.
\begin{enumerate}
\item [(i)] $A_{\tau,g}=\BC^{\Sigma_F}/\CL_{\tau,g}$.
\item [(ii)] The polarization on $A_{\tau,g}$ is induced by the Riemann form $\pair{\ ,\ }\circ p_{\tau}^{-1}$.
\item [(iii)] $\eta:V\otimes_{\BZ}\BA_{F,f}\simeq V_f(A_{\tau,g}),\quad (a,b)\mapsto (a,b)g'\begin{pmatrix} \tau \\1\end{pmatrix}$.
\end{enumerate}

Let $A/S$ be an abelian scheme of relative dimension $g$ with real multiplication by $\CO_F$. Let $\mathfrak{N}\subset \CO_F$ be an integer ideal with norm given by some power of $p$. A $\mu_{\mathfrak{N}}$-level structure of $A$ over $S$ is close immersion $j_{\mathfrak{N}}:\CO_F^*\otimes_{\BZ}\BG_m[\mathfrak{N}]\ra A[\mathfrak{N}]$ of $\CO_F$-group schemes over $S$.

We have the following $p$-integral model of Hilbert modular scheme that has $\mu$-type level at $p$. 
\begin{prop}\label{mode} When $U^{(p)}$ is sufficiently small,
the moduli problem on $\BZ_{(p)}$-schemes which parameterize prime-to-$p$-isogeny equivalent classes $(A,\iota,\ov{\lambda},\ov{\eta}^{(p)},j_{\mathfrak{N}})$ is represented by a normal $\BZ_{(p)}$ scheme $\CX_{U^{(p)},\mathfrak{N}}$. The geometric irreducible components of $\CX_{U^{(p)},\mathfrak{N}}$ are one-to-one correspondence with \[F_+^\times\bs\BA_{F,f}^\times/\det U^{(p)}\cdot\CO_{F,p}^\times.\] If $p$ is unramified in $F$, $\CX_{U^{(p)},\mathfrak{N}}$ is smooth over $\BZ_{(p)}$. One has complex uniformization: $\CX_{U^{(p)},\mathfrak{N}}(\BC)= G(F)^+\bs\CH^{\Sigma_F}\times G(\BA_{F,f})/U^{(p)}\cdot U_1(\mathfrak{N})$, here for $\mathfrak{N}=\prod_{\fp|p}\fp^{n_\fp}$, $U_1(\mathfrak{N})=\prod_{\fp|p}U_1(\fp^{n_\fp})$ with $U_{1}(\fp^{n_\fp})=\left\{\begin{pmatrix}
a&b\\c&d
\end{pmatrix}\in \GL_2(F_\fp)\ \Big|\ a\in 1+\fp^{n_\fp}, b\in \fp^{\delta_\fp}, c\in \fp^{n_\fp-\delta_\fp}, d\in \CO_{F,\fp}\right\}$.
\end{prop}

The smoothness of the model follows from the properties of the Kottwitz model and \cite[Lem.~6.17]{Ra}. Also see \cite[Thm.~4.5]{DR} for the discussion on geometric irreducible components.

The complex uniformization is similar as before, except the level structure corresponding to $[\tau,g]$ at $p$ is given by\[i_ {\mathfrak{N}}:\mu_{p^\infty}\otimes_{\BZ}\CO_F^*[\mathfrak{N}]\xrightarrow{}\mu_{p^n}\otimes_{\BZ}\CO_F^*\xrightarrow{e^{2\pi ix/p^n}\otimes y\mapsto(0,x\otimes y)}\BZ_p/p^n\otimes(\CO_F\oplus\CO_{F}^*)\xrightarrow{(a,b)\mapsto (a,b)g'\begin{pmatrix}\tau \\ 1\end{pmatrix}}A_{\tau,g}[p^n]\]for $n\gg 0$.
From now on, fix $U^{(p)}$, we will also simply denote $\CX_{U^{(p)},\mathfrak{N}}$ by $\CX$. Let $F_0$ be the Galois closure of $F$ and $B_0=\CO_{F_0,(p)}$. We may base change $\CX$ to $B_0$ and still denote it by $\CX$. Denote $f:\CA\ra \CX$ the universal abelian scheme and $\underline{\omega}=f_*\Omega^1_{\CA/\CX}$. We have a natural map $\displaystyle\underline{\omega}\ra\bigoplus_{\sigma\in \Sigma_F}\underline{\omega}_{\sigma}$, where $\underline{\omega}_\sigma$ is the maximal sub-sheaf of $\underline{\omega}$ on which $\CO_{F}$ acts by embedding $\sigma$. The map is an isomorphism if $D_F$ is invertible in $B_0$. We may identify $\Hom_{B_0}(\Res_{\CO_F/\BZ}\BG_m,\BG_m)$ with $\BZ[\Sigma_F]$ by mapping a character $\rho(x):=\prod_{\sigma\in\Sigma_F}\sigma(x)^{n_\sigma}$ to $\sum_{\sigma\in\Sigma_F}n_{\sigma}\sigma$. For each $\rho:\Res_{\CO_F/\BZ}\BG_m\ra \BG_m$ over $B_0$, define $\underline{\omega}_\rho=\otimes\underline{\omega}_\sigma^{\otimes n_\sigma}$. When $F\neq \BQ$, the geometric modular form of weight $\rho$ over a $B_0$ algebra $B$ is the $B$-module $H^0(\CX_{B},\underline{\omega}_\rho)$. If $F=\BQ$, we need to extend everything to the compactified modular curve and define the geometric modular form to be the global section of the related sheaf on the compactified modular curve. The modular form has the following equivalent definition:
\begin{defn}(geometric modular form) A test object over a $B_0$-algebra $B$ is a $B$-point $(A,\iota,\ov{\lambda},\ov{\eta}^{(p)},j_{\mathfrak{N}})$ of $\CX$ together with a $\CO_F\otimes_{\BZ}B$-basis $\omega$ of $H^0(A,\Omega_{A/B}^1)$. A geometric modular form $f$ of weight $\rho\in \Hom_{B_0}(\Res_{\CO_F/\BZ}\BG_m,\BG_m)$ on $\CX$ with coefficients in a $B_0$ algebra $B$ is a rule that assigns to each test object $(A,\iota,\ov{\lambda},\ov{\eta}^{(p)},j_{\mathfrak{N}},\omega)$ over a $B$-algebra $B'$ a value $f(A,\iota,\ov{\lambda},\ov{\eta}^{(p)},j_{\mathfrak{N}},\omega)\in B'$ such that
\begin{enumerate}
\item [(i)] The value $f(A,\iota,\ov{\lambda},\ov{\eta}^{(p)},j_{\mathfrak{N}},\omega)$ only depends on the equivalent class of the test object;
\item [(ii)] For a $B$-homomorphism $\phi:B'\ra B''$ and $(A,\iota,\ov{\lambda},\ov{\eta}^{(p)},j_{\mathfrak{N}},\omega)$ over $B'$, we have \[\phi(f(A,\iota,\ov{\lambda},\ov{\eta}^{(p)},j_{\mathfrak{N}},\omega))=f((A,\iota,\ov{\lambda},\ov{\eta}^{(p)},j_{\mathfrak{N}},\omega)_{B''});\]
\item [(iii)] $f(A,\iota,\ov{\lambda},\ov{\eta}^{(p)},j_{\mathfrak{N}},a\omega)=\rho(a)^{-1}f(A,\iota,\ov{\lambda},\ov{\eta}^{(p)},j_{\mathfrak{N}},\omega)$ for all $a\in \Res_{\CO_F/\BZ}\BG_m(B)$;
\item[(iv)] $f$ is holomorphic at all the cusps.
\end{enumerate}
 Denote $\CM_\rho(U^{(p)},\mathfrak{N},B)$ the $B$ module of all such geometric modular forms over $B$.
\end{defn}
We now recall the algebraic definition of $q$-expansion. The set of cusps of $\CX$ is given by $G(F)_+\bs\BP^1(F)\times G(\BA_{F,f})/U^{(p)}\cdot U_1(\mathfrak{N})$. Take $(N,p)=1$ to be a positive integer such that \[\left\{\begin{pmatrix}
a&b\\c&d
\end{pmatrix}\ \left|\ 
a, d\in 1+N\wh{\CO}_F, b\in (N\wh{\CO}_F^*)^{(p)}, c\in (N\wh{\CO}_F^{*,-1})^{(p)}\right.\right\}\] is contained in $U^{(p)}$. For each $a,b\in \BA_{F,f}^{\times,(pN)}$, let $\fa$, $\fb$ be the associated fractional ideal, denote $(\fa,\fb)$ be the cusp $\left[\infty,\begin{pmatrix}a^{-1} & 0 \\ 0 & b \end{pmatrix}\right]$. Over $\BC$, the Fourier expansion of a holomorphic modular form of level $U$ containing $\left\{\begin{pmatrix}
1&x\\&1
\end{pmatrix}\ \left|\ x\in N\wh{\CO_F^*}\right.\right\}$ at the cusp $(\fa,\fb)$ has the form \[a_0+\sum_{\beta\in (N^{-1}\fa\fb)_+}\alpha_\beta q^{\beta},\] where $q=e^{2\pi i\tr_{F/\BQ}(\cdot\tau)}$ and $\alpha_\beta\in \BC$. Let $S\subset \Hom_{\BQ-\text{alg}}(F,\BQ)$ be a set of $[F:\BQ]$ linearly independent linear forms such that $\forall \ell\in S$, $\ell$ maps $F_+$ to $\BQ_+$. Denote $\BZ[[N^{-1}\fa\fb, S]]$ be the ring of formal power series consists of \[\sum_{\substack{\beta\in N^{-1}\fa\fb\\ \ell(\beta)\geq 0,\ \forall\ell\in S}}a_\beta q^\beta.\] The index in the above summation is a finitely generated monoid. Denote $\BZ((N^{-1}\fa\fb,S))$ the ring consists of all formal series \[\sum_{\beta\in N^{-1}\fa\fb}a_\beta q^\beta\]such that the image of $\{ \beta\in N^{-1}\fa\fb|\ \alpha_\beta\neq 0\}$ under each $\ell\in S$ is bounded below. For each cusp, the Mumford's construction gives an abelian scheme $\text{Tate}_{(\fa,\fb)}$ defined over $\BZ((\fa\fb, S))$ whose associated rigid analytic space is the rigid analytic quotient of $u:\fb\ra \BG_m\otimes \fa^*$ (cf.~\cite{Hu}~\cite{Mu}). Here the map $u$ is induced by \[\{\beta\mapsto q^\beta\}\in \Hom_{\BZ}(\fa\fb,\BG_m(\BZ((\fa\fb,S))))\simeq \Hom_{\CO_F}(\fb, \BG_m\otimes_{\BZ}\fa^*(\BZ((\fa\fb, S)))).\] The Tate abelian variety has the following properties:
\begin{itemize}
\item [(i)] For any positive integer $n$, $\Tate_{(\fa,\fb)}[N]$ equals to the quotient of the inverse image of $u(\fb)$ under $N:\BG_m\otimes\fa^*\ra\BG_m\otimes\fa^*$ by $u(\fb)$. In particular, under the assumption $(Np,\fa\fb)=1$, $\Tate_{(\fa,\fb)}$ has a canonical $U^{(p)}$ level structure $\eta_{\text{can}}^{(p)}$ and a canonical $\mu_{p^n}\otimes\CO_F^*$ level structure $j_{p^n,\text{can}}$ defined over $\BZ[\zeta_N]((N^{-1}\fa\fb, S))$.
\item [(ii)] $\Tate_{(\fa,\fb)}$ has a canonical polarization $\lambda_{\text{can}}$.
\item [(iii)] There exists a canonical isomorphism $\Omega^1_{\Tate_{(\fa,\fb)}/\BZ((\fa\fb,S))}\simeq \Omega^1_{\BG_m\otimes\fa^*/\BZ((\fa\fb,S))}$. Thus there exists a canonical $\fa\otimes \BZ((\fa\fb,S))$-basis $\omega_{\text{can}}\in \Omega^1_{\Tate_{(\fa,\fb)}/\BZ((\fa\fb,S))}$ induced by $dt/t\in\Omega^1_{\BG_m\otimes\fa^*/\BZ((\fa\fb,S))}$.
\end{itemize}
Thus, the Tate abelian variety gives rise to a unique test object \[\underline{\text{Tate}}_{(\fa,\fb),\mathfrak{N}}=(\text{Tate}_{(\fa,\fb)},\iota_{\text{can}},\ov{\lambda}_{\text{can}},\ov{\eta}_{\text{can}}^{(p)},j_{\mathfrak{N},\text{can}},\omega_{\text{can}})\] over $\BZ[\zeta_N]((N^{-1}\fa\fb, S))$.

The following $q$-expansion principle is known (cf.~\cite[Thm.~6.7]{Ra}). Also see Corollary 1.6.2 of \cite{Katz}, where only need the Cohen-Macaulay condition on each geometric irreducible component.
\begin{prop}\label{q-exp}\
\begin{enumerate}
\item Over $\BC$, the natural map \[\CM_\rho(U^{(p)},\mathfrak{N},\BC)\ra M_\rho(U^{(p)}\cdot U_1 (\mathfrak{N}),\BC), f\mapsto \{[\tau,g_f]\mapsto f(A,\iota,\ov{\lambda},\ov{\eta}^{(p)},j_\mathfrak{N},2\pi idz)_{[\tau,g_f]}\}\] is an isomorphism, here $(A,\iota,\ov{\lambda},\ov{\eta}^{(p)},j_\mathfrak{N},2\pi idz)_{[\tau,g_f]}$ is the test object associated to $[\tau,g_f]$ under complex uniformization. Furthermore, $f(\underline{\Tate}_{(\fa,\fb),\mathfrak{N}})$ equals to the classical Fourier expansion of $f$ at the cusp $(\fa,\fb)$.
\item Let $B$ be a $B_0$-algebra, for each geometric irreducible component $s$ of $\CX$, choose a cusp $(\fa_s,\fb_s)\in s$ as above, then the $q$-expansion map with respect to these cusps \[\begin{aligned}\CM_\rho(U^{(p)},\mathfrak{N},B)&\ra ((B\otimes_{\BZ_{(p)}}\BZ_{(p)}[\zeta_N])[[q^{\beta}]]_{\beta\in(N^{-1}\fa_s\fb_s)_{+}\cup \{0\}})_s\\
 f&\mapsto (f_{B\otimes_{\BZ_{(p)}}\BZ_{(p)}[\zeta_N]}(\underline{\Tate}_{(\fa_s,\fb_s),\mathfrak{N}})_{s}
\end{aligned}\]is well-defined and injective.

In particular, let $B'$ be a sub-$B_0$-algebra of $B$ and $f\in \CM_\rho(U^{(p)},\mathfrak{N},B)$. If for each geometric connected irreducible component of $\CX$, exists a cusp $(\fa,\fb)$ in that component such that \[f_{B\otimes_{\BZ_{(p)}}\BZ_{(p)}[\zeta_N]}(\underline{\Tate}_{(\fa,\fb),n})\in (B'\otimes_{\BZ_{(p)}}\BZ_{(p)}[\zeta_N])[[q^\beta]]_{\beta\in(N^{-1}\fa\fb)_{+}\cup \{0\}},\] then $f$ lies in $\CM_\rho(U^{(p)},\mathfrak{N},B')$.
\end{enumerate}
\end{prop}
Fix embedding $\ov{\BQ}\ra \BC_p$. Let $B\subset \BC_p$ be a $p$-adic algebra, a trivialized test object with prime-to-$p$ level $U^{(p)}$ over $B$ is a $B$ point of the Igusa tower $I:=\varprojlim_nI_n$ where $I_n:=\CX_{U^{(p)},p^n}$, i.e. $(A,\iota,\ov{\lambda},\ov{\eta}^{(p)})$ together with a $\CO_F$-group scheme morphism $j:\mu_{p^\infty}\otimes_{\BZ}\CO_F^*\hookrightarrow A[p^\infty]$.

Let $B\subset \BC_p$ be a $p$-adic algebra, the $p$-adic modular forms of coefficient in $B$ are the formal functions on the Igusa tower $\{I_{n,B/p^m}\}_{m,n}$, i.e. \[\varprojlim_{m}\varinjlim_{n}H^0(I_{n,B/p^mB},\CO_{m,n}),\]here $\CO_{m,n}$ is the structure sheaf on $I_{n,B/p^mB}$.

There is also an equivalent definition:
\begin{defn}[$p$-adic modular form]
 A $p$-adic modular form $F$ over $B$ is a rule that assigns to each trivialized test object $(A,\ov{\lambda},\iota,\ov{\eta}^{(p)},j)$ over a $p$-adic $B$ algebra $B'$ an element $F(A,\ov{\lambda},\iota,\ov{\eta}^{(p)},j)\in B'$ such that
\begin{itemize}
\item [(1)] The value only depends on the isomorphic class of $(A,\ov{\lambda},\iota,\ov{\eta}^{(p)},j)_{/\Spec B'}$.
\item [(2)] If we have an algebra homomorphism $i:B'\ra B''$ of $p$-adic algebras, and $(A,\ov{\lambda},\iota,\ov{\eta}^{(p)},j)$ is a trivialized test object over $B'$ then \[F((A,\ov{\lambda},\iota,\ov{\eta}^{(p)},j)_{\Spec B''})=i(F((A,\ov{\lambda},\iota,\ov{\eta}^{(p)},j))).\]
\end{itemize}
Denote $\CV(U^{(p)},B)$ the space of such $p$-adic modular forms over $B$. Denote $\underline{\Tate}_{(\fa,\fb)}$ the trivialized $U^{(p)}$ test object over $\BZ_p[\zeta_N]((N^{-1}\fa\fb,S))^{\wh{}}$ induced by $\{\underline{\Tate}_{(\fa,\fb),p^n}\}_{n}$. Here ${}^{\wh{}}$ means $p$-adic completion.
\end{defn}
There is a natural $q$-expansion preserving map \[\begin{aligned}\wh{\cdot}:\CM_\rho(U^{(p)},\mathfrak{N},B)&\ra \CV(K^{(p)},B)\\
 f&\mapsto \left(\wh{f}:(A,\ov{\lambda},\iota,\ov{\eta}^{(p)},j)_{B'}\mapsto f(A,\ov{\lambda},\iota,\ov{\eta}^{(p)},j|_{\BG_m\otimes_{\BZ}\CO^*[\mathfrak{N}]},j^* dt/t)\right),
\end{aligned}\]here $dt/t$ is the standard invariant differential on $\wh{\BG}_{m,B'}\otimes_{\BZ}\CO_F^*$.
In summary, we have
\begin{prop}For different $n$, the natural map between the schemes $I_{n,B/p^mB}(U^{(p)},G)$ induces one-to-one correspondence on geometric irreducible components.
The $p$-adic modular forms also have $q$-expansion principle. The natural map from geometric modular form to $p$-adic modular form is a $q$-expansion preserving embedding. More precisely, let $B\subset \BC_p$ be a $p$-adic $B_0$-algebra, supposing that for each geometric component $s$, choose a cusp $(\fa_s,\fb_s)$ as above in that geometric component, then the $q$-expansion map with respect to these cusps \[\begin{aligned}\CV(U^{(p)},B)&\ra ((B\otimes_{\BZ_{p}}\BZ_{p}[\zeta_N])[[q^{\beta}]]_{\beta\in(N^{-1}\fa_s\fb_s)_{+}\cup \{0\}})_s\\
 f&\mapsto (f_{B\otimes_{\BZ_{p}}\BZ_{p}[\zeta_N]}(\underline{\Tate}_{(\fa_s,\fb_s)}))_{s}
\end{aligned}\]is injective. (Here Tate object is viewed as defined over the $p$-adic completion of the base change of the original one to $B\otimes_{\BZ_p}\BZ_p[\zeta_N]$-coefficients formal power series.)
\end{prop}
For any $p$-adic complete $B_0$ algebra $B$, there is a $p$-adic Maass-Shimura operator $\theta_\sigma$ on $\CV(U^{(p)},B)$ given by \[\theta_\sigma\left(\sum_{\beta\in (N^{-1}\fa\fb)_+\cup\{0\}}a_\beta q^\beta\right)=\sum_{\beta\in (N^{-1}\fa\fb)_+\cup\{0\}} a_\beta\sigma(\beta)q^{\beta}\] on $q$-expansions at each cusp $(\fa,\fb)$ with $(\fa\fb,Np)=1$. Furthermore, for geometric modular form, the action of $p$-adic Maass-Shimura operator and classical Maass-Shimura are equal whenever take the value on $p$-ordinary CM points (cf.~\cite[Thm.~2.6.36]{Ka}).

\subsection{CM points}\label{3.3}
\subsubsection{Algebraic theory and integral theory of CM points}
Let $K/F$ be a quadratic CM extension. Let $\Sigma$ be a CM type of $K$. Let $\tau\in K\bs F$ such that $\sigma(\tau)$ has positive imaginary part for each $\sigma\in \Sigma$. Then $\tau$ can be viewed as an element in $\CH^{\Sigma_F}$ given by $(\sigma(\tau))_{\sigma\in\Sigma}$. Recall that we have a $\BC$-point on Hilbert modular variety corresponding to $[\tau,1]$ given by $(A,\ov{\lambda},\iota,\ov{\eta})$, here $A=\BC^\Sigma/\CO_F\tau+\CO_F^*$ and 
$\eta:\CL\otimes_{\BZ}\BA_{F,f}\simeq V_f(A),\quad (a,b)\mapsto (a,b)\begin{pmatrix} \tau\\ 1 \end{pmatrix}$. 
Here view $K$ in $\BC^{\Sigma}$ given by the CM type $\Sigma$, we have $A$ is abelian variety of CM type $(K,\Sigma)$. Similar for $[\tau,g]$, $g\in G(\BA_{F,f})$. Take $U$ to be sufficiently small such that the moduli problem in Proposition \ref{rat} with $\square=\emptyset$ is represented by a $\BQ$-scheme $X_U$. As a complex point of $X_U$, the CM point $[\tau,g]$ has a model over $\ov{\BQ}$. Furthermore, it has a model over its field of moduli (cf.~\cite{Sh}), in particular, $[\tau,g]$ is defined over the maximal abelian extension of the reflex field $K^*$ of $K$.

Denote $(K^*,\Sigma^*)$ be the reflex of $(K,\Sigma)$ and $N_{\Sigma^*}:\Res_{K^*/\BQ}\BG_m\ra \Res_{K/\BQ}\BG_m$ be the reflex norm. Let $\rho_\tau:K^\times\ra G(F)$ be the embedding defined by \[\rho_\tau(x)'\begin{pmatrix}\tau \\ 1\end{pmatrix}=\begin{pmatrix} x\tau \\ x\end{pmatrix}\] for all $x\in K$.
Recall the Galois action of CM points (cf.~\cite[Cor.~4.20]{Hi3}) states as follows:
\begin{thm}\label{CMpt}
Let $\sigma\in\Gal(K^{*,\text{ab}}/K^*)$ and $t^{-1}_\sigma\in K^\times\bs \BA_{K,f}^\times$ be the image of $\sigma$ under the reciprocity map. The galois action of $\Gal(K^{*,\text{ab}}/K^*)$ on CM points $[\tau,G(\BA_{F,f})]$ is given by
\[\sigma[\tau,g]=[\tau,\rho_\tau\circ N_{\Sigma^*}(t_\sigma)g].\]
\end{thm}
\label{pol}
Fix a prime $p$ and let $\Sigma$ be a $p$-ordinary CM type. This implies that every prime of $F$ above $p$ splits in $K$. For $v$ a prime of $F$ above $p$, denote $w\in \Sigma_p$ the prime of $K$ above $v$. Consider a CM point $[\tau, g_0]$ with $\CL_{\tau,g_0}=\CO_K$ and $\tau$ a purely imaginary element such that $2\tau\CD_F^{-1}$ is prime to $p$. The alternating pairing $\pair{\ ,\ }\circ \rho_{\tau}^{-1}$ is given by \[K\times K\ra F, (x,y)\mapsto \frac{xc(y)-yc(x)}{2i\Im(\tau)},\] here $c$ is the complex conjugation. The above pairing induces \[\wedge \CO_{K}^2\simeq\frac{\CD_{K/F}\CD_{F}}{2\Im (\tau)}\CO_{F}^*.\] Thus the polarization ideal of the CM point $[\tau,g_0]$ is $\fc:=\CD_{F}^{-1}2\Im(\tau) \CD_{K/F}^{-1}$. Also equals the fractional ideal associated with $\det(g_0^{-1})$. For any CM points of the form $[\tau,\rho_\tau(a)g_0]$, its polarization ideal is $\fc(\fa):=\fc N_{K/F}(\fa)^{-1}$, here $\fa$ is the fractional ideal associated to $a$.
\begin{prop}\label{pintCM}
Suppose (1) $\tau$ is a purely imaginary element with $2\tau\CD_F^{-1}$ being prime to $p$; $(2)$ for each $v|p$, $g_v\in \Stab(\CO_{F,v}\oplus\CO_{F,v}^*)$ $g_v^{-1}\rho_\tau(t)g_v=\begin{pmatrix}t_w& 0\\0&t_{\ov{w}}\end{pmatrix}$ for all $t\in K_v^\times$, then all the CM points $[\tau,\rho_\tau(\BA_{K,f}^{\times,(p)}\cdot\CO_{K,p}^\times)g]$ can be extended to $\ov{\BZ}_{(p)}$ points on $\CX_{U^{(p)},\mathfrak{N}}$.
\end{prop}
\begin{proof} It is enough to consider the case $\mathfrak{N}=p^n$.
Let $[\tau,\rho_\tau(x)g]$ be as above. Denote \[(\BC^\Sigma/L,\iota,\lambda,\eta^{(p)},j_{p^n})\in\CX_{U^{(p)},p^n}(\BC)\] the corresponding abelian variety with polarization and level structure under the complex uniformization. As a complex point, the CM point without level structure at $p$ has a model $(A,\iota',\lambda',(\eta^{(p)})')$ over its field of moduli $E$, in particular, $E$ is a finite abelian extension of $K^*$ and contained in $K'$, here $K'$ is the maximal abelian extension of $K^*$ unramified outside the prime above $\fp^*$. Here $\fp^*$ is the prime of $K^*$ above $p$ induced by $\iota_p$.

 Fix an isomorphism $\alpha:\BC^\Sigma/L\simeq A$ over $\BC$ such that $(A,\iota',\lambda',(\eta^{(p)})')$ induces $(\BC^\Sigma/L,\iota,\lambda,\eta^{(p)})$ over $\BC$.
Denote $\fp$ the prime of $E$ above $p$ induced by $\iota_p$. We now show that $A$ has good reduction at $\fp$.
By the main theorem of complex multiplication (cf.~\cite[Thm.~6.1]{Lang}), let $\sigma\in \Aut(\BC/E)$ and $s\in \BA_{K^*}^\times$ be the image of $\sigma|_{K^{*,ab}}$ under the reciprocity map. There exists a unique isomorphism of abelian varieties $\alpha':\BC^\sigma/\rho_\tau( N_{\Sigma^*}(s)^{-1})\CL\ra A$ over $\BC$ such that the following diagram commutes:
\[\xymatrix{
K/L\ar[d]_{\rho_\tau( N_{\Sigma^*}(s)^{-1})}\ar[rr]^{\alpha}&&A_{\text{tor}}\ar[d]^{\sigma}\\
K/\rho_\tau( N_{\Sigma^*}(s)^{-1})L\ar[rr]^{\alpha'}&&A^\sigma_{\text{tor}}.
}\]
If we take $\sigma\in \Aut(\BC/K')$, we can take $s\in \CO_{K^*,\fp^*}^\times$, thus $\rho_\tau( N_{\Sigma^*}(s)^{-1})L=L$. Note that $\sigma$ fixes $A$, the above commutative diagram becomes:
\[\xymatrix{
K/L\ar[d]_{\rho_\tau( N_{\Sigma^*}(s)^{-1})}\ar[rr]^{\alpha}&&A_{\text{tor}}\ar[d]^{\sigma}\\
K/L\ar[rr]^{\alpha'}&&A_{\text{tor}}.}\]
We claim $\alpha'=\alpha$, this is because the rigidity of the moduli problem implies that the automorphism of $(A,\iota',\lambda',(\eta^{(p)})')$ is trivial. Note that for $s\in\CO_{K^*,\fp^*}$, $\rho_\tau( N_{\Sigma^*}(s)^{-1})\in \CO_{K,\Sigma_p}^\times$. Thus the prime-to-$\Sigma_p$ torsion points of $A$ are defined over $K'$. By the criterion of good reduction, $A$ has good reduction at $\fp$. As the prime-to-$p$ torsion point of the integral model of $A$ at $\fp$ is \etale, $\eta^{(p)}$ can be extended to $\CO_{E,(\fp)}$. By our choice, the polarization ideal is prime to $p$, thus $\lambda$ can also be extended to $\CO_{E,(\fp)}$.
Let $\wp$ be the prime of $K'$ above $p$ induced by $\iota_p$.
By the above analysis, we also know that $\oplus_{\fp|\Sigma^c}A[\fp^\infty]$ is isomorphic to $\BQ_p/\BZ_p\otimes \CO_F$ over $K'$ and thus can be extended to $\CO_{K',(\wp)}$.
By our choice of level structure, $j_{p^n}:\mu_{p^n}\otimes\CO_{F}^*\simeq (\oplus_{\fp|\Sigma}(\BC^{\Sigma}/L)[\fp^\infty]\cap (\BC^{\Sigma}/L)[p^n])$. Thus by Weil pairing, $j_{p^n}'$ can also be extended to $\CO_{K',(\wp')}$.
\end{proof}
\subsubsection{Density of CM points} Let's first introduce good level structure of CM points. $\tau$ is a purely imaginary element with $2\tau\CD_F^{-1}$ being prime to $p$.
Let $\fl$ be a prime of $F$ prime to $p$. Choose $\varsigma_{n,f}=\varsigma^{(\fl)}\cdot\varsigma_{n,\fl}\in G(\BA_{F,f}^{(\fl)})\cdot G(F_{\fl})$ for each $n\in \BZ_{\geq 0}$ where $\varsigma^{(\fl)}$ does not depend on $n$ and 
\begin{itemize}\label{CM0}\item [(i)] The associated lattice
\[\CL_{\tau,\varsigma_{n,f}}\] equals to $\CO_{\fl^n}$. In particular, the CM points $[\tau,\rho_\tau(\BA_{K,f}^\times)\varsigma_{n,f}]$ are CM by $\CO_{\fl^n}$, where $\CO_{\fl^n}=\CO_F+\fl^n\CO_K$ is the order of $\CO_K$ with conductor $\fl^n$. Equivalently, for each place $v\neq \fl$, choose $\varsigma_v$ such that $(\CO_{F,v}\oplus\CO_{F,v}^*)\varsigma_{v}'\begin{pmatrix}\tau \\1\end{pmatrix}=\CO_{K,v}$ and for each $n$, choose $\varsigma_{n,\fl}$ such that \[(\CO_{F,\fl}\oplus\CO_{F,\fl}^*)\varsigma_{n,\fl}'\begin{pmatrix}\tau \\1\end{pmatrix}=\CO_{F,\fl}+\fl^n\CO_{K,\fl}.\]
\item [(ii)] For each $v|p$, \[\varsigma_v^{-1}\rho_\tau(t)\varsigma_v=\begin{pmatrix}
t_w& 0\\0&t_{\ov{w}}\end{pmatrix}\] for all $t\in K_v^\times$. In particular all the CM points $[\tau,\rho_\tau(\BA_{K,f}^\times)\varsigma_{n,f}]$ can be extended to $\ov{\BZ}_p\cap \ov{\BQ}$ points.
\item [(iii)] $\det(\varsigma_{n,\fl})/\det(\varsigma_{0,\fl})=x^n$, here $x$ is a uniformizer of $\CO_{F,(\fl)}$ that is prime to $p$ such that $x=\N(y)$ for $y\in \CO_{K,(\fl)}$ if $\fl$ is not inert.
\end{itemize}
For $a\in\BA_K^{(p)}\cdot\CO_{K,p}^\times$, denote $\CX_n(a)$ the CM point
\[\begin{cases}
[\tau,\rho_\tau(a)\rho_\tau^{(\fl)}(y^n)\varsigma_{n,f}], & \text{if $\fl$ is not inert},\\
[\tau,\rho_\tau(a)\rho_\tau^{(\fl)}(x^{n/2})\varsigma_{n,f}],& \text{if $\fl$ is inert and $n$ is even},\\
[\tau,\rho_\tau(a)\rho_\tau^{(\fl)}(x^{(n-1)/2})\varsigma_{n,f}],& \text{if $\fl$ is inert and $n$ is odd}.\\
 \end{cases}\]
 
Consider these CM points on the Hilbert modular scheme \[\CX^{(p)}:=\varprojlim_{U^{(p\fl)}}\CX_{ U^{(p\fl)}\cdot U_\fl, U_1(\mathfrak{N})}.\] 
Then the CM points $\CX_n(a)$ have the following properties:
\begin{lem}\
\begin{itemize}
\item [(1)] Let $A$ be the abelian scheme associated to $\CX_n(a)$, the level structure $\eta_p$ at $p$ of CM points is of the form $\eta=(\eta_{\Sigma_p},\eta_{\Sigma_p^c})$, where $\eta_{\Sigma_p^c}:(\CO_{F,p},0)\simeq T_{\Sigma_p^c}(A):=\varprojlim_{n}\oplus_{v|\Sigma_p^c}A[\fp_v^n]$ and $\eta_{\Sigma_p}:(0,\CO_{F,p}^*)\simeq T_{\Sigma_p}(A):=\varprojlim_{n}\oplus_{v|\Sigma_p}A[\fp_v^n]$, here $\fp_v$ is the prime ideal corresponding to $v$. In particular the point $\CX_n(a)$ is defined over $\ov{\BZ}_p$ and can be lifted to a $\ov{\BZ}_p$ point of the Igusa scheme over $\CX^{(p)}$ with $\mu_{p^\infty}\otimes \CO_F^*$ level structure at $p$.
\item [(2)] When $\fl$ is not inert, the CM points $\CX_n(1)$ lie in the same geometric irreducible component.
\item [(3)] When $\fl$ is inert, the CM points $\CX_{2n}(1)$ (resp. $\CX_{2n+1}(1))$ lie in the same geometric irreducible component.
\end{itemize}
\end{lem}
Now we recall Hida's result on density of CM points.

Let $n_1<n_2<\cdots$ be an increasing sequence of positive integers such that if $\fl$ is inert, then all $n_i$ are all even or all odd. Let $\Xi_{n_i}$ be those CM points $\CX_{n_i}(a)$ such that $a$ maps to $1$ under the map
\[\phi_{n_i}:\BA_{K,f}^\times/K^\times\wh{\CO}_{\fl^{n_i}}^\times\ra \BA_{K,f}^\times/K^\times\wh{\CO}_{\fl^{n_1}}^\times.\] 
Then for each $n_i,i\geq 1$, there exists a set of representatives $a_i$ of $\ker \phi_{n_i}$ in $\BA_{K,f}^{\times,(p)}\cdot\CO_{K,p}^\times$ such that $\CX_{n_i}(a_i)\in\Xi_{n_i}$ lie in the same geometric irreducible component of $\CX_{n_1}(1)$. For example, let $a_i$ be a set of representatives of $1+\varpi_\fl^{n_1}\CO_{K,\fl}/1+\varpi_\fl^{n_i}\CO_{K,\fl}$.

View the set of CM points $\Xi_n$ as points on the special fiber of the Hilbert modular scheme.

We have two group actions on $\Xi_n$:
\begin{itemize}
\item [(i)] The group \[T_1=\frac{\{x\in\CO_{K,\fl}^\times|\ x\equiv 1(\mod \fl^{n_0})\}}{\{x\in\CO_{F,\fl}^\times|\ x\equiv 1(\mod \fl^{n_0})\}}\] given by $t[\tau,\rho_\tau(a)\varsigma_{n,f}]:=[\tau,\rho_{\tau,\fl}(t)\rho_\tau(a)\varsigma_{n,f}].$
Here take $n_0$ sufficiently large such that the action is trivial on $\Xi_{n_1}$.
\item [(ii)] The subgroup \[T_2=\frac{\{x\in\CO_{K,(\fl p)}^\times|\ x\equiv 1(\mod \fl^{n_0})\}}{\{x\in\CO_{F,(\fl p)}^\times|\ x\equiv 1(\mod \fl^{n_0})\}}\]also acts as Hecke action $(\varsigma_0^{(\fl)})^{-1}\rho_\tau^{-1}(t)^{(\fl)}\varsigma_0^{(\fl)}$: \[[\tau,\rho_{\tau,\fl}(t)\rho_\tau(a)\varsigma_{n,f}]=[\tau,\rho_\tau(a)\varsigma_{\fl,n}\cdot(\varsigma_{1,f}^{(\fl)})^{-1}\rho_\tau^{-1}(t)^{(\fl)}\varsigma_{1,f}^{(\fl)}].\] This group is $p$-adic dense in $\frac{\{x\in\CO_{K,p}^\times|\ x\equiv 1(\mod \fl^{n_0})\}}{\{x\in\CO_{F, p}^\times|\ x\equiv 1(\mod \fl^{n_0})\}}$.
\end{itemize}
The Hida's density result (cf.~\cite[Prop.~2.8]{Hi1}) states as follows:
\begin{thm}\label{de}Assume $p$ is unramified in $F$ and every prime $F$ of $p$ split in $K$. Suppose that
\begin{itemize}\item [(i)] $\gamma_1,\cdots,\gamma_r\in T_1$ that pairwise independent modulo $T_2$,
\item [(ii)] $V$ be a geometric irreducible component of $
\CX^{(p)}/_{\ov{\BF}_p}$ that contains a CM point in $\Xi_{n_1}$.
\end{itemize}Consider the set \[\Xi_V:=\cup_{n_i}(\Xi_{n_i}\cap V),\] which is stable under the action of $T_i$. Assume further $\{n_i\}_i$ contains an arithmetic progression, then the image \[\Xi_V\mapsto V^n,\quad x\mapsto (\gamma_i(x))_i\] is dense in $V^n$.
\end{thm}
See also \cite{Hf} for an explanation of the assumption that $\{n_i\}_i$ contains an arithmetic progression.

\section{Construction of Eisenstein measure}
\label{S4}
In this section, we give a construction of the $p$-integral measure to interpret algebraic parts of Hecke L-values in the $\ell$-adic family with $\ell\neq p$. In \S\ref{4.1}, we introduce the period formula for L-value, and in \S\ref{4.2}--\S\ref{exf}, we give an explicit construction of the Eisenstein series and the anticyclotomic twist family version of the formula. In Section \S\ref{aocmp}, we give an explicit choice of good level structure of CM point. In \S\ref{44.3}, we make some local and global analysis on periods of the Eisenstein series. In \S\ref{4.4}, we show $p$ integrality of the periods formula and accomplish the construction of the measure based on the theory of $p$-adic modular forms in \S\ref{S3}.

\subsection{Waldspurger formula for Eisenstein series}\label{4.1}
 Denoted by $B$ the parabolic subgroup of $G$ given by upper triangular matrixes and $U_0=\prod_{v}U_{0,v}$ the maximal compact subgroup of $G(\BA_{F})$ given by $U_{0,\sigma}=\CO_2(\BR)$ for $\sigma\in\Sigma_F$ and $U_{0,v}=\Stab(\CO_{F,v}\oplus\CO_{F,v}^*)$. Here we view $F^2$ as a row vector with a right action of $G$ by multiplication.

Given $(\chi_1,\chi_2)$ a pair of Hecke characters of $\BA_F^\times$, let $I(s,\chi_1,\chi_2)$ the space of smooth $U_0$ finite complex valued functions on $G(\BA_F)$ such that \[\phi\left(\begin{pmatrix}a & b \\0 & d
\end{pmatrix}g\right)=\chi_1(a)\chi_2(d)\left|\frac{a}{d}\right|_{\BA_F^\times}^{s}\phi(g)\] for all $\begin{pmatrix} a & b \\0& d\end{pmatrix}\in B(\BA_F)$ and $g\in G(\BA_F)$. Denote $V(\chi_1,\chi_2)$ the bundle $\bigsqcup_{s\in\BC}I(s,\chi_1,\chi_2)$ over $\BC$. \begin{defn}[Meromorphic section] Let $H(g)=\left|\frac{a}{d}\right|_{\BA_F^\times}$ if $g=\begin{pmatrix}a & b \\0& d \end{pmatrix}u$ with $u\in U_{0}$.
A section $\phi$ for the bundle $V(\chi_1,\chi_2)$ is called meromorphic if there exists a finite dimensional subspace $V$ of $I(0,\chi_1,\chi_2)$ such that for any $s$, $\phi(s,\cdot)H(\cdot)^{-s}\in V$ and the map $s\mapsto \phi(s,\cdot)H(\cdot)^{-s}$ is meromorphic. Denote $\MS(\chi_1,\chi_2)$ the space of meromorphic sections of $V(\chi_1,\chi_2)$.
\end{defn}
\begin{remark}
A section $\phi$ is meromorphic if and only if it is a finite sum of meromorphic multiples of flat sections. Let $I(s,\chi_{1,v}, \chi_{2,v})$ be the local constitute of $I(s,\chi_1,\chi_2)$. Similarly, we can define the local version of meromorphic sections.
\end{remark}
For any $\phi\in \MS(\chi_1,\chi_2)$, let \[E_{\phi}(s,g)=\sum_{\gamma\in B(F)\bs G(F)}\phi(s,\gamma g),\] then it is absolutely convergent for $\Re (s)\gg 0$ and has meromorphic continuation to the whole $s$-plane. 

Let $K/F$ be a quadratic field extension and $\chi$ be a Hecke character on $\BA_K^\times$. Suppose that \[\chi_1\chi_2\cdot \chi|_{\BA_{F}^\times}=\mathbf{1}.\] Fix an embedding $\BA_{K}^\times\ra G(\BA_F)$ of the form $\varsigma^{-1}\rho\varsigma$, where $\rho$ is an embedding $\rho:K^\times\ra G$ and $\varsigma\in G(\BA_F)$.

Consider
\[\begin{aligned}\alpha_{v,s}:I(s,\chi_{1,v},\chi_{2,v})&\ra \chi_v^{-1}\\
\phi_v(s,g)&\mapsto \int_{K_v^\times/F_v^\times}\phi_v(s,\rho_{v}(t_v)\varsigma_v)\chi_v(t)d^\times t_v\end{aligned}.\]
Let $\lambda:=\chi\cdot \chi_1\circ\Nm$ and $\omega_\lambda:=\lambda|_{\BA_{F}^\times}$, \[\beta_{v,s}:=\frac{L(2s,\omega_{\lambda,v})}{L(s,\lambda_v)}\alpha_{v,s},\] then $\beta_{v,s}\in \Hom_{K_v^\times}(I(s,\chi_{1,v},\chi_{2,v}),\chi_v^{-1})$ nonzero and $\beta_{v,s}(\phi_v)=1$ if \begin{itemize}
\item [(i)] $I(s,\chi_{1,v},\chi_{2,v})$ is spherical, $\phi_v(s,u)=1$ for $u\in U_{0,v}$,
\item [(ii)] $\lambda_v$ is unramified, 
\item [(iii)] $\CO_{K,v}^\times/\CO_{F,v}^\times$ has volume $1$.
\end{itemize}
Denote $V(\chi^{-1})$ the bundle $\chi^{-1}\times\BC$ over $\BC$ and $\MS(V(\chi^{-1}))$ the space of meromorphic sections of $V(\chi^{-1})$, then $\prod_{v}\beta_{v,s}$ is an element in \[\Hom_{\BA_{K}^\times}(\MS(\chi_1,\chi_2),\MS(\chi^{-1})).\]

Define the period of the Eisenstein series associated with a meromorphic section $\phi\in \MS(\chi_1,\chi_2)$ by \[\CP(E_\phi)=\int_{K^\times\BA_{F}^\times\bs\BA_{K}^\times}\chi(t)E_{\phi}(s,\rho(t)\varsigma)d^\times{t},\] then $\CP(E_{\cdot})$ also lies in $\Hom_{\BA_K^\times}(\MS(\chi_1,\chi_2),\MS(\chi^{-1}))$.
Note that $B(F)\rho(K^\times)= G(F)$, by unfolding, we have
\begin{thm}
For any meromorphic section $\phi$ that is a pure tensor, we have
\[\CP(E_{\phi}(s,g))=\frac{L(s,\lambda)}{L(2s,\omega_{\lambda})}\prod_{v}\beta_{v,s}(\phi_v(s,g)),\]here the global L-functions are the complete L-functions.
\end{thm}
\subsection{Test vector}\label{4.2}
In this subsection, we will introduce a suitable toric-eigen test vector space and study its uniqueness and uniform property (cf.~Proposition~\ref{test}). Furthermore, it also has a good level at $p$ (see condition $(T_2)$) for $p$-adic interpretation.

Let $S$ be a finite set of places of $F$ that contain all infinite places and all finite places such that $\lambda$ is ramified. Assume $S=S_0\bigsqcup S_{\text{split}}\bigsqcup S_{\text{n-split}}\bigsqcup S_\infty\bigsqcup S'$, where $S_{\text{split}}=S_{\text{split},1}\bigsqcup S_{\text{split},2}$ (resp. $S_{\text{n-split}}$) is a subset of finite places that split (resp. nonsplit) in $K$, $S_{\infty}$ is the subset of infinite places, $S_0$ is a set of finite places such that $\chi_i$ and $\lambda|_{\BA_F^\times}$ are unramified at all $v\in S_0$, $S'$ is a subset of finite places such that $\lambda_v$ is unramified for all $v\in S'$, and take $S'$ to be large enough such that $\chi_1$ is unramified outside $S$. 

Let $(\rho,\varsigma)$ be as in \S\ref{4.1}. For each $v\in S_0$, suppose $\varsigma_v\begin{pmatrix}
\CO_{F,v}&\fp_v^{\delta_v} \\\fp_v^{1-\delta_v}&\CO_{F,v}
\end{pmatrix}\varsigma_v^{-1}\cap \rho(K_v)=\rho(\CO_{\fc_v})$, here $\CO_{\fc_v}=\CO_{F,v}+\fc_v\CO_{K,v}$ is the order of $\CO_{K,v}$ with conductor $\fc_v$ for $\fc_v\in F_v$. Let $\fc=\prod_{v\in S_0}\fc_v$ and $\CO_\fc$ be the order of $\CO_K$ with conductor $\fc$.
\begin{defn}
A meromorphic section $\phi\in \MS(\chi_1,\chi_2)$ is called $\chi^{-1}$-eigen with respect to $(\rho,\varsigma)$, $U_\fc:=\wh{\CO}_\fc^\times\prod_{v\in S_{\text{n-split}\sqcup S_\infty}}K_v^\times$ if
\[\phi(\rho(ut)\varsigma)=\chi^{-1}(u)\phi(\rho(t)\varsigma)\] for all $u\in U_\fc$.
\end{defn}
In the following, we will construct a $\chi^{-1}$-eigen section that has good arithmetic properties.

For each $v\in S_{\text{split},i}$, choose a prime $w$ of $K$ above $v$ and denote $S_i$ the set of all such $w$. For each finite place $v$ of $F$, let $e_{1,v}$ (resp. $e_{2,v}$) be the conductor of $\chi_{1,v}$ (resp. $\chi_{2,v}$).

Define the test vector space\label{sec} $V(\chi_1,\chi_2,\chi)$ (depending on $\rho$ and $\varsigma$) to be the set of pure tensor meromorphic sections $\phi=\otimes\phi_v$ such that:

\begin{itemize}
\item[$(T_1)$] If $v\notin S$ or $v\in S'$, $\phi_v(\cdot\varsigma_v)|_{K_v^\times}$ is $\chi_v^{-1}$-eigen under the action of $\CO_{K,v}^\times$, $\phi_v$ is $\chi_{1,v}\circ\det$-eigen under the action of $U_{0,v}$ and \[\beta_{v,s}(\phi_v)=\mathbf{1}_{\BC}.\]
\item[$(T_2)$] If $v\in S_{\text{split}}$, let $\varphi_v$ be the section such that
\begin{itemize}
\item [(a)] $\varphi_{v}(\cdot\varsigma_v)|_{K_v^\times}$ is $\chi_v^{-1}$-eigen under the action of $\CO_{K,v}^\times$;
\item[(b)] $\varphi_v$ is fixed by $U_1(\fp_v^{e_{1,v}+e_{2,v}}):=\left\{\begin{pmatrix}
 a&b\\ c&d
\end{pmatrix}\in U_{0,v}\ \left|\ \substack{ a-1\in {\fp_v^{e_{1,v}+e_{2,v}}}\\ c\in \fp_v^{e_{1,v}+e_{2,v}-\delta_v}}\right.\right\}$;
\item [(c)]\[\alpha_{v,s}(\varphi_v)=\frac{L(s,\lambda_v)G(\lambda_w|\cdot|_w^s)}{L(2s,\omega_{\lambda,v})},\quad w\in S_1\bigsqcup S_2,\]here\[G(\lambda_w|\cdot|_w^s)=|\varpi_v|_w^{e_w}\sum_{u\in(\CO_{F,v}/\varpi_v^{e_{w}}\CO_{F,v})^\times}\lambda_w|\cdot|_w^s(u\varpi_v^{-e_w})\psi_w(u\varpi_v^{\delta_v-e_w}).\]
\end{itemize}
If $v\in S_{\text{split},1}$, let $\phi_v=\varphi_v$, and if $v\in S_{\text{split},2}$, let $\phi_v$ be the $v$ depletion of $\varphi_v$:
\[\phi_{v}:=\left(1-|\varpi_v|_v\sum_{[u]\in\CO_{F,v}/\varpi_v\CO_{F,v}}R\left(\begin{pmatrix}1 & \varpi_v^{\delta_v-e_w-1}u \\0 & 1\end{pmatrix}\right)\right)\varphi_v.\]
\item [$(T_3)$] If $v\in S_{\text{n-split}}\sqcup S_{\infty}$,
 \begin{itemize}\item [(a)] $\phi_v(\cdot\varsigma_v)|_{K_v^\times}$ is $\chi_v^{-1}$-eigen under the action of $K_v^\times$;
\item [(b)] $\beta_{v,s}(\phi_v)=\vol(K_v^\times/F_v^\times, d^\times t_v)$ if $v$ is finite and $\alpha_{s,v}(\phi_v)=\vol(K_v^\times/F_v^\times, d^\times t_v)$ if $v$ is infinite.
 \end{itemize}
For $v\in S_0$ denote \[U(v)=\sum_{u\in\CO_{F,v}^*/\varpi_v\CO_{F,v}^*}
\begin{pmatrix}\varpi_v & u \\0& 1
\end{pmatrix},\] then the space of $U_0(\fp_v)$-invariant and $U(v)$-eigen vector in $I(s,\chi_{1,v},\chi_{2,v})$ for any fixed $s$ is two dimensional, where $U_0(\fp_v)=\left\{\begin{pmatrix}
a&b\\c&d
\end{pmatrix}\in U_{0,v}\ |\ c\in \fp_v^{1-\delta_v}\right\}$.
\item[$(T_4)$] If $v\in S_0$, $\phi_v(\rho(\cdot)\varsigma_v)|_{K_v^\times}$ is $\chi_v^{-1}$-eigen under the action of $\CO_{\fc_v}^\times$, $\phi_v$ is $U_0(\fp_v)$-invariant, $U(v)$-eigen, and supported on the big cell $B(\CO_{F,v})\begin{pmatrix} 0 & -1 \\1 & 0\end{pmatrix}U_0(\fp_v)$ such that whenever $\lambda$ is trivial on $\CO_{\fc_v}^\times$, \[\begin{aligned}
\alpha_{v,s}(\phi_v,\varepsilon):=&\int_{K_v^\times/F_v^\times}\phi_v(s,\rho(t_v)\varsigma_v)\varepsilon\chi(t_v)d^\times t_v\\
=&\vol(\CO_{\fc_v}^\times F_v^\times/F_v^\times,d^\times t_v)a_v(s)^{\ord_v(\fc_v)}\end{aligned}\] for any character $\varepsilon:K_v^\times/F_v^\times \CO_{\fc_v}^\times \ra \BC^\times$. Here $a_v(s)$ (in fact equals to $\chi_2(\varpi_v)|\varpi_v|_v^{-s}$) is the eigenvalue of $\phi_v$ under the action of $U(v)$.
\end{itemize}
We may assume further that $K$ and $G$ have the following good relative positions (may call assumption (*)): 
\begin{itemize}
\item[(i)] For each $v\notin S$ or $v\in S'\sqcup S_{\text{split}}$, $\mathbf{1}_{\CO_{F,v}\oplus\CO_{F,v}^*}(0,1)\rho(t_v)\varsigma_v=\mathbf{1}_{\CO_{K,v}}(t_v)$. For $v\in S_{\text{split}}$, view $K_v=K_{\ov{w}}\oplus K_{w}, t_v\mapsto (t_{\ov{w}},t_w)$, $w\in S_1\bigsqcup S_2$. Further assume $(0,1)\rho(t_v)\varsigma_v=(t_{\ov{w}},\varpi_v^{\delta_v}u_vt_{w})$ for $u_v$ a $v$ adic unit. In particular, $\varsigma_v^{-1}\rho(K^\times)\varsigma_v\cap U_{0,v}=\varsigma_v^{-1}\rho(\CO_{K,v}^\times)\varsigma_v$. 
\item[(ii)] For $v\in S_0$, assume \[\text{$\mathbf{1}_{\CO_{F,v}\oplus\CO_{F,v}^*}(0,1)\rho(t_v)\varsigma_v=\mathbf{1}_{\fc_v^{-1}\CO_{\fc_v}}(t_v)$},\quad\mathbf{1}_{\CO_{F,v}^\times\oplus\CO_{F,v}^*}(0,1)\rho(t_v)\varsigma_v=\mathbf{1}_{\fc_v^{-1}\CO_{\fc_v}^\times}(t_v).\] In particular, $\varsigma_v^{-1}\rho(K_v^\times)\varsigma_v\cap U_{0,v}=\varsigma_v^{-1}\rho(K_v^\times)\varsigma_v\cap U_0(\fp_v)=\varsigma_v^{-1}\rho(\CO_{\fc_v}^\times)\varsigma_v$.
\end{itemize}
\begin{remark} Will see the $p$ depletion section is used to ensure that the Fourier coefficients have no constant term in some cases (cf.~Proposition~\ref{cusps}). See also \cite{His3} for the use of $p$-depletion section at $p$ to construct $p$-adic L-function.\end{remark}
\begin{prop}\label{test} Under the assumption $(*)$ for $(\rho,\varsigma)$ and assume further that for each $v\in S_{\text{split}}$, $\chi_{1,v}|_{\CO_{F,v}^\times}=\lambda_w|_{\CO_{F,v}^\times}$ and for each $v\in S_0$, $\lambda_v|_{\CO_{\fc_v}^\times}=1$, then
\begin{itemize}
\item [(1)] The test vector space $V(\chi_1,\chi_2,\chi)$ contains exactly one element and is $\chi^{-1}$-eigen with respect to $U_\fc$.
\item [(2)] For each $s$ and each character $\varepsilon$ of $\Cl_{\fc}'$, we have $\phi$ is test vector for \[\displaystyle \prod_{\substack{v<\infty\\ v\notin S_{\text{split}, 2}}}\beta_{v,s}\prod_{\substack{v|\infty\\ v\in S_{\text{split},2}}}\alpha_{v,s}\in \Hom_{\BA_K^\times}(I(s,\chi_1,\chi_2),(\varepsilon\chi)^{-1})\] whenever $\prod_{w\in S_2}Z(s,\lambda\varepsilon_w,\wh{1_{\CO_{K,w}^\times}\lambda_w}) \neq 0$, where $Z(s,\lambda\varepsilon_w,\wh{1_{\CO_{K,w}^\times}\lambda_w})$ is the Tate's local zeta integral.
\end{itemize} 
\end{prop}

\begin{proof}\
\begin{itemize}
\item[(i)] If $v\notin S$ or $v\in S'$, then for each $s$, the second requirement of $(T_1)$ together with the decomposition $G(F_v)=B(F_v)U_{0,v}$ determine $\phi_v$ uniquely up to a constant. 

The assumption $(*)$ implies that $\varsigma_v^{-1}\rho(\CO_{K,v}^\times)\varsigma_v\subset U_{0,v}$. Together with the fact that $\lambda_v$ is unramified, it implies that the first requirement of $(T_1)$ holds. Thus there exists at most one $\phi_v$ satisfies $(T_1)$. 

Existence: Recall that for each $s$, we have a surjective map \[\begin{aligned}
\CS(F_v^2)\ra &I(s,\chi_{1,v},\chi_{2,v})\\ \Phi\mapsto &\phi_{v,\Phi}(s,g):=\chi_{1,v}\circ\det(g)|\det(g)|_v^{s}\int_{F_{v}^\times}\Phi([0,x]g)\chi_{1,v}(x)\chi_{2,v}^{-1}(x)|x|_{v}^{2s}d^\times x. \end{aligned}\]
 The unique $\phi_{v}$ satisfies $(T_1)$ is given by \[(s,g)\mapsto \frac{(\chi_{1,v}\circ\det(\varsigma_v)\cdot|\det(\varsigma_v)|_v^{s})^{-1}}{L(2s,\omega_{\lambda,v})}\phi_{v,\mathbf{1}_{\CO_{F,v}\oplus\CO_{F,v}^*}}(s,g).\]
\item[(ii)] If $v\in S_{\text{split}}$, then the second requirement of $\varphi_v$ in $(T_2)$ implies that $\varphi_v$ is a new vector (cf.~\cite{Ca}~\cite{RS}). By assumption $(*)$ and second requirement of $\varphi_v$, we have $\varphi_v(\rho(tu)\varsigma_v)=\chi_{1,v}\chi_{2,v}(u_{\ov{w}})\varphi_v(\rho(t)\varsigma_v)$ holds for any $t\in K_v^\times$ and $u\in\CO_{K,v}^\times$. The assumption $\chi_{1,v}|_{\CO_{F,v}^\times}=\lambda_w|_{\CO_{F,v}^\times}$ implies that the first requirement of $\varphi_v$. We have the following explicit construction of new vector given by \[\varphi_v=\frac{\lambda_w(u_v)(\chi_1\circ\det(\varsigma_v)\cdot|\det(\varsigma_v)|_v^{s})^{-1}}{L(2s,\omega_{\lambda,v})}\cdot \phi_{v,\Phi}\] for $\Phi=\phi_1\otimes\phi_2\in\CS(F_v^2)$ given by
\[\phi_1=\left\{\begin{aligned}
&\mathbf{1}_{\CO_{F, v}},&&\text{if $\lambda_{\ov{w}}$ is unramified},\\
&\mathbf{1}_{\CO_{F, v}^\times}\lambda_{\ov{w}}^{-1},& &\text{if $\lambda_{\ov{w}}$ is ramified},
\end{aligned}\right.\] and
\[\phi_2(x)=\left\{\begin{aligned}
&\wh{\mathbf{1}_{\CO_{F,v}^\times}\lambda_w}(x)=\mathbf{1}_{\varpi_v^{\delta_v-e_w}\CO_{F,v}^\times}(x)\wh{\mathbf{1}_{\CO_{F,v}^\times}\lambda_w}(x),&&\text{if $\lambda_w$ is ramified},\\
&\mathbf{1}_{\varpi_v^{\delta_v}\CO_{F,v}},& &\text{if $\lambda_{w}$ is unramified}.
\end{aligned}\right.
\] 
Here we use the additive character $\psi_w$ for the Fourier transformation.
The vector $\varphi_v$ also satisfies the third property.

When $v\in S_{\text{split},2}$, $\phi_v$ is given by \[\frac{\lambda_w(u_v)(\chi_1\circ\det(\varsigma_v)\cdot|\det(\varsigma_v)|_v^{s})^{-1}}{L(2s,\omega_{\lambda,v})}\cdot \phi_{v,\Phi'}\] for $\Phi'=\phi_1'\otimes\phi_2'\in\CS(F_v^2)$, where \[\phi_1'=\mathbf{1}_{\CO_{F, v}^\times}\lambda_{\ov{w}}^{-1},\quad \phi_2'=\wh{\mathbf{1}_{\CO_{F,v}^\times}\lambda_w}(x).\]
\item[(iii)] If $v\in S_{\text{n-split}}\sqcup S_{\infty}$, \[G(F_v)=B(F_v)\rho(K_v^\times),\] thus for each $s$, the space of $\chi_v^{-1}$-eigen vector is one dimensional. We have $\phi_v$ is uniquely determined by $(T_3)$.
\item[(iv)] If $v\in S_0$, we have $U_{0,v}=B(\CO_{F,v})\begin{pmatrix} 0 & -1 \\1 & 0\end{pmatrix}U_0(\fp_v)\sqcup U_0(\fp_v)$, here \[B(\CO_{F,v})=\left\{\begin{pmatrix}
a&b\\&c
\end{pmatrix}\in \GL_2(F_v)\ \left|\ a,c\in \CO_{F,v}^\times, b\in \CO_{F,v}^*\right.\right\}.\]
 Let $\phi_1$ (resp. $\phi_2$) be section such that its restriction on $U_{0,v}$ is the characteristic function of $B(\CO_{F,v})\begin{pmatrix} 0 & -1 \\1 & 0\end{pmatrix}U_0(\fp_v)$ (resp. $U_0(\fp_v)$). Then eigen value of $\phi_1$ under $U(v)$ is $\chi_2(\varpi_v)|\varpi_v|_v^{-s}$ and eigen value of $\phi_2$ is $|\varpi_v|_v^{s-1}\chi_1(\varpi_v)$. Then up to a constant, $\phi_1=\phi_{v,\Phi_1}$ with $\Phi_1=\mathbf{1}_{\CO_{F,v}^\times\oplus\CO_{F,v}^*}$. 

By assumption $(*)$, we have \[\alpha_v(\phi_1,\varepsilon)=\chi_1\circ\det(\varsigma_v)|\varsigma_v|_v^s\int_{K_v^\times}\mathbf{1}_{\fc_v^{-1}\CO_{\fc_v}^\times}(t_v)\lambda\varepsilon(t_v)|t_v|_v^sd^\times t_v.\]

 Thus the test vector $\phi_v$ is given by $\frac{\chi_1(\fc_v)|\fc_v|_v^s}{\chi_1\circ\det(\varsigma_v)|\varsigma_v|_v^s}\phi_{v,\Phi_1}$.
 \end{itemize}
 \end{proof}
Note that if for each $v\in S_0$, $\ord_v\det(\varsigma_v)=\ord_v(\fc_v)$, then the test vector space even does not depend on $\fc$.

\subsection{Explicit formulae}
In this section, we consider the uniform interpretation property of the test vector constructed in \S\ref{4.2} (cf.~Theorem~\ref{exp}~and~Corollary~\ref{wald}). \label{exf}
 Let $\Cl_{\fc}'=\BA_{K}^\times/K^\times\BA_{F}^\times U_\fc$. For $\varepsilon$ a finite order Hecke character on $\Cl_{\fc}'$, then $V(\chi_1,\chi_2,\chi)=V(\chi_1,\chi_2,\chi\varepsilon)$.

Let $\phi_{\lambda}=\otimes\phi_{\lambda,v}$, where for $v\in S_{0}\sqcup S_{\infty}$, $\phi_{\lambda,v}$ is the local test vector in definition of $V(\chi_1,\chi_2,\chi)$, and if $v\notin S_{0}\sqcup S_{\infty}$, $\phi_{\lambda,v}$ is $L(2s,\omega_{\lambda,v})$ multiple of the local test vector in in definition of $V(\chi_1,\chi_2,\chi)$.

By the construction, the toric period of the Eisenstein series associated to the test vector $V(\chi_1,\chi_2,\chi)$ has the following uniform interpretation property:
\begin{thm}\label{exp}
Let assumption be as in Proposition \ref{test}, then we have
\[\begin{aligned}
\prod_{v\in S_0}a_v(s)^{-\ord_v(\fc_v)}\sum_{t\in\Cl_\fc'}E_{\phi_{\lambda}}(s,\rho(t)\varsigma)\varepsilon\chi(t)=&L_f^{(S_0\sqcup S_{\text{split},2})}(s,\varepsilon\lambda)\\
 \cdot &\kappa_{\fc}2^{r_{\fc}}
\prod_{w\in S_{1} }G(\lambda_w|\cdot|_w^s)\prod_{w\in S_2}Z(s,\lambda\varepsilon_w,\wh{1_{\CO_{K,w}^\times}\lambda_w})\end{aligned}\] for any finite order Hecke character $\varepsilon$ on $\Cl_{\fc}'$. Here $\kappa_{\fc}=\ker \Pic(\CO_F)\ra \Pic(\CO_{\fc})$, $r_{\fc}$ is the number of independent $\BF_2$ linear relations of \[\{\varpi_v|\ v\in S_{\text{n-split}}, \text{ramified in $K$}\}\] in the relative ideal class group $\Pic_{K/F}(\CO_{\fc})$.
 \end{thm}
\begin{proof}
Recall 
\[\begin{aligned}
&\int_{\BA_{K}^\times/\BA_F^\times K^\times}E_{\phi_\lambda}(s,\rho(t)\varsigma)\chi(\rho(t)\varsigma)d^\times {t}=\prod_{v}\alpha_{v,s}(\phi_{\lambda,v}(s,\cdot)).
\end{aligned}\]
As $\phi_{\lambda}$ is $\chi^{-1}$-eigen, the integral factor through $\Cl_\fc'$, thus
\[\displaystyle\begin{aligned}
&\sum_{t\in \Cl_\fc'}E_{\phi_\lambda}(s,\rho(t)\varsigma)\chi(\rho(t)\varsigma)\\
&=\frac{\kappa_\fc}{2^{\# S_{\text{n-split, ram}}-r_\fc}\vol(\prod_{v\in S_\infty}K_v^\times/F_v^\times\wh{\CO}_\fc^\times,d^\times t)}\int_{\BA_K^\times/\BA_F^\times K^\times} E_{\phi_\lambda}(s,\rho(t)\varsigma)\chi(\rho(t)\varsigma)d^\times {t}\\
\end{aligned}\] here $S_{\text{n-split, ram}}$ is the subset of $S_{\text{n-split}}$ consists of all the primes that are ramified in $K$, \[\vol(\prod_{v\in S_\infty}K_v^\times/F_v^\times\wh{\CO}_\fc^\times,d^\times t)\] is the volume of $\wh{\CO}_\fc^\times\prod_{v\in S_\infty}K_v^\times/F_v^\times$ in $\BA_{K}^\times\bs \BA_F^\times$.

By the definition of $\phi_\lambda$,
\[\begin{aligned}
&\prod_{v}\alpha_{v,s}(\phi_{\lambda,v}(s,\cdot))\\
=&a_v(s)^{\ord_v(\fc_v)}L_f^{(S_0\sqcup S_{\text{split},2})}(s,\lambda)\prod_{w\in S_{1} }G(\lambda_w|\cdot|_w^s)\prod_{w\in S_2}Z(s,\lambda\varepsilon_w,\wh{1_{\CO_{K,w}^\times}\lambda_w})\prod_{v\in S_0}\\
 & \cdot 2^{\# S_{\text{n-split,ram}}}\vol(\wh{\CO}_\fc^\times\prod_{v\in S_\infty}K_v^\times/F_v^\times,d^\times t).
\end{aligned}\]
Combine the three equations, we get the explicit formula.
\end{proof}
\begin{remark}
When $\Nm(\fc)$ is sufficiently large, $c_\fc=1$ and $r_\fc=0$.
\end{remark}

 For arithmetic applications, there are some requirements for the choice of $\rho$ and $\varsigma$.
Assume that $\lambda$ is an algebraic Hecke character over $K$ with infinity type $k\Sigma+\kappa(1-c)$, where $k\in\BZ_{>0}$, $\kappa\in \BZ_{\geq 0}[\Sigma]$ and $\Sigma$ a CM type of $K$. 
 \begin{prop}\label{CM1}Let $\sigma\in \Sigma$ and suppose that $\chi_\sigma=\lambda_\sigma$.
For a given $\rho:K^\times\ra G(F)$, then there exists a unique $\tau_0\in\CH^{\pm,\sigma}$ fixed by $\rho(K_\sigma^\times)$ such that for $\varsigma_\sigma\in G(F_\sigma)$ with $\varsigma_\sigma(i)=\tau_0$, the standard weight $k+2\kappa_\sigma$ section $\phi_{\sigma,k+2\kappa_\sigma}$ has the property that $\phi_{\sigma,k+2\kappa_\sigma}(\cdot\varsigma_{\sigma})|_{K_\sigma^\times}$ is $\lambda_\sigma^{-1}$-eigen.
 \end{prop}
\begin{proof}
 Note that there are only two fixed points $\tau$ in $\CH^{\pm,\sigma}$ of $\rho(K_\sigma^\times)$.
For $\varsigma_\sigma$ with $\varsigma_\sigma(i)=\tau$ , one has 
\[\varsigma_\sigma^{-1}\rho(e^{i\theta})\varsigma_\sigma=\kappa_\theta, \theta\in \BR\] or 
\[\varsigma_\sigma^{-1}\rho(e^{i\theta})\varsigma_\sigma=\kappa_{-\theta}, \theta\in \BR,\] where $\kappa_\theta=\begin{pmatrix}
\cos \theta & \sin \theta\\
-\sin\theta & \cos\theta
\end{pmatrix}$.
Only one fixed point satisfies \[\varsigma_\sigma^{-1}\rho(e^{i\theta})\varsigma_\sigma=\kappa_{-\theta}, \theta\in \BR.\]
\end{proof}
 We have the following explicit Waldspurger formula for Eisenstein series in terms of nearly holomorphic modular forms.

Note that for each $\sigma\in \Sigma$, the representation $I(0,\chi_{1,\sigma},\chi_{2,\sigma})$ has a unique irreducible submodule called (limit of) discrete series of weight $k$ given by \[\bigoplus_{\substack{m\equiv k\pmod 2\\ |m|\geq k}}\BC\phi_{\sigma,m},\] where $\phi_{\sigma,m}\in I(0,\chi_{1,\sigma},\chi_{2,\sigma})$ is the standard weight $m$ section given by \[\phi_{\sigma,m}\left(\begin{pmatrix}\cos \theta&\sin\theta \\ -\sin\theta&\cos\theta
\end{pmatrix}\right)=e^{im\theta}.\] Note that $\BC\phi_{\sigma,k}$ is the subspace of the irreducible submodule consisting of vectors killed by the weight lowering operator $L_\sigma=\begin{pmatrix}1 & -i \\ i & -1\end{pmatrix}\in \fg\fl_{2,\sigma}$.
 Thus for $\mathbf{m}:=\sum_{\sigma\in\Sigma}m_\sigma \cdot \sigma$, $\phi_{\lambda,f}$ is the finite part of $\phi_{\lambda}$, \[\displaystyle E_{\lambda,\mathbf{m}}(g_\infty(i),g_f):=E_{\phi_{\lambda,f}\otimes \bigotimes_{\sigma\in\Sigma}\phi_{\sigma,m_\sigma}}(0,g)J(g_\infty, i)^{\mathbf{m}}\] is a holomorphic Hilbert modular form if and only if $\mathbf{m}=k\Sigma$. Here for $g_\infty=\left(\begin{pmatrix}a_\sigma & b_\sigma \\c_\sigma & d_\sigma\end{pmatrix}\right)_{\sigma\in \Sigma}$, $g_\infty(i):=\left(\frac{a_\sigma i+b_\sigma}{c_\sigma i+d_\sigma}\right)_{\sigma\in \Sigma}$.
 
 Denote $R_\sigma=\begin{pmatrix}1 & i \\i & -1\end{pmatrix}\in \fg\fl_{2,\sigma}$ the weight raising operator, then \[ R_\sigma^{\kappa_\sigma}\phi_{\sigma,k_\sigma}=\frac{2^{\kappa_\sigma}\Gamma(k_\sigma+\kappa_\sigma)}{\Gamma(k_\sigma)} \phi_{\sigma,k_\sigma+2\kappa_\sigma}\] for $\kappa_\sigma\in\BZ_{\geq 0}$. Translate this relation into modular form, for $\kappa\in\BZ_{\geq 0}[\Sigma]$, we have \[\delta_{k\Sigma}^\kappa E_{\lambda,k\Sigma}(g_\infty(i),g_f)=\frac{\Gamma_{\Sigma}(k\Sigma+\kappa)}{(-4\pi)^{\kappa}\Gamma_{\Sigma}(k\Sigma+2\kappa)}E_{\lambda,k\Sigma+2\kappa}(g_\infty(i),g_f)\det(g_\infty)^{-\kappa}.\]Here \[\delta_{k\Sigma}^\kappa=\prod_{\sigma\in\Sigma}\delta_{\sigma,k_\sigma+2\kappa_\sigma}\circ\delta_{\sigma,k\sigma+2\kappa_\sigma-2}\cdots\circ\delta_{\sigma,k_\sigma}\] and $\delta_{\sigma,m}$ is the Maass-Shimura operator as defined before. Will simply denote $E_{\lambda,k\Sigma}$ by $E_{\lambda,k}$.
 \begin{cor}\label{wald}Let assumption be as in Theorem \ref{exp} and Proposition \ref{CM1}, assume further $\chi_{1,\infty}$ is trivial, then
 \[\begin{aligned}&\prod_{v\in S_0}a_v^{-\ord_v(\fc_v)}\sum_{t\in\Cl_\fc'}\delta_{k\Sigma}^\kappa E_{\lambda,k}(\tau_0,\rho(t)\varsigma)\varepsilon\chi(t)\\
 &=\frac{\Gamma_{\Sigma}(k\Sigma +\kappa)}{(-4\pi)^\kappa\Gamma_{\Sigma}(k\Sigma)\Im \tau_0^\kappa}c_\fc2^{r_{\fc}}L_f^{(S_0\sqcup S_{\text{split},2})}(s,\varepsilon\lambda)\prod_{w\in S_{1} }G(\lambda_w|\cdot|_w^s)\prod_{w\in S_2}Z(s,\lambda\varepsilon_w,\wh{1_{\CO_{K,w}^\times}\lambda_w}).\end{aligned}\]
 \end{cor}
For a given $\lambda$, the explicit choice of $\chi_i$, $\chi$, and $S$ will be given in \S\ref{lambda}.

\subsection{Good level structures of CM points}\label{aocmp}
We only consider the case $S_0$ consists of a single prime $\fl$ that is prime to $p$. We now choose embedding and construct CM points that have good arithmetic properties. For example, satisfy conditions in \S\ref{CM0} and Proposition \ref{CM1}.
\label{the}

Let $\vartheta\in K$ is a purely imaginary such that $\Im (\sigma(\vartheta))>0$ for each $\sigma\in \Sigma$ and $2\vartheta \CD_{K/F}^{-1}\CD_{F}^{-1}$ prime to $\CD_ {K/F}$ and $S$, where $\CD_{K/F}\subset \CO_K$ is the relative different. Let $\rho$ be the embedding given by \[\rho:K^\times\ra G(F),\quad a+b\vartheta\mapsto \begin{pmatrix}a&b\vartheta^2 \\ b&a\end{pmatrix}.\] Then $\vartheta=\tau_0$ as in Proposition \ref{CM1}.
(Note that this embedding is conjugate to the one in the main theorem of complex multiplication in Theorem \ref{CMpt}.)

We now define a sequence of CM points $[\vartheta,\varsigma_{n,f}]$ with $\varsigma_{n,f}=\varsigma_f^{(\fl)}\cdot\varsigma_{\fl,n}$ such that the associated lattice of $\CL_{[\vartheta,\varsigma_{n,f}]}$ is $\CO_{\fl^n}$. Equivalently, for each finite place $v$ of $F$ that is prime to $\fl$ (resp. $v=\fl$), we choose isomorphism $\CO_{K,v}=\CO_{F,v}e_{1,v}\oplus\CO_{F,v}^*e_{2,v}$ (resp. $\CO_{\fl^n,v}=\CO_{F,v}e_{1,v,n}\oplus\CO_{F,v}^*e_{2,v,n}$), then $\varsigma_v$ is given by $((1,0)\varsigma_v',(0,1)\varsigma_v')=(p_{\vartheta}(e_{1,v}),p_{\vartheta}(e_{2,v})))$. (resp.$((1,0)\varsigma_{v,n}',(0,1)\varsigma_{v,n}')=(p_{\vartheta}(e_{1,v.n},e_{2,v,n}))$), where $p_\vartheta(a\vartheta+b)=(a,b)$.

For $v$ prime to $\fl\cdot 2\vartheta\CD_{F}^{-1}\CD_{K/F}^{-1}$ and $v$ split. We fix a prime $w$ of $K$ above $v$ and let $S''$ be the set of all such primes $w$. Consider the idempotents $e_w=\frac{1}{2}+\frac{\vartheta}{2\vartheta_w}$ and $e_{\ov{w}}=\frac{1}{2}-\frac{\vartheta}{2\vartheta_w}$, where $\vartheta_w$ is the image of $\vartheta$ in $K_{w}$. Let $e_{1,v}=e_{\ov{w}}$ and $e_{2,v}=-2\vartheta_{w}e_{w}$, then $\varsigma_v=\begin{pmatrix}
 -\vartheta_w & \frac{-1}{2}\\ 1 & \frac{-1}{2\vartheta_w}\end{pmatrix}$ and it has the following properties:
\[\varsigma_v^{-1}\rho(a+b\vartheta_w,a-b\vartheta_w)\varsigma_v=\begin{pmatrix} a-b\vartheta_w&0 \\ 0 & a+b\vartheta_w\end{pmatrix}\] for any $(a+\vartheta_wb,a-b\vartheta_w)\in K_{w}^\times\times K_{\ov{w}}^\times$. 
\[(0,1)\rho(a+b\vartheta)\varsigma_v= (a-b\vartheta_w,\varpi_v^{\delta_v}(a+b\vartheta_w)\cdot u_v),\]here $u_v$ is the $v$-adic unit $(-2\vartheta_{w}\varpi_v^{\delta_v})^{-1}$.

For $v$ prime to $\fl\cdot 2\vartheta\CD_{F}^{-1}\CD_{K/F}^{-1}$ and $v$ non-split, $2\vartheta\varpi_v^{\delta_v}$ generates the inverse of the local relative different, for those places exists $a_v\in 2^{-1}\CO_{F,v}$ such that $1$, $\vartheta\varpi_v^{\delta_v}+a_v$ is a relative integral basis of $\CO_{K,v}$ over $\CO_{F,v}$ and $\vartheta\varpi_v^{\delta_v}+a_v$ is a uniformizer if $v$ is ramified. Then we may choose $e_{1,v}=-1$ and $e_{2,v}=\vartheta+a_v\varpi_v^{-\delta_v}$. Then $\varsigma_v=\begin{pmatrix}a_v\varpi^{-\delta_v} & 1 \\ -1 & 0 \end{pmatrix}$ and \[\mathfrak{1}_{\CO_{F,v}\oplus\CO_{F,v}^*}((0,1)\rho(t)\varsigma_v)=\mathfrak{1}_{\CO_{K,v}}\]
\[\varsigma_v^{-1}\rho(-\vartheta\varpi_v^{\delta_v}+a_v)\varsigma_v\in \begin{pmatrix} 1 & 0 \\0 & N(\vartheta\varpi_v^{\delta_v}+a_v)\end{pmatrix}U_{0,v}.\]\label{CM} Here $N$ is the form from $K_v$ to $F_v$.

For $v|2\vartheta\CD_{F}^{-1}\CD_{K/F}^{-1}$, choose $e_{1,v},e_{2,v}\in\CO_{K,v}$ such that $\CO_{K,v}=\CO_{F,v}e_{1,v}\oplus \CO_{F,v}^*e_{2,v}$. Then we have\[\mathfrak{1}_{\CO_{F,v}\oplus\CO_{F,v}^*}((0,1)\rho(t)\varsigma_v)=\mathfrak{1}_{\CO_{K,v}}.\]
\label{el}For $v=\fl$, $2\vartheta\varpi_\fl^{\delta_\fl}$ is a generator of $\CD_{K_\fl/F_\fl}$. Here we choose uniformizer $\varpi_\fl$ to satisfy the condition in \S\ref{CM0}. Choose $a_\fl\in 2^{-1}\CO_{F,\fl}$ such that $1$, $\vartheta\varpi_\fl^{\delta_\fl}+a_\fl$ is a relative integral basis of $\CO_{K,\fl}$ over $\CO_{F,\fl}$. Let $e_{1,\fl,n}=-1$ and $e_{2,\fl,n}=-\varpi_\fl^{n}(\vartheta+a_\fl\varpi_\fl^{-\delta_\fl})$. Then $\varsigma_{\fl,n}=\begin{pmatrix}-\varpi_{\fl}^{n-\delta_\fl}a_\fl & 1 \\\varpi_{\fl}^{n} & 0\end{pmatrix}$.

For $\sigma\in\Sigma$, $\varsigma_{\sigma}=\begin{pmatrix} \Im(\sigma(\vartheta)) & 0 \\0 & 1\end{pmatrix}$.
Let $\varsigma_n=\varsigma_{n,f}\prod_{\sigma\in \Sigma}\varsigma_\sigma$.
\subsection{Analysis on test vector and Eisenstein series}\label{44.3}
\subsubsection{Preliminaries}
\label{lambda}let $\lambda$ be an algebraic Hecke character of $\BA_{K}^\times$ with infinity type $k\Sigma+\kappa(1-c)$, where $k\in\BZ_{>0}$, $\kappa\in \BZ_{\geq 0}[\Sigma]$ and $\Sigma$ a $p$-ordinary CM type of $K$.

Let's first recall the Grunwald-Wang theorem (cf.~\cite{SW}).
\begin{fact}\label{eh}
Let $L$ be a number field, and given any finite set of places $S$ and a finite order character $\chi_v$ on each $L_v^\times$ for $v\in S$, then there exists a finite order Hecke character $\chi$ over $L$ such that its local constitute at each $v\in S$ is given by $\chi_v$.
\end{fact}

In the following, we fix the choice of the datum in the test vector and Waldspurger formula for Eisenstein series.
 Let $S_0=\{\fl\}$, where $\fl$ is a prime that is coprime to $p$ and $\lambda|_{F_\fl^\times}$ is unramified. Let $S_{\text{split},1}$ consists of split primes $v$ of $F$ that divides $p$, or such that $\lambda|_{K_v^\times}$ is ramified and $v\nmid \fl$. Let $S_{\text{n-split}}$ consists of nonsplit primes $v$ of $F$ such that $\lambda|_{K_v^\times}$ is ramified and $v\nmid \fl$. 
 
Let \[S_{\text{split}, 2}=\begin{cases}\emptyset,&\text{$k>1$, or $\lambda$ self-dual, $\fl$ splits, or $k=1$, $\left\{w|\prod_{v\in S_{\text{split},1}}\left|\ \substack{\text{$\lambda_w$ is ramified and $w\notin \Sigma_p^c$}}\right.\right\}\neq \emptyset$ }, \\
v_0,&\text{otherwise},\end{cases}\] where $v_0$ is a split prime of $F$ such that $v_0$ is prime to $p\fl\cdot 2\vartheta\CD_{F}^{-1}\CD_{K/F}^{-1}$ and such that $\lambda$ is unramified at $v_0$ and let $S_2$ consists of a prime $w_0$ of $K$ above $v_0$. 

 Assume $S_{\text{split},1}$-part of the conductor of $\lambda$ is $\ff\ff'$ such that each of $\ff'$, $\ff$ is prime to its own conjugate. We may choose further that (1) $w|\ff$, if $w\in\Sigma_p$ and $\lambda_w$ is ramified; (2) If $w$ prime to $p$ and $\lambda_w$ is ramified but $\lambda_{\ov{w}}$ is unramified, then $w|\ff$. Let $S_1$ contain the primes that divide $\Sigma_p$ or $\ff$. May choose $S''$ in \S\ref{aocmp} such that $S_1\sqcup S_2\subset S''$.
 
 By the Fact \ref{eh}, there exists a finite order Hecke character $\chi_1$ on $\BA_F^\times$ such that its restriction on
 \[\prod_{v|N(\ff)}\CO_{F,v}^\times/(1+N(\ff)_v)\simeq \prod_{w|\ff}\CO_{K,w}^\times/(1+\ff_w)\]
is given by $\lambda|_{\prod_{w|\ff}\CO_{K,w}^\times/(1+\ff_w)}$, and the prime to $S_{\text{split}, 1}$-part of the conductor of $\chi_1$ is prime to $p\fl\CD_{K/F}$, $S_{\text{n-split}}$, and $\chi_{1,\infty}=1$. Here $N$ is the norm of ideals from $K$ to $F$. Choose $\chi=\lambda\cdot\chi_1^{-1}\circ \Nm$ and $\chi_2=\chi^{-1}|_{\BA_F^\times}\cdot\chi_1^{-1}=\lambda^{-1}|_{\BA_{F}^\times}\cdot\chi_1$.
\begin{remark}
Under the above auxiliary choice of $\chi_1$ at primes $\fq$ dividing $N(\ff)$, we could choose new vector at $\fq$ (cf.~Condition~$(T_2)$~in~\S\ref{4.2}), which is key to get $p$-integrality of the associated Eisenstein series via the integral model introduced in Proposition \ref{mode}. To extend it to be a global character, we need to use Grunwald-Wang (cf.~Fact~\ref{eh}) whenever the class number of $F$ is not $1$. 
\end{remark}

\subsubsection{Explicit local test vector}\label{expli}

We give an explicit construction of $\phi_\lambda$ in Theorem \ref{exp} under the choice of $(\rho,\varsigma_{n})$ in \S\ref{CM}. Also, see the proof of Proposition \ref{test}. 
\begin{enumerate}
\item If $v\notin S\bs S'$, or $v\in S_{\text{split}}$, $\phi_{\lambda,v}=c_v\cdot\phi_{v,\Phi}$ for $\Phi=\phi_1\otimes\phi_2\in\CS(F_v^2)$ given by
\[\phi_1=\left\{\begin{aligned}
&\mathbf{1}_{\CO_{F, v}},&&\text{if $v\notin S\bs S'$, or $v=w\ov{w}$ with $w\in S_1$ and $\lambda_{\ov{w}}$ is unramified },\\
&\mathbf{1}_{\CO_{F, v}^\times}\lambda_{\ov{w}}^{-1},& &\text{if $v=w\ov{w}$: either $w\in S_1$ and $\lambda_{\ov{w}}$ is ramified, or $w\in S_2$ },
\end{aligned}\right.\] 
\[\phi_2(x)=\left\{\begin{aligned}
&\wh{\mathbf{1}_{\CO_{F,v}^\times}\lambda_w}(x),&&\text{if $v=w\ov{w}$: either $w\in S_1$ and $\lambda_{{w}}$ is ramified, or $w\in S_2$, }\\
&\mathbf{1}_{\varpi_v^{\delta_v}\CO_{F,v}},& &\text{if $v\notin S\bs S'$, or $v=w\ov{w}$ with $w\in S_1$ and $\lambda_{w}$ is unramified. }
\end{aligned}\right.
\] Here \[c_v=\lambda_w(u_v)(\chi_1\circ\det(\varsigma_v)\cdot|\det(\varsigma_v)|_v^{s})^{-1}\] if $v\in S_{\text{split}}$ and \[c_v=(\chi_1\circ\det(\varsigma_v)\cdot|\det(\varsigma_v)|_v^{s})^{-1}\] otherwise. Here recall $u_v\in \CO_{F,v}^\times$ is some unit.

Then for $v\in S_{\text{split},1}$, $\phi_{\lambda,v}$ is fixed by \[U_1(\fp_v^{e_w+e_{\ov{w}}})=\left\{\begin{pmatrix}
a&b\\ c&d
\end{pmatrix}\in U_{0,v}\ \left|\ \substack{ a-1\in {\fp_v^{e_{w}+e_{\ov{w}}}},\\ c\in \fp_v^{e_{w}+e_{\ov{w}}-\delta_v}}\right.\right\}.\]Here $e_w$ (resp. $e_{\ov{w}})$ is the conductor of $\lambda_w$ (resp. $\lambda_{\ov{w}})$.
For $v\notin S$, $\phi_{\lambda,v}$ is fixed by $U_{0,v}$.
\item If $v\in S_{\text{n-split}}\sqcup S_{\infty}$, write $K_v^\times=B(F_v)\cdot \rho(K_v^\times)$, \[\phi_{\lambda,v}(s,\rho(t)\varsigma_v)=\chi_v(t)^{-1}.\] Then $\phi_\sigma$ is the standard weight $k_\sigma+2\kappa_\sigma$ section for $\sigma\in \Sigma$.
\item If $v=\fl\in S_0$, $\phi_{\lambda,\fl}$ is the unique $U_0(\fl)$ invariant function that is supported on the big cell $B(F_\fl)\mathbf{w}N(\CO_{F,\fl}^*)$ and $\phi_{\lambda,\fl}(w)=1$, where 
\[\mathbf{w}=\begin{pmatrix}
 0&-1 \\1&0\end{pmatrix}.\]
\end{enumerate}
\label{level}Let $U_\lambda\subset G(\BA_{F,f})$ be an open compact group that fixes $\phi_{\lambda}$ such that at $v\in S_{\text{split},1}$ is given by \[U_1(\fp_v^{e_w+e_{\ov{w}}}),\] at the place $v\notin S$ is given by $U_{0,v}$ and at the place $\fl$ is given by $U_0(\fl)$.

\subsubsection{Fourier coefficient of Eisenstein series}\label{add}

Let $\phi_{\lambda,k}=\phi_{\lambda,f}\otimes\bigotimes_{\sigma\in \Sigma}\phi_{\sigma,k}$, we have the following Fourier expansion:
\[\begin{aligned}
E_{\phi_{\lambda,k}}(s,g)=\phi_{\lambda,k}(g)+M_{\mathbf{w}}(\phi_{\lambda,k})(g)+\sum_{0\neq\beta\in F}W_{\beta}(\phi_{\lambda,k},g)
\end{aligned},\] here for $\beta\neq 0$, \[W_\beta(\phi_{\lambda,k},g)=\int_{\BA_F}\phi_{k,\lambda}\left(\begin{pmatrix}0 & -1 \\1 & 0\end{pmatrix}\begin{pmatrix}1 & x \\0 & 1 \end{pmatrix}g\right)\psi(-\beta x)dx\]and $M_\mathbf{w}=\prod_{v}M_{\mathbf{w},v}$ is the intertwining operator with $M_{\mathbf{w},v}$ defined by:\[M_{\mathbf{w},v}: I(s,\chi_{1,v},\chi_{2,v})\ra I(1-s,\chi_{2,v},\chi_{1,v}),\quad \phi_v\mapsto M_{\mathbf{w},v}(\phi_v)(\cdot):=\int_{F_v}\phi_v\left(\begin{pmatrix}0 & -1 \\1 & 0\end{pmatrix}\begin{pmatrix}
1 & x \\
 0 & 1\end{pmatrix}\cdot\right)dx.\] For $\beta\in F_v^\times$, define the $\beta$-th local Whittaker integral \[W_\beta(\phi_v,g_v)=\int_{F_v}\phi_v\left(\begin{pmatrix}0 & -1 \\1 & 0\end{pmatrix}\begin{pmatrix}1 & x \\0 & 1
\end{pmatrix}g_v\right)\psi_v(-\alpha x)dx,\] then $W_\beta(\phi_{\lambda,k},g)=\prod_{v}W_\beta(\phi_{\lambda,k,v},g_v)$ for $\beta\in F^\times$. Let $g_\infty=\begin{pmatrix}
x & y \\
0 & 1
\end{pmatrix}$ and consider $s=0$, then we get Fourier expansion of $E_{\lambda,k}$ at the cusp $[\infty,g_f]$.

In the following, we consider the local Fourier coefficients of the local test vector. 
For $v$ finite place, let \[\phi_v=\phi_{\lambda,v}.\] For $\sigma\in \Sigma$, let $\phi_{\sigma,k}$ be the standard weight $k$ section. Consider the local Whittaker function: \[W_\beta\left(\phi_v,g_v\right)\Big|_{s=0}=\int_{F_v}\phi_v\left(\left(\begin{pmatrix}
0 & -1 \\1 & 0\end{pmatrix}\begin{pmatrix}
1 & x \\ 0 & 1\end{pmatrix}g_v\right)\right)\psi_v(-\beta x)dx\Big|_{s=0},\] here $\beta\in F^\times$ and $g_v\in G(F_v)$. In application, we may consider $g_v=\begin{pmatrix} 1 & 0 \\0 & \fc_v^{-1}\end{pmatrix}$ for $v$ finite with $\fc_v=1$ for $v\in S$ and $g_\sigma=\begin{pmatrix}y_\sigma& x_\sigma\\0 & 1 \end{pmatrix}$, $y_\sigma\in\BR_{>0}$ for $\sigma\in \Sigma$.

For each place $v\notin S_\infty\sqcup S_{\text{n-split}}\sqcup S_0$, let $c_v$ be as before with $s=0$. Then for almost all places, $c_v=1$ and $c_v$ is $p$-adic unit for all $v$.

Let $\lambda^*:=\lambda\cdot |\cdot|_{\BA_K^\times}^{-1/2}$.
\begin{enumerate}
\item[(1)] If $v\notin S\bs S'$, 
\[\begin{aligned}
W_\beta\left(\phi_v,\begin{pmatrix}1 & 0 \\0 & \fc_v^{-1}\end{pmatrix}\right)\Big|_{s=0}&=\int_{F_v}\phi_v\left(\left(\begin{pmatrix} 0 & -1 \\1 & 0\end{pmatrix}\begin{pmatrix}1 & x \\0 & 1\end{pmatrix}\begin{pmatrix}1 & 0 \\ 0 & \fc_v^{-1}\end{pmatrix}\right)\right)\psi_v(-\beta x)dx\Big|_{s=0}\\
&=c_v\int_{F_v}\chi_{1,v}(\fc_v^{-1})\int_{F_v^\times}\Phi(t,tx\fc_v^{-1})\chi_{1,v}\chi_{2,v}^{-1}(t)d^\times t\psi_v(-\beta x)dx\\
&=c_v|\fc_v\varpi_v^{\delta_v}|_v\chi_{1,v}(\fc_v^{-1})\mathbf{1}_{\CO_{F,v}}(\beta\fc_v)\sum_{0\leq n\leq \ord_v(\beta\fc_v)}\lambda_v^*(\varpi_v^n).
\end{aligned}\]
If $\lambda$ is self-dual, \[W_\beta\left(\phi_v,\begin{pmatrix} 1 & 0 \\0 & \fc_v^{-1}\end{pmatrix}\right)\Big|_{s=0}=c_v|\fc_v\varpi_v^{\delta_v}|_v\chi_{1,v}(\fc_v^{-1})\mathbf{1}_{\CO_{F,v}}(\beta\fc_v)\sum_{0\leq n\leq \ord_v(\beta\fc_v)}\tau_{K_v/F_v}(\varpi_v)^{n},\] thus under $\ord_v(\beta\fc_v)\geq 0$, it is nonzero if and only if $v$ is not inert, or $v$ inert and $\ord_v(\beta\fc_v)$ is even.
\item[(2)] If $v\in S_{\text{split},1}$
 \[\begin{aligned}
W_\beta\left(\phi_v,\mathfrak{1}\right)\Big|_{s=0}&=\int_{F_v}\phi_v\left(\left(\begin{pmatrix}0 & -1 \\1 & 0\end{pmatrix}\begin{pmatrix} 1 & x \\0 & 1\end{pmatrix}\right)\right)\psi_v(-\beta x)dx\Big|_{s=0}\\
&=c_v\int_{F_v}\int_{F_v^\times}\phi_1(t)\phi_2(tx)\chi_{1,v}\chi_{2,v}^{-1}(t)d^\times t\psi_v(-\beta x)dx\\
&=c_v\int_{F_v^\times}\phi_1(t)\chi_{1,v}\chi_{2,v}^{-1}(t)|t|_v^{-1}\int_{F_v}\phi_2(x)\psi_v(-\beta t^{-1} x)dx d^\times t.\\
\end{aligned}\] If $\lambda_w$ is ramified,
\[\begin{aligned}
W_\beta\left(\phi_v,\mathfrak{1}\right)\Big|_{s=0}&=c_v|\varpi_v^{\delta_v}|_v\int_{F_v^\times}\phi_1(t)\chi_{1,v}\chi_{2,v}^{-1}(t)|t|_v^{-1}(\mathfrak{1}_{\CO_{F,v}^\times}\lambda_{w})(\beta t^{-1}) d^\times t\\
&=c_v|\varpi_v^{\delta_v}|_v|\beta|_v^{-1}\lambda_w(\beta)\int_{\CO_{F,v}^\times}\phi_1(\beta u)\lambda_{\ov{w}}(\beta u)d^\times u(t=\beta u)\\
&=c_v|\varpi_v^{\delta_v}|_v\phi_1(\beta)\lambda_v^*(\beta),
\end{aligned}\]
if $\lambda$ is further self-dual, \[W_\beta\left(\phi_v,\mathfrak{1}\right)\Big|_{s=0}=c_v|\varpi_v^{\delta_v}|_v\phi_1(\beta)\tau_{K_v/F_v}(\beta).\]
If $\lambda_{w}$ and $\lambda_{\ov{w}}$ are unramified, thus $v$ is coprime to conductor of $\lambda$, the result is the same as (1):
\[\begin{aligned}
W_\beta\left(\phi_v,\mathfrak{1}\right)\Big|_{s=0}
&=c_v|\varpi_v^{\delta_v}|_v\mathfrak{1}_{\CO_{F,v}}(\beta)\sum_{0\leq n\leq\ord_{v}(\beta)}\lambda_v^*(\varpi_v^n),
\end{aligned}\]
if $\lambda$ is further self-dual, then \[ W_\beta\left(\phi_v,\mathfrak{1}\right)\Big|_{s=0}=c_v|\varpi_v^{\delta_v}|_v\mathfrak{1}_{\CO_{F,v}}(\beta)\sum_{0\leq n\leq\ord_{v}(\beta)}\tau_{K_v/F_v}(\varpi_v)^n.\]
If $\lambda_w$ is unramified and $\lambda_{\ov{w}}$ is ramified,
\[\begin{aligned}
W_\beta\left(\phi_v,\mathfrak{1}\right)\Big|_{s=0}&=c_v|\varpi_v^{\delta_v}|_v\int_{F_v^\times}\phi_1(t)\chi_{1,v}\chi_{2,v}^{-1}(t)|t|_v^{-1}\mathfrak{1}_{\CO_{F,v}}(\beta t^{-1}) d^\times t\\
&=c_v|\varpi_v^{\delta_v}|_v \mathfrak{1}_{\CO_{F,v}}(\beta).
\end{aligned}\]
\item[(3)] If $v\in S_{\text{split},2}$, \[\begin{aligned}
W_\beta\left(\phi_v,\mathfrak{1}\right)\Big|_{s=0}&=c_v|\varpi_v^{\delta_v}|_v\int_{F_v^\times}\phi_1(t)\chi_1\chi_2^{-1}(t)|t|_v^{-1}(\mathfrak{1}_{\CO_{F,v}^\times}\lambda_{w})(\beta t^{-1}) d^\times t(t=\beta u)\\
&=c_v|\varpi_v^{\delta_v}|_v|\beta|_v^{-1}\lambda_w(\beta)\int_{\CO_{F,v}^\times}\phi_1(\beta u)\lambda_{\ov{w}}(\beta u)du\\
&=c_v|\varpi_v^{\delta_v}|_v\phi_1(\beta)\lambda_v^*(\beta).
\end{aligned}\]
\item[(4)] If $\sigma\in\Sigma$, recall that $\phi_{\sigma,k}\in I(s,\chi_{1,\sigma},\chi_{2,\sigma})$ is the weight $k$ section such that $\phi_{\sigma,k}(1)=1$, more precisely, \[\phi_{\sigma,k}(g)=\chi_{1,\sigma}(\det(g))J(g,i)^{-k}\left|\frac{\det(g)}{J(g,i)\overline{J(g,i)}}\right|_v^s.\]
\item[] We have
\[\begin{aligned}W_\beta\left(\phi_{\sigma,k},\begin{pmatrix}y_\sigma &x_\sigma \\0& 1 \end{pmatrix}\right)\Big|_{s=0}&=\int_{\BR}\phi_{\sigma,k}\left(\begin{pmatrix}0 & -1 \\1 & 0\end{pmatrix}\begin{pmatrix}1 & x \\0 & 1\end{pmatrix}\begin{pmatrix}y_\sigma &x_\sigma \\0& 1\end{pmatrix}\right)e^{-2\pi i\sigma(\beta) x}dx\Big|_{s=0}\\
 &=\int_{\BR}\left(\frac{1}{x_\sigma+x+y_\sigma i}\right)^ke^{-2\pi i\sigma(\beta) x}dx\\
 &=c_v\cdot (-1)^k\frac{(2\pi i)^k}{(k-1)!}\lambda_{\sigma}^*(\beta) e^{2\pi i\sigma(\beta)(x_\sigma+y_\sigma i)}\mathbf{1}_{\BR_+}\sigma(\beta)
\end{aligned}\] for any $x_\sigma\in\BR,\ y_\sigma\in\BR_{>0}$, where $c_\sigma$ satisfies $\prod_{\sigma\in\Sigma}c_\sigma=|D_F|_{\BR}^{-1/2}$.
\item[(5)] If $v\in S_0$
 \[\begin{aligned}W_\beta\left(\phi_v,\mathfrak{1}\right)\Big|_{s=0}&=|\varpi_v^{\delta_v}|_v\mathbf{1}_{\CO_{F,v}}(\beta).
\end{aligned}\]
\item[(6)] If $v\in S_{\text{n-split}}$,
in this case, we have \[G(F_v)=B(F_v)\rho(K_v^\times).\] Let $\phi_v$ be the unique function of $G(F_v)$ such that  \[\phi_v\left(\begin{pmatrix} a & b \\ 0 & d\end{pmatrix}\rho(k)\varsigma_v\right)=\chi_1(a)\chi_2(d)\left|\frac{a}{d}\right|_v^{s}\chi_v^{-1}(k)\] for all $\begin{pmatrix}a & b \\0 & d\end{pmatrix}\in B(F_v)$ and $k\in K_v^\times$. In this case, $\fc_v=1$. The goal is to estimate the $p$ adic valuation of the following local Fourier coefficient:
 \[\begin{aligned} W_\beta\left(\phi_v,\begin{pmatrix} 1 & 0 \\0 & \fc_v^{-1}\end{pmatrix}\right)\Big|_{s=0}&=\int_{F_v}\phi_v\left(\left(\begin{pmatrix} 0 & -1 \\ 1 & 0 \end{pmatrix}\begin{pmatrix}1 & x \\ 0 & 1\end{pmatrix}\right)\right)\psi_v(-\beta x)dx.\\\end{aligned}\]
 Note that
\[\begin{pmatrix} 0 & -1 \\1 & x\end{pmatrix}=\begin{pmatrix}-1 &-a_v\varpi_v^{-\delta_v} \\x& xa_v\varpi_v^{-\delta_v}-1\end{pmatrix}\cdot\begin{pmatrix}a_v\varpi_v^{-\delta_v} & 1 \\-1& 0 \end{pmatrix}=\begin{pmatrix}-1 &-a_v\varpi_v^{-\delta_v} \\ x& xa_v\varpi_v^{-\delta_v}-1\end{pmatrix}\cdot\varsigma_v\] and \[\begin{pmatrix} -1 &-a_v\varpi_v^{-\delta_v} \\x& xa_v\varpi_v^{-\delta_v}-1 \end{pmatrix}=\begin{pmatrix}\frac{1}{(xa_v\varpi_v^{-\delta_v}-1)^2-x^2\vartheta^2} & * \\0 & 1\end{pmatrix}\begin{pmatrix}xa_v\varpi_v^{-\delta_v}-1 & x\vartheta^2 \\x & xa_v\varpi_v^{-\delta_v}-1\end{pmatrix}\] thus the above integral equals to 
 \[\begin{aligned}
 &\int_{F_v}\lambda_v^{-1}(x(\vartheta+\varpi_v^ {-\delta_v}a_v)-1)\psi_v(-\beta x)dx\\
 =&\frac{\lambda_v^{-1}(u)}{|N(u)|_v}\psi_v\left(\frac{-\beta \varpi_v^ {-\delta_v}a_v}{N(\vartheta+\varpi_v^ {-\delta_v}a_v)}\right)\int_{F_v}\lambda_v^{-1}(x+\frac{1}{4\vartheta})\psi_v(N(u_v)\beta x)dx\\
\end{aligned}\]
\end{enumerate}
here $N$ is the norm from $K_v^\times$ to $F_v^\times$, $u_v=\frac{2\vartheta}{\vartheta+\varpi_v^{-\delta_v}a_v}$.
Let $d^0x$ be the self-dual Haar measure with respect to $\psi_v$, since $v$ is prime to $p$, the invariant \[A_a(\lambda_v^*,\theta):=\int_{F_v} (\lambda_v^*)^{-1}\left(x+\frac{1}{4\theta}\right)\left|x+\frac{1}{4\theta}\right|_{K_v^\times}^{-s-\frac{1}{2}}\psi_v(-a x)d^0x\Big|_{s=0},\quad a \in F_v, \theta\in K_v, \ov{\theta}=-\theta\] is $p$-integral. We will discuss the arithmetic properties of $A_a(\lambda_v^*,\theta)$ in the \S\ref{S5}.

In the following, we consider the constant term of the Eisenstein series.
\begin{prop}\label{cusps}
Assume one of the following:
\begin{enumerate}
\item [(i)] $k\geq 2$,
\item [(ii)] $k=2$ and $\omega_{\lambda}\neq |\cdot|_{\BA_F^\times}^2$,
\item [(iii)] $k=1$, (i) $\{w\ |\ \text{$\lambda_w$ is ramified and $w|\prod_{v\in S_{\text{split},1}}$ and $w\notin \Sigma_p^c$}\}\bigsqcup S_{\text{split},2}\neq\emptyset$; or (ii) $\lambda$ self-dual, $\fl$ split in $K$.
\end{enumerate} Then for any cusp of the form $\left[\infty,\begin{pmatrix} 1 & 0 \\0 & \fc^{-1} \end{pmatrix}\right]$ with $\fc\in\BA_{K}^{\times,(S)}$, the constant term of Fourier expansion of $E_{\lambda,k}$ at this cusp is zero.
\end{prop}
\begin{proof}
The constant term of the Fourier coefficient of $E_{\lambda,k}$ at the cusp $\left[\infty,\begin{pmatrix}
1 & 0 \\
 0 & \fc^{-1}
 \end{pmatrix}\right]$ is given by \[\phi_{\lambda,k}(g)|_{s=0}+M_{\mathbf{w}}(\phi_{\lambda,k})(g)|_{s=0},\] here $g_f=\begin{pmatrix} 1 & 0 \\0 & \fc^{-1}
 \end{pmatrix}$ and $g_\sigma=\begin{pmatrix}y_\sigma & x_\sigma \\0 & 1\end{pmatrix}$ for $\sigma\in\Sigma$. First note that $\phi_{\lambda,v}(1)=0$ for $v\in S_0$, thus $\phi_{\lambda,k}(g)|_{s=0}=0$.

 For all $v\notin S\bs S'$, $\phi_{\lambda,v}$ is the standard spherical function with $\phi_{\lambda,v}(1)=1$ and \[M_{\mathbf{w},v}(\phi_{\lambda,v})\begin{pmatrix}1&\\ &\fc_v^{-1}\end{pmatrix}=\chi_{1,v}(\fc_v^{-1})|\fc_v|_vL(2s-1,\omega_{\lambda,v}).\]

For each $\sigma\in\Sigma$, \[\begin{aligned}M_{\mathbf{w},\sigma}(\phi_{k,\sigma})(g_\sigma)&=M_{\mathbf{w},\sigma}(1)\phi_{k,\sigma}(-s,g_\sigma)
\end{aligned}\] where $\phi_{k,\sigma}$ is the standard weight $k$ section in $I(s,\chi_{2,v},\chi_{1,v})$ and the constant is given by \[M_{\mathbf{w},\sigma}(1)=(-i)^k\frac{\pi}{2^{k+2s-2}}\frac{\Gamma(2s+k-1)}{\Gamma(s+k)\Gamma(s)},\]and \[\phi_{k,\sigma}(1-s,g_\sigma)=\chi_{2,\sigma}(\det(g_\sigma))\ov{J(g_\sigma,i)}^{k}\left|\frac{\det(g_\sigma)}{J(g_\sigma,i)\overline{J(g_\sigma,i)}}\right|_\sigma^{1-s}.\]

Thus for the first two cases, the vanishing of $M_{\mathbf{w}}(\phi_{\lambda,k})(g)|_{s=0}$ follows from the finiteness of $L_f(-1,\omega_\lambda)$ and $M_{\mathbf{w},\sigma}(1)|_{s=0}=0$. When $k=1$ and under the condition (i), the result follows from $ M_{\mathbf{w},v}(\phi_{\lambda,k,v})(1)|_{s=0}=0$ for $v|\ff\ff^c$ or $v\in S_{\text{split},2}$, finiteness of $L_f(-1,\omega_\lambda)$ and finiteness of $M_{\mathbf{w},\sigma}(\phi_{\lambda,k,\sigma})(g_\sigma)|_{s=0}$. When $k=1$, $\lambda$ self-dual and $\fl$ splits in $K$, the result follows from $L(-1,\omega_{\lambda,\fl})^{-1}=0$, finiteness of $L_f(-1,\omega_{\lambda,\fl})$ and finiteness of $M_{\mathbf{w},\sigma}(\phi_{\lambda,k,\sigma})(g_\sigma)|_{s=0}$.
 \end{proof}

\subsubsection{Local zeta integral}

In the following, $\phi_v=\phi_{\lambda,v}$.
 For $\varepsilon$ a finite order character on $\Cl'_{\fl^n}$, we may consider the following zeta integral
\[\alpha_v(\phi_v,\varepsilon_v)(s):=\int_{K_v^\times/F_v^\times}\varepsilon_v\chi_v(t)\phi_v(s,\rho(t)\varsigma){d}^\times t.\]
By the construction of the test vector, we have
\begin{enumerate}
\item[(1)] If $v\notin S\bs S'$, then
\[\begin{aligned}
\alpha_v(\phi_v,\varepsilon_v)(s)&=\int_{\CO_{K_v}}\varepsilon_v\lambda_v(t)|t|_v^sd^\times t\\
&=L(s,\varepsilon_v\lambda_v).
\end{aligned}\]
\item[(2)] If $v|S_{\text{split},1}$, let $v=w\ov{w}$ with $w\in S_1$
\[\begin{aligned}
\alpha_v(\phi_v,\varepsilon_v)(s)&=\int_{K_v^\times}\varepsilon_v(t)\lambda_v(t)\Phi_v([0,1]\rho(t)\varsigma_v)|t|_v^sd^\times t\\
&=\int_{F_v^\times}\varepsilon_{\ov{w}}\lambda_{\ov{w}}(t)\phi_1(t)|t|_v^sd^\times t\int_{F_v^\times}\varepsilon_{w}\lambda_{w}(t)\phi_2(\varpi_v^{\delta_v} t)|t|_v^sd^\times t\\
&=L(s,\epsilon_v\lambda_v)G(\lambda_w\epsilon_w|\cdot|_w^s),
\end{aligned}\]where \[G(\lambda_w\epsilon_w|\cdot|_w^s)=|\varpi_v|_w^{e_w}
\sum_{u\in(\CO_{F,v}/\varpi_v^{e_{w}}\CO_{F,v})^\times}\lambda_w\epsilon_w|\cdot|_w^s(u\varpi_v^{-e_w})\psi_w(u\varpi_v^{\delta_v-e_w}).\]
\item[(3)] If $v\in S_{\text{split},2}$, let $v=w\ov{w}$ with $w\in S_2$
\[\alpha_v(\phi_v,\varepsilon_v)(s)=Z(s,\lambda\varepsilon_w,\wh{1_{\CO_{K,w}^\times}\lambda_w}).\]
\item[(4)] If $v\in S_{\text{n-split}}\sqcup S_{\infty}$
\[\begin{aligned}
\alpha_v(\phi_v,\varepsilon_{v})(s)=\vol(K_v^\times/F_v^\times, d^\times t_v).
\end{aligned}\]
\item[(5)] If $v\in S_0$, suppose $\varepsilon_v$ is trivial on $\CO_{\varpi_v^n}^\times$ the local zeta integral: \[\alpha_v(\phi_v,\varepsilon_v)(s)=\vol(\pi(\CO_{\varpi_v^n}^\times), d^\times t_v)\chi_2(\varpi_v^n)|\varpi_v|_v^{-s},\]where $\pi:K_\fl^\times\ra K_\fl^\times/F_\fl^\times$.
\end{enumerate}

\subsection{$p$-integral Waldspurger formula and construction of measure}\label{4.4}

Consider \[\BE_{\lambda,k}=\frac{\Gamma_{\Sigma}(k\Sigma)}{\sqrt{|D_F|_{\BR}}(2\pi i)^{k\Sigma}}E_{\lambda,k}.\]

By our choice of $S_{\text{split},2}$, the assumption of Proposition \ref{cusps} is satisfied. Thus for any $\fc\in\BA_{K}^{\times,(S)}$, $\BE_{\lambda,k}$ has no constant term at the cusp $\left(\infty,\begin{pmatrix}
1 & 0 \\
0& \fc^{-1}
\end{pmatrix}\right)$.
Let $\beta\in F^\times$, recall that the $\beta$-th Fourier coefficient of $\BE_{\lambda,k}$ at the cusp $\left(\infty,\begin{pmatrix}
1 & 0 \\
0& \fc^{-1}
\end{pmatrix}\right)$
has the form
\[\displaystyle \alpha_\beta\left(\BE_{\lambda,k},\begin{pmatrix}
1 & 0 \\
0& \fc
\end{pmatrix}\right)=(-1)^{k\Sigma}\mathbf{1}_{F_{>0}}(\beta)\cdot\prod_{\text{$v$ is finite}}W^*_{\beta,v},\] here \[\displaystyle W^*_{\beta,v}=\begin{cases}
|\varpi_v^{\delta_v}|_v^{-1} W_{\beta}\left(\phi_{\lambda,v},\begin{pmatrix}
1 & 0 \\
0 & \fc_v^{-1}
\end{pmatrix}\right)\Big|_{s=0},\quad &\text{if $v$ is prime to $p$},\\
\displaystyle |\varpi_v^{\delta_v}|_v^{-1} \prod_{\substack{\sigma:F\ra \ov{\BQ}\\ \text{$\iota_p\circ \sigma$ induces $v$ }}}\sigma(\beta)^{k-1} W_{\beta}\left(\phi_{\lambda,v},\begin{pmatrix}
1 & 0 \\
0 & \fc_v^{-1}
\end{pmatrix}\right)\Big|_{s=0}, \quad &\text{if $v|p$}.
\end{cases}\] For $\beta$ prime to $p$, up to a $p$-adic unit, $W^*_{\beta,v}$ is given by the following:
\[\begin{cases}
\displaystyle\mathbf{1}_{\CO_{F,v}}(\beta\fc_v)\cdot\sum_{0\leq n\leq \ord_v(\beta\fc_v)}(\lambda_v^*(\varpi_v))^{n},& \text{if $v\notin S\bs S'$},\\
A_{N(u)\beta}(\lambda_v^*,\vartheta),& \text{if $v\in S_{\text{n-split}}$},\\
\mathbf{1}_{\CO_{F,v}}(\beta),& \text{if $v\in S_0$},\\
\mathfrak{1}_{\CO_{F,v}^\times}(\beta)+c_v(\beta),& v\in S_{\text{split}},\\
\end{cases}\]where $c_v(\cdot)$ is a $\ov{\BZ}_p$ valued function on $\CO_{F,v}\bs\CO_{F,v}^\times$. 

By $q$-expansion principle (cf.~Proposition ~\ref{q-exp}) and local analysis of Fourier coefficient, then $\BE_{\lambda,k}$ lies in $\CM_{k}(U_\lambda,\CO_{(\wp)})$, here $\CO_{(\wp)}$ is the localization at $\wp$ of the ring of integer $\CO$ of a number field $E$ with $\wp$ the prime induced by $\iota_p$, $U_\lambda$ is given \S\ref{level}. It has zero constant term at the cusps.

\label{CM def}Let $K^*$ be the reflex field of $K$ and $\fp^*$ the prime of $K^*$ above $p$ induced by $\iota_p$. Let $K'$ be the maximal abelian extension of $K^*$ unramified above $\fp^*$. Let $\fp'$ be the prime of $K'$ induced by $\iota_p$. For any $a\in\BA_{K,f}^{\times,(p)}\cdot\CO_{K,p}^\times$, write $x_n(a)$ the CM point on the Igusa scheme corresponding to $[\vartheta,\rho(a)\varsigma_{n,f}]$. Then by Proposition \ref{pintCM}, the CM points $x_n(a)$ are defined over $W:=\CO_{K',(\fp')}$ and have good reduction. Fix a $\CO_{F}\otimes W$ basis $\omega$ of the invariant differential of the abelian scheme given by $x_0(1)$. Denote $j$ the $\mu_{p^\infty}\otimes \CO_F^*$ level structure induced by $x_0(1)$ and let $dt/t$ be the standard $\CO_{F}\otimes W$ basis of invariant differentials on $\wh{\BG}_m\otimes \CO_F^*$. Denote $\CW$ the $p$-adic completion of $W$. Let $\Omega_\infty\in (\BC^\times)^\Sigma$ and $\Omega_p\in (\CO_F\otimes \CW)^\times\subset (\ov{\BZ}_p^\times)^\Sigma$ such that $\omega=\Omega_pj_*dt/t$ and $\omega=\Omega_\infty dz$.
Write $x_n(a)=(y_n(a),j_a)$, where $y_n(a)$ is the CM point obtained from $x_n(a)$ by forgetting the level structure at $p$ and $j_a$ is the $\mu_{p^\infty}\otimes\CO_F^*$ structure in $x_n(a)$.\label{CMp}
\begin{prop}\label{p adic}\
We have

\[\frac{1}{(\Omega_p a_{\Sigma_p^c})^{k\Sigma+2\kappa}}\theta^\kappa\wh{\BE}_{\lambda,k}(y_n(a),j_{a,*}(dt/t))=\frac{(2\pi i)^{k\Sigma+2\kappa}}{\Omega_\infty^{k\Sigma+2\kappa}}\delta_k^\kappa\BE_{\lambda,k}(x_n(a),2\pi i dz)\in\ov{\BZ}_p\cap\ov{\BQ}.\]

\end{prop}
\begin{proof}
Let $\omega_n(a)$ be any $\CO_{F}\otimes W$ basis of invariant differential of the abelian scheme $A_{n}(a)$ given by $x_n(a)$. Suppose that $\omega_n(a)=\Omega_{n,p}(a)j_{a,*}dt/t$ and $\omega_n(a)=\Omega_{n,\infty}(a)dz$. The CM splitting and Hodge splitting of the de Rham cohomology of $p$ ordinary CM abelian scheme coincide \cite[(2.6.3)]{Ka}, thus
 \[\frac{1}{\Omega_{n,p}(a) ^{k\Sigma+2\kappa}}\theta^\kappa\wh{\BE}_{\lambda,k}(y_n(a),j_{a,*}(dt/t))=\frac{(2\pi i)^{k\Sigma+2\kappa}}{\Omega_{n,\infty}^{k\Sigma+2\kappa}(a)}\delta_k^\kappa\BE_{\lambda,k}(x_n(a),2\pi i dz).\]
Now fix an invariant differential $\omega$ of the abelian scheme in the data $x_0(1)/W$ as before. Denoted by $\BC^\Sigma/\CL_{[\vartheta,\rho(a)\varsigma_{n,f}]}$ the complex abelian variety given by the point $x_n(a)$ as in complex uniformization, then \[\CL_{[\vartheta,\rho(a)\varsigma_{n,f}]}\otimes \BZ_p=\CL_{[\vartheta,\varsigma_{0,f}]}\otimes \BZ_p=\CO_{K,p}.\] We now choose $\omega_n(a)$ be the invariant differential on $A_n(a)_{/\BC}\simeq\BC^\Sigma/\CL_{[\vartheta,\rho(a)\varsigma_{n,f}]}$ such that its pull back under the natural quotient map \[\BC^\Sigma/\CL_{[\vartheta,\varsigma_{0,f}]}\cap \CL_{[\vartheta,\rho(a)\varsigma_{n,f}]}\ra \BC^{\Sigma}/\CL_{[\vartheta,\rho(a)\varsigma_{n,f}]}\] equals to the pull back of $\omega$ under the quotient map \[\BC^\Sigma/\CL_{[\vartheta,\varsigma_{0,f}]}\cap \CL_{[\vartheta,\rho(a)\varsigma_{n,f}]}\ra \BC^{\Sigma}/\CL_{[\vartheta,\varsigma_{0,f}]}.\] Then $\omega_n(a)$ is a $\CO_{F}\otimes W$ basis of invariant different of $A_n(a)_{/W}$ and $\Omega_{n,\infty}(a)=\Omega_\infty$. Furthermore we have $\omega_n(a)=\Omega_\infty dz$ and $\omega_n(a)=\Omega_pa_{\Sigma_p^c}j_{a,*}dt/t$. To see the last equality, write $a=a_pa^{(p)}$, then $\omega_n(a^{(p)})=\Omega_pj_{a^{(p),*}}dt/t$. As $j_{a}=j_{a^{(p)}}\circ a_{\Sigma_p^c}$,
\[j_{a^{(p)},*}dt/t=a_{\Sigma_p^c}j_{a,*}dt/t.\]
\end{proof}

We have the following relation between Fourier coefficients:
\[\alpha_\beta(\theta^\kappa\wh{\BE}_{\lambda,k},\fc)=\beta^\kappa \alpha_\beta(\wh{\BE}_{\lambda,k},\fc)\] for cusps $\Big[\infty,\begin{pmatrix} 1 & 0 \\ 0 & \fc^{-1}\end{pmatrix}\Big]$ with $\BA_{K}^{\times,(S)}$.

Assume $\CX$ has infinity type $k\Sigma+\kappa$, $k\in \BZ_{>0}$, $\kappa\in \BZ_{\geq_0}[\Sigma]$ and $\lambda\in \CX$ fixed.

Denoted by $\Cl_n':=\Cl'_{\fl^n}$ and $\Cl_\infty'=\varprojlim_{n}\Cl_n'$. Define the measure on $\wh{\Gamma}$ to be
\[\int_{\Cl_\infty'}\varepsilon d \phi_\lambda :=a_\fl^{-n}\sum_{[a]_n\in\Cl_n'}\varepsilon(a)\wh{\chi}^{(p)}(a)(\wh{\lambda}_{\Sigma_p^c}\wh{\lambda}_{\Sigma_p}^{-1})(a_{\Sigma_p^c})\theta^{\kappa}\wh{\BE}_{\lambda,k} (x_n(a)),\quad a\in \BA_{K,f}^{\times,(p)}\cdot\CO_{K,p}^\times, n\gg0,\] here $a_\fl$ is the $U(\fl)$ eigen value of $E_{\lambda,k}$ and $[a]_n$ is the equivalent class of $a$ in $\Cl_n'$. Note that when choosing $a\in \BA_{K,f}^{\times,(p)}$, the definition has a simpler form
\[\int_{\Cl_\infty'}\varepsilon d \phi_\lambda :=a_\fl^{-n}\sum_{[a]_n\in\Cl_n'}\varepsilon(a)\wh{\chi}^{(p)}(a)\theta^{\kappa}\wh{\BE}_{\lambda,k} (x_n(a)),\quad a\in \BA_{K,f}^{\times,(p)}, n\gg0.\]

\begin{thm}\label{measure}There exists a $p$-integral measure on $\Cl_\infty'$ such that for any finite order $\varepsilon$ on $\Cl_\infty'$,
\[\frac{1}{\Omega_p^{k\Sigma+2\kappa}}\int_{\Cl_\infty'}\varepsilon d \phi_\lambda=\prod_{w\in S_2}Z(0,\lambda\varepsilon_{w},\wh{1_{\CO_{K,{w}}^\times}})\prod_{w\in S_1}G(\lambda_w)\frac{\pi^\kappa\Gamma_{\Sigma}(k\Sigma+\kappa)}{\sqrt{|D_F|_\BR}\Im(\vartheta)^\kappa\Omega_\infty^{k\Sigma+2\kappa}} L_f^{(S_{\text{split}, 2}\fl)}(0,\varepsilon\lambda)\in\ov{\BZ}_p\cap\ov{\BQ} .\]
\end{thm}
\begin{proof}
This follows directly from the Proposition \ref{p adic} and Corollary \ref{wald}.
\end{proof}
\begin{remark}\ \begin{itemize}
\item[(1).] The measure does not depend on the choice of $n$.
\item[(2).] For $a\in\BA_{K,f}^{\times,(p)}\cdot\CO_{K,p}^\times$ the value $\wh{\chi}^{(p)}(a)(\wh{\lambda}_{\Sigma_p^c}\wh{\lambda}_{\Sigma_p}^{-1})(a_{\Sigma_p^c})\theta^{\kappa}\wh{\BE}_{\lambda,k} (x_n(a))$ only depends on the class $[a]_n$ of $a$ in $\Cl_n'$.
\item[(3).] Once we choose $S_{\text{split},2}$ such that the Frobenius associated to any prime of $K$ above $S_{\text{split},2}$ maps non-trivially into $\Gamma$, then the extra factor $\prod_{w\in S_2}Z(0,\lambda\varepsilon_{w},\wh{1_{\CO_{K,{w}}^\times}})L_{S_{\text{split},2}}(0,\epsilon\lambda)^{-1}$ is harmless for $p$-stability problem.
\end{itemize}
\end{remark}
\section{$p$-adic Properties of the Eisenstein Series}
\label{S5}
In this section, we study $p$-divisibility of Eisenstein series (cf.~Proposition~\ref{mod}) constructed in the \S\ref{S4}. In \S\ref{S51}--\S\ref{mtd}, we consider $p$-divisibility of local Fourier coefficients. In \S\ref{pprop}, we consider the $p$-divisibility of the global Fourier coefficients, where the (residual) theta dichotomy is crucially used to relate global $p$-divisibility to the sum of local $p$-divisibility in the (residually) self-dual case respectively.
\subsection{Preliminaries}\label{S51}
In this subsection, we consider the $p$-adic valuation of the following local integral, which appears as local Fourier coefficients of the Eisenstein series (cf.~\S\ref{add}).

Fix $v\nmid \fp$ a finite place of $F$ such that $K_v^\times$ is a field and $\theta\in K_v$ such that $\ov{\theta}=-\theta$.
Recall for $\lambda_v$ a character on $K_v^\times$, $\beta\in F_v$, \[A_\beta(\lambda_v^*,\theta):=\int_F (\lambda_v^*)^{-1}\left(x+\frac{1}{4\theta}\right)\left|x+\frac{1}{4\theta}\right|_{K_v^\times}^{-s-\frac{1}{2}}\psi_v(-\beta x)d^0x\Big|_{s=0},\] here $d^0x$ is the self-dual Haar measure with respect to $\psi_v$ and recall $\lambda_v^*=\lambda_v|\cdot|_{K_v^\times}^{-1/2}$.

In the following lemma, let $F/\BQ_v$ be a local field with uniformizer $\varpi$, $\psi$ is an additive character on $F$, and $dx$ the Haar measure on $F$ that is self-dual with respect to $\psi$. 
\begin{lem}\label{pe00}
Suppose that $v$ is prime to $p$. Let $\phi:F\ra \CO_{\BC_p}$ be a locally constant function such that
\begin{itemize}
\item [(i)] $\phi(x_0)=0$ at some point $x_0\in F$;

\item [(ii)] For any $\eta\in F^\times$, there exists $n\gg 0$ such that for any $m\leq -n$,\[\int_{\varpi^m\CO_F^\times}\phi(x)\psi(\eta x)dx=0.\]\end{itemize}
Let \[\mu(\phi):=\inf_{x\in F}\ord_p(\phi(x)),\] and \[A_\eta(\phi):=\int_{F}\phi(x)\psi(\eta x)dx\] be the $\eta$-th Fourier coefficient of $\phi$ for $\eta\in F^\times$, then \[\mu(\phi)=\inf_{\eta\in F^\times}\ord_p(A_\eta(\phi))\] and there exists a $\eta\in F^\times$ such that \[\mu(\phi)=\ord_p A_\eta(\phi)\] if $\mu(\phi)\neq \infty$.
\end{lem}
\begin{proof}
It is sufficient to prove that for any rational number $t$, $p^t|\phi(x)$ for all $x$ is equivalent to $p^t|A_\eta(\phi)$ for all $\eta\in F^\times$.

Assume that $p^t|\phi(x)$ for all $x\in F$. By the second assumption, $A_\eta(\phi)$ is well-defined and the integral is in fact a finite sum, thus $p^t|A_\eta(\phi)$ for all $\eta\in F^\times$.

Now assume $p^t|A_\eta(\phi)$ for all $\eta\in F^\times$. Let $m\in \BZ$, consider the following:
\[\begin{aligned}
&\int_{\varpi^m\CO_F^\times}A_\eta(\phi)\psi(\eta y)d\eta\\
&=\int_{F}\phi(x)\int_{\varpi^m\CO_F}\psi(\eta (x-y))d\eta dx-\int_{F}\phi(x)\int_{\varpi^{m+1}\CO_F}\psi(\eta (x-y))d\eta dx\\
&=\int_{y+\varpi^{-m}\CO_F}\phi(x)dx-\int_{y+\varpi^{-m-1}\CO_F}\phi(x)dx.
\end{aligned}\]
Let $a_m(y)=\int_{y+\varpi^{-m}\CO_F}\phi(x)dx$, then \[a_m(y)\equiv a_{m+1}(y)\pmod{p^t}\] for all $m$. Note that $a_m(y)$ is constant on $y+\varpi^{-m}\CO_F$, thus $a_m(y)\pmod{p^t}$ is constant and does not depend on $y$. Now for $y=x_0$, $a_m(x_0)\equiv 0\pmod{p^t}$ for $m$ sufficiently small, thus $p^t|\phi(y)$ for all $y$.
\end{proof}
\begin{cor}\label{pe01}
For $K_v/F_v$ a field extension,
\[\inf_{\beta\in F_v}\ord_p(A_\beta(\lambda_v^*,\theta))=\mu_p(\lambda_v):=\inf_{x\in K_v^\times}\ord_p(\lambda_v(x)-1).\]
\end{cor}

\subsection{Theta dichotomy}\label{td}
In this subsection, suppose that $\lambda$ is self-dual, we recall results on the relation between the local root number and $p$-indivisibility of the local period $A_\beta(\lambda_v^*,\theta)$. It is crucial in the work of Hsieh \cite{Hs1} to study the $p$-divisibility of the Eisenstein series in the self-dual case.

It turns out that for a character $\lambda_v$ on $K_v^\times$, $A_\beta(\lambda_v^*,\theta)$ can be viewed as Whittaker functional of certain Siegel-Weil section on $U(1,1)$ and the theta dichotomy gives the following relation to the local root numbers (cf.~\cite[(6.10)~and~Cor.~8.3~(ii)]{hks}):

Let $\psi_{{K_v}}=\psi_v\circ\tr_{K_v/F_v}$. 
\begin{prop}\label{pe1}Let $\lambda_v$ be a character on $K_v^\times$ and let $\lambda^\vee(x)=\lambda(x^{c,-1})|\cdot|_{K_v^\times}$, where $c$ is the generator of $\Gal(K_v/F_v)$. Then
\[A_\beta((\lambda_v^\vee)^*,\theta)=\frac{\epsilon(1/2,\lambda_v^*,\psi_{K_v})}{\lambda_v^*(2\beta\theta)}A_\beta(\lambda_v^*,\theta).\]
\end{prop}
Note that the result does not depend on the choice of $\psi_v$.
\begin{cor}\label{pe2}\
 Suppose that $\lambda$ is self-dual, then for each finite place $v\in S_{\text{n-split}}$, if $A_{\beta}(\lambda_v^*,\theta)\neq 0$, then \[\epsilon(1/2,\lambda_v^*,\psi_{K_v})=\tau_{K_v/F_v}(\beta)\lambda_v^*(2\theta).\]
\end{cor}
In general, the local root number for self-dual character has the following properties (cf.~\cite[Prop.~3.7]{AT}):
\begin{prop}\label{pe3}Further choose $\theta$ such that exists $a\in F_v$ such that $\{1,\theta+a\}$ is a $\CO_{F,v}$-basis of $\CO_{K,v}$.
Suppose that $\lambda$ is self-dual, then 
\begin{enumerate}
\item $\epsilon({1}/{2},\lambda_v^*,\psi_{K_v})=\pm\lambda^*_v(2\theta)$ if $v$ is finite;
\item $\epsilon({1}/{2},\lambda_v^*,\psi_{K_v})=\lambda^*_v(2\theta)$ if $v$ split;
\item $\epsilon({1}/{2},\lambda_v^*,\psi_{K_v})=(-1)^{e_v+c_v}\lambda_v^*(2\theta)$ if $v$ inert. Here $e_v$ (resp. $c_v$) is the conductor of $\lambda_v^*$ (resp. $\psi_{v}$).
\end{enumerate}
\end{prop}
\begin{remark}Recall that if the infinity type of $\lambda$ is $\Sigma+\kappa(1-c)$ with $\kappa\in\BZ_{\geq 0}[\Sigma]$, then $\epsilon({1}/{2},\lambda_\sigma^*,\psi_{K_\sigma})=i^{1+2\kappa_\sigma}=\lambda_\sigma^*(2\theta)$ for $\sigma\in \Sigma$ and the standard character $\psi_\sigma$. 
\end{remark}
\subsection{Modulo $p$ theta dichotomy}\label{mtd}
In this subsection, we consider modulo $p$ analogue of the Corollary \ref{pe2} and Proposition \ref{pe3} in the case $\lambda$ is residually self-dual.

The main result is the following: Let $p$ be an odd prime. Let $v\nmid p$ be a place of $F$ such that $K_v$ is a field.

Recall we have $\inf_{\beta\in F_v}\ord_p(A_\beta(\lambda_v^*,\theta))=\mu_p(\lambda_v)$ by Corollary \ref{pe01}.

\begin{prop}\label{pe4}
Suppose that $\lambda_v$ is residually self-dual, then \[\epsilon(1/2,\lambda_v^*,\psi_{K_v})\equiv\pm \lambda_v^*(2\theta)\pmod{\fm_p}.\]Further more,
\begin{itemize}
\item [(1)] If $\mu_p(\lambda_v)=0$, then for $\beta\in F_v$, \[A_{\beta}(\lambda_v^*,\theta)\nequiv 0\pmod{\fm_p}\] implies that
\[\epsilon(1/2,\lambda_v^*,\psi_{K_v})\equiv\tau_{K_v/F_v}(\beta)\lambda_v^*(2\theta)\pmod{\fm_p}.\]
\item [(2)] Suppose that $\mu(\lambda_v)>0$.
\begin{itemize}
\item[(a)] If $\ord_p (\lambda_v^*(\varpi_v)+1)=\inf_{x\in\CO_{K,v}^\times}\ord_p(\lambda_v(x)-1)$, then for each $\delta\in\{\pm 1\}$, there exists $\beta_\delta\in F_v$ such that
\[\ord_p A_{\beta_\delta}(\lambda_v^*,\theta)=\mu_p(\lambda_v)\] and
\[\epsilon(1/2,\lambda_v^*,\psi_{K_v})\equiv \delta\cdot \tau_{K_v/F_v}(\beta_\delta)\lambda_v^*(2\theta)\pmod{\fm_p}.\]
\item[(b)] If $\ord_p (\lambda_v^*(\varpi_v)+1)>\inf_{x\in\CO_{K,v}^\times}\ord_p(\lambda_v(x)-1)$, then for $\beta\in F_v$,
\[\ord_p A_{\beta}(\lambda_v^*,\theta)=\mu_p(\lambda_v)\] implies that
\[\epsilon(1/2,\lambda_v^*,\psi_{K_v})\equiv \tau_{K_v/F_v}(\beta)\lambda_v^*(2\theta)\pmod{\fm_p}.\]
\item[(c)] If $\ord_p (\lambda_v^*(\varpi_v)+1)<\inf_{x\in\CO_{K,v}^\times}\ord_p(\lambda_v(x)-1)$, then for $\beta\in F_v$,
\[\ord_p A_{\beta}(\lambda_v^*,\theta)=\mu_p(\lambda_v)\] implies that
\[\epsilon(1/2,\lambda_v^*,\psi_{K_v})\equiv -\tau_{K_v/F_v}(\beta)\lambda_v^*(2\theta)\pmod{\fm_p}.\]
\end{itemize}
\end{itemize}
\end{prop}
\begin{remark} 
If $p>2$ and $\mu_p(\lambda_v)>0$, we must have $K_v/F_v$ is unramified and $\lambda_v|_{F_v^\times}$ is unramified.
\end{remark}
We need the following properties of local epsilon factor:
\begin{prop}\label{pe5}let $\theta$ be satisfied that there exists $a\in F_v$ such that $\{1,\theta+a\}$ is a $\CO_{F,v}$-basis of $\CO_{K,v}$.
Assume $\lambda$ is residually self-dual , then for any inert place $v$ such that $\lambda|_{F_v^\times}$ is unramified,
\[\epsilon({1}/{2},\lambda_v^*,\psi_{K_v})\equiv(-1)^{e_v+c_v}\lambda_v^*(2\theta)\pmod{\fm_p}.\] Here $e_v$ (resp. $c_v$) is the conductor of $\lambda_v^*$ (resp. $\psi_{v}$).
\end{prop}
\begin{proof}
The proof is the same as the proof of Proposition 3.7 (iii) in \cite{AT}, except consider everything modulo $p$. The key is that the Gauss sum $S$ in \cite{AT} only depends on $\lambda_v|_{\CO_{K,v}^\times}$, so it is still a positive real number under the assumption $\lambda|_{F_v^\times}$ is unramified.
\end{proof}
\begin{proof}[Proof of Proposition \ref{pe4}]
We will see the first claim holds from the proof of the furthermore part.

Part (1) is a direct consequence of the Proposition \ref{pe1}.

Note that $\mu(\lambda_v)>0$ and $p>2$ implies that $K_v/F_v$ can not be ramified, otherwise, since $\lambda_v$ is residually self-dual, there exists a unit $u$ such that $\lambda_v(u)\equiv -1\pmod{\fm_p}$, thus $\lambda_v(u)-1$ is a $p$-adic unit whenever $p>2$, contradiction. In the following, we give proof of (2):
May assume the conductor of $\psi_v$ is trivial and let $\theta$ be satisfied that exists $a\in F_v$ such that $\{1,u=\theta+a\}$ is a $\CO_{F,v}$-basis of $\CO_{K,v}$. Let \[A'_\beta(\lambda_v^*,\theta)=\int_{F_v}\lambda_v^{-1}(x+u)\psi_v(-\beta x)d^0x,\] then up to a $p$-adic unit, $A_\beta(\lambda_v^*,\theta)$ equals to $A'_{\frac{\beta}{4\theta^2}}(\lambda_v^*,\theta)$. As $\ord_v(4\theta^2)$ is even, thus it is enough to show the result by replacing $A_\beta(\lambda_v^*,\theta)$ with $A'_{\beta}(\lambda_v^*,\theta)$.

For $v$ inert, if $\mu_{p}(\lambda_v)>0$ and $p>2$, then $p|(|\varpi_v|_v+1)$ and $\lambda_v|_{\CO_{F,v}^\times}=1$. We must have the conductor of $\lambda_v$ is $\leq 1$.

\begin{itemize}
\item[(i)] If $\ord_v(\beta)<-1$, then $A'_{\beta}(\lambda_v^*,\theta)=0$.
\item[(ii)] If $\ord_v(\beta)=-1$, \[A'_{\beta}(\lambda_v^*,\theta)=\sum_{u\in \CO_{F,v}/\varpi_v\CO_{F,v} } \lambda_v^{-1}(x+u)\psi_v(\beta\cdot x)|\varpi_v|_v\]and by harmonic analysis on $\CO_{F,v}/(\varpi_v)$, we have\[\inf_{\beta\in \varpi_v^{-1}\CO_{F,v}^\times}\ord_p(A'_{\beta}(\lambda_v^*,\theta))=\inf_{x\in\CO_{K,v}^\times}\ord_p((\lambda_v(x)-1)).\]
\item[(iii)] If $\ord_{v}(\beta)\geq 0$, \[A'_{\beta}(\lambda_v^*,\theta)=\sum_{j=1}^{\ord_v(\beta)}\lambda_v^*(\varpi_v^j)(1-|\varpi_v|_v)-|\varpi_v|_v\cdot\lambda_v^*(\pi_{v}^{\ord_v(\beta)+1})-|\varpi_v|_v.\]
\end{itemize}
Now let $t=\lambda_v^*(\varpi_v)+1$, for $\ord_v(\beta)\geq 0$, view $A'_{\beta}(\lambda_v^*,\theta)$ as the value of a polynomial at $t$.

The constant term is given by
\[\begin{cases*}
0,& if $\ord_v(\beta)$ is even,\\
-(1+|\varpi_v|_v), & if $\ord_v(\beta)$ is odd.\\
\end{cases*}\]
The linear term is given by
\[ \begin{cases*}
-|\varpi_v|_v-\frac{\ord_v(\beta)}{2}(1+|\varpi_v|_v),& if $\ord_v(\beta)$ is even,\\
\frac{\ord_v(\beta+1)}{2}(1+|\varpi_v|_v), & if $\ord_v(\beta)$ is odd.\\
\end{cases*}\]
Thus \[\inf_{\ord_v(\beta)\in 2\BZ_{\geq 0}}\ord_{p}(A'_{\beta}(\lambda_v^*,\theta))= \ord_{p}(t)\] and
 \[\inf_{\ord_v(\beta)\in 1+2\BZ_{\geq 0}}\ord_{p}(A'_{\beta}(\lambda_v^*,\theta))\geq \min\{\ord_p(|\varpi_v|_v+1),\ 2\ord_p(t)\}.\]
The lemma follows from the following observation \[\inf_{x\in\CO_{K,v}^\times}\ord_p((\lambda_v(x)-1))<\ord_p(|\varpi_v|_v+1)\] and Proposition \ref{pe5}.
\end{proof}

\subsection{$p$-adic properties of Eisenstein series}\label{pprop}
Let $\fc(1)=\det(\varsigma_{0,f})^{-1}$, then $\fc(1)\in\BA_F^{\times,(S)}$ and the CM point $x_0(1)$ lies in the geometric irreducible component that contains the cusp $\left(\infty,\begin{pmatrix}
1 & 0 \\
0 & \fc(1)^{-1}
\end{pmatrix}\right)$. For $a\in\BA_{K}^{\times,(S)}$, let $\fc(a)=\fc(1)N(a)^{-1}$, here $N$ is the norm from $\BA_K^\times$ to $\BA_F^\times$.

The following is the main result of this subsection on $p$-indivisibility of sufficiently many Fourier coefficients of the normalized $p$-integral Eisenstein series. Let $\BE_{\lambda,k}$ be as in \S\ref{4.4}.

The invariants $\mu(\CX)$, $\mu(\CX^\pm)$ appeared in the following proposition in separated cases will be given in \S\ref{541}--\S\ref{543} , which closely related to (residual) root number in the (residually) self-dual case.
\begin{prop}\label{mod}Suppose that $\lambda|_{F_\fl^\times}$ is unramified.
\begin{itemize}
\item[(1)] If $\CX$ is not residually self-dual, then there exists $\mu_p(\CX)\in\BQ_{\geq 0}$ such that 

\[p^{-\mu_p(\CX)}\BE_{\lambda,k}\] is a geometric modular form over $\ov{\BZ}_p$ and the set of $\beta\in \CO_{F,(p\fl)}$ satisfying the following properties is dense in $\CO_{F,\fl}$.
\begin{itemize}
\item [(i)] $\beta$ is prime to $\{p\}\cup\{v|\CD_{K/F}, v\nmid \fl, \text{and}\ v\notin S_{\text{n-split}}\}$;
\item [(ii)] $\alpha_\beta\left(p^{-\mu_p(\CX)}\BE_{\lambda,k},\begin{pmatrix} 1 & 0 \\ 0 & \fc(a)^{-1} \end{pmatrix}\right)\ \nequiv 0\pmod {\fm_p}$ for some $a\in\BA_K^{\times,(S)}$. 
\end{itemize}
\item[(2)] If $\CX$ is self-dual, then there exists $\mu_p(\CX^+)\in\BQ_{\geq 0}$ such that \[p^{-\mu_p(\CX^+)}\BE_{\lambda,k}\] is a geometric modular form over $\ov{\BZ}_p$ and
\begin{itemize}
\item[(a)] If $\fl$ is split and the global root number $\epsilon(\lambda)=1$, then the set of $\beta\in \CO_{F,(p\fl)}$ satisfying the following properties is dense in $\CO_{F,\fl}$.
\begin{itemize}
\item [(i)] $(\beta,p)=1$;
\item [(ii)] $\alpha_\beta\left(p^{-\mu_p(\CX^+)}\BE_{\lambda,k},\begin{pmatrix}1 & 0 \\0 & \fc(a)^{-1}\end{pmatrix}\right)\ \nequiv 0\pmod {\fm_p}$ for some $a\in\BA_K^{\times,(S)}$.
\end{itemize}
\item[(b)] If $\fl$ is inert, let $\delta^+\in\{0,1\}$ be such that for all $\lambda'\in\CX^+$, the conductor of $\lambda'_\fl$ is congruent to $\delta^+$ modulo $2$. Then the set of $\beta\in \CO_{F,(p\fl)}$ with $\ord_{\fl}(\beta)=\delta^+$ satisfying the following properties is dense in $\varpi_\fl^{\delta^+}\CO_{F,\fl}^\times$.

\begin{itemize}
\item [(i)] $(\beta,p)=1$;
\item [(ii)] $\alpha_\beta\left(p^{-\mu_p(\CX^+)}\BE_{\lambda,k},\begin{pmatrix}1 & 0 \\ 0 & \fc(a)^{-1}\end{pmatrix}\right)\ \nequiv 0\pmod {\fm_p}$ for some $a\in\BA_K^{\times,(S)}$.
\end{itemize}
\end{itemize}
\item[(3)] If $\CX$ is not self-dual but residually self-dual and $p>2$, then
\begin{itemize}
\item[(a)] If $\fl$ is split, then there exists $\mu_p(\CX)\in\BQ_{\geq 0}$ such that \[p^{-\mu_p(\CX)}\BE_{\lambda,k}\] is a geometric modular form over $\ov{\BZ}_p$ when restricting to irreducible components that contain cusps of the form \[\left[\infty,\begin{pmatrix}1 &0\\0& \fc(a)\end{pmatrix}\right]\] for $a\in \BA_{K}^{\times,(S)}$. The set of $\beta\in \CO_{F,(p\fl)}$ satisfying the following properties is dense in $\CO_{F,\fl}$.
 
\begin{itemize}
\item [(i)] $(\beta,p)=1$;
\item [(ii)] $\alpha_\beta\left(p^{-\mu_p(\CX)}\BE_{\lambda,k},\begin{pmatrix}1 & 0 \\ 0 & \fc(a)^{-1}\end{pmatrix}\right)\ \nequiv 0\pmod {\fm_p}$ for some $a\in\BA_K^{\times,(S)}$.
\end{itemize}
\item[(b)] If $\fl$ is inert, then there exists $\mu_p(\CX^\pm)\in\BQ_{\geq 0}$ such that \[p^{-\mu_p(\CX^m)}\BE_{\lambda,k}\] is a geometric modular form over $\ov{\BZ}_p$ where $m\in\{\pm\}$ is such that \[\mu_p(\CX^m)=\min\{\mu_p(\CX^+),\mu_p(\CX^-)\}.\] Let $\delta^m\in\{0,1\}$ be such that for all $\lambda'\in\CX^m$, the conductor of $\lambda'_\fl$ is congruent to $\delta$ modulo $2$. Then the set of $\beta\in \CO_{F,(p\fl)}$ with $\ord_{\fl}(\beta)=\delta^m$ satisfying the following properties is dense in $\varpi_\fl^{\delta^m}\CO_{F,\fl}^\times$.
\begin{itemize}
\item [(i)] $(\beta,p)=1$;
\item [(ii)] $\alpha_\beta\left(p^{-\mu_p(\CX^m)}\BE_{\lambda,k},\begin{pmatrix} 1 & 0 \\0 & \fc(a)^{-1}\end{pmatrix}\right)\ \nequiv 0\pmod {\fm_p}$ for some $a\in\BA_K^{\times,(S)}$.
\end{itemize}
\end{itemize}
\end{itemize}
\end{prop}

\label{muX}
In the following left subsections, we prove Proposition \ref{mod} in separated cases.
\subsubsection{The case $\lambda$ is not residually self-dual}\label{541}
\begin{lem}\label{ressf}Let $\CP$ be a finite set of primes of $F$ that contain all infinite places, all finite places $v|p$, and those places $v$ such that $\lambda_v$ is ramified. Then,
$\lambda$ is not residually self-dual if and only if there exists an inert prime $v\notin \CP$ and \[\lambda_v(\varpi_v)|\varpi_v|_v^{-1}\nequiv -1\pmod{\fm_p}.\]
\end{lem}
Let \[\mu_p(\CX)=\sum_{v\in S_{\text{n-split}}}\mu_p(\lambda_v),\] where $S_{\text{n-split}}$ is the set of primes prime to $\fl$ such that $\lambda$ is ramified. 
\begin{proof}[Proof of Proposition \ref{mod}]

 By Corollary \ref{pe01}, explicit Fourier coefficients of $\BE_{\lambda,k}$ in \S\ref{4.4}, and $q$-expansion principle (cf.~Proposition~\ref{q-exp}), $p^{-\mu_p(\CX)}\BE_{\lambda,k}$ is defined over $\ov{\BZ}_p$.

For any $u\in\CO_{F,\fl}$ and $r>0$ an integer, need to find $\beta\in\CO_{F,(p\fl)}$ such that $\beta\equiv u\pmod {\fl^r}$, $(\beta,p)=1$, and exists $c(a)$ for some $a\in\BA_K^{\times,(S)}$ such that \[\alpha_\beta\left(p^{-\mu_p(\CX)}\BE_{\lambda,k},\begin{pmatrix} 1 & 0 \\0 & \fc(a)^{-1} \end{pmatrix}\right)\nequiv 0\pmod {\fm_p}.\]
Let $\eta\in\BA_F^\times$ such that
\begin{enumerate}
\item [(a)] For $v\in S_{\text{n-split}}$, $p^{-\mu_p(\lambda_v)}A_{\eta_v N(u_v)}(\lambda_v^*,\vartheta)$ is $p$-adic unit, where $u_v=\frac{2\vartheta}{\vartheta+\varpi_v^{-\delta_v}a_v}$ as in \S\ref{add};
\item [(b)] For $v=\fl$, $\eta_v\equiv u\pmod{ \fl^r}$;
\item [(c)] For $v\in S_{\text{split}}$, $\eta_v$ is a $v$ adic unit;
\item [(d)] $\eta_v=\fc(1)_v^{-1}$ at other places.
\end{enumerate}
If $\tau_{K/F}(\eta)=1$, then $\eta$ has the form $\eta=\beta N(a)^{-1}$ for some $\beta\in F^\times$ and $a\in\BA_{K}^\times$. May choose $\beta$ close enough to $\eta_v$ at the places in $S$ such that $\beta$ also satisfies the properties (a) to (c) of $\eta_v$ as above at the places in $S$ and $\beta$ is a unit at the places $\{v|\ v|\CD_{K/F},v\neq \fl, v\notin S_{\text{n-split}}\}$. Then the $\beta$-th coefficient of $p^{-\mu_p(\CX)}\BE_{\lambda,k}$ at the cusp $\left(\infty,\begin{pmatrix} 1 & 0 \\ 0 & \fc(a^{(S)})^{-1} \end{pmatrix}\right)$ is a $p$-adic unit.

If $\tau_{K/F}(\eta)=-1$, by the Lemma \ref{ressf}, there exists an inert place $v$ of $F$ not in $S$ such that \[\lambda_v(\varpi_v)|\varpi_v|_v^{-1}\nequiv -1\pmod{\fm_p}\] and $\fc(1)_v$ is a $v$-adic unit.
Then considering $\eta'=\eta\varpi_v$, by the same analysis as above, we prove the proposition.
 \end{proof}

 \subsubsection{The case $\lambda$ is self-dual}\label{542}
In this case, the assumption $\lambda|_{F_\fl^\times}$ is unramified implies $\fl$ is unramified in $K/F$.
 Let \[\mu_p(\CX^+)=\sum_{v\in S_{\text{n-split}}}\mu_p(\lambda_v).\] 
 \begin{proof}[Proof of Proposition \ref{mod}]
Same as in the non residually self-dual case, $p^{-\mu_p(\CX^+)}\BE_{\lambda,k}$ is defined over $\ov{\BZ}_p$.

 Let $u\in\CO_{F,\fl}$ (resp. $\fl^{\delta^+}\CO_{F,\fl}^\times$) if $\fl$ split (resp. $\fl$ inert). Denoted by $\psi_K=\psi\circ\tr_{\BA_K/\BA_F}$.
Let $r>0$ be an integer, and let $\eta\in\BA_{F,f}^\times$ such that 
\begin{enumerate}
\item [(a)] $p^{-\mu_p(\lambda_v)}A_{\eta_v N(u_v)}(\lambda_v^*,\vartheta)$ is $p$-adic unit for $v\in S_{\text{n-split}}$, where $u_v=\frac{2\vartheta}{\vartheta+\varpi_v^{-\delta_v}a_v}$ as in \S\ref{add};
\item [(b)] $\eta_\fl\equiv u\pmod{\fl^r}$ if $\fl$ is split and $\varpi_\fl^{-\delta^+}\eta_v\equiv \varpi_\fl^{-\delta^+}u\pmod{\fl^r}$ if $\fl$ is inert;
\item [(c)] $\epsilon(1/2, \lambda_v^*,\psi_{K_v})=\tau_{K/F}(\eta_v)\lambda_v^*(2\vartheta) $ for all $v$ if $\fl$ is split and $\epsilon(1/2, \lambda_v^*,\psi_{K_v})=\tau_{K/F}(\eta_v)\lambda_v^*(2\vartheta) $ for all $v\neq \fl$ if $\fl$ is inert;
\item [(d)] For $v\in S_{\text{split}}$, $\eta_v$ is a $v$-adic unit.
\end{enumerate}
By Corollary \ref{pe2} and Proposition \ref{pe3}, this can be down, and we can choose $\eta_v$ at those places prime to $S_0\sqcup S_{\text{n-split}}$ to be $\fc(1)_v^{-1}$.
If $\fl$ is split, by the assumption $\epsilon(\lambda)=1$
\[\tau_{K/F}(\eta)=1.\] If $\fl$ is inert, by the definition of $\delta^+$, $\vartheta$ and Proposition \ref{pe3},
$(-1)^{\delta^+}=\epsilon(\lambda)(-1)^{e_\fl}$ when $\fl$ inert, here $e_\fl$ is the conductor of $\lambda_\fl$.
Thus, we have $\eta=\beta N(a)^{-1}$ for some $a\in\BA_{K}^\times$ and a totally positive element $\beta\in F$.
We further choose $\beta$ close enough to $\eta$ such that $\beta$ and $\eta$ have the same properties at $v\in S$, $a$ is prime to $S$.

Then the $\beta$-th coefficient of $p^{-\mu_p(\CX^+)}\BE_{\lambda,k}$ at the cusp $\left(\infty,\begin{pmatrix} 1 & 0 \\ 0 & \fc(a^{(S)})^{-1}\end{pmatrix}\right)$ is a $p$-adic unit.
 \end{proof}
\subsubsection{The case $\lambda$ is residually self-dual but not self-dual}\label{543}
Assume $p>2$. In this case, the assumption $\lambda|_{F_\fl^\times}$ is unramified implies that $\fl$ is unramified.

\bigskip
{\bf The case $\fl$ is split}
\bigskip

 Let $\mu_p(\CX)$ be the following:
 \begin{itemize}
\item[(i)] \[\mu_p(\CX)=\sum_{v\in S_{\text{n-split}}}\mu_p(\lambda_v)=\sum_{\substack{\text{$v\in S_{\text{n-split}}$ and }\\ \text{$v$ is inert in $K/F$}}}\mu_p(\lambda_v),\] if one of the following holds:
\begin{itemize}
\item[(a)] There exists a $v\in S_{\text{n-split}}$ inert in $K/F$ with $\mu_p(\lambda_v)>0$ and \[\inf_{x\in\CO_{K,v}^\times}\ord_p(\lambda_v(x)-1)=\ord_p (\lambda_v^*(\varpi_v)+1),\]here $\varpi_v$ is any uniformizer $F_v$;
\item[(b)] Any $v\in S_{\text{n-split}}$ such that $v$ inert in $K/F$ with $\mu_p(\lambda_v)>0$, we have\[\ord_p (\lambda_v^*(\varpi_v)+1)\neq\inf_{x\in\CO_{K,v}^\times}\ord_p(\lambda_v(x)-1);\]And the residual root number $\ov{\epsilon}(\lambda)$ equals to \[\prod_{v\in S_{\text{n-split}}}(-1)^{\varepsilon_v},\] where $\varepsilon_v$ is given by \[\varepsilon_v=\begin{cases}
1, & \text{if $\mu_p(\lambda_v)>0$ and $\inf_{x\in\CO_{K,v}^\times}\ord_p(\lambda_v(x)-1)>\ord_p (\lambda_v^*(\varpi_v)+1)$},\\
0,& \text{else}.
 \end{cases}\]
\end{itemize}
\item[(ii)] Any $v\in S_{\text{n-split}}$ such that inert in $K/F$ with $\mu_p(\lambda_v)>0$, we have\[\ord_p (\lambda_v^*(\varpi_v)+1)\neq\inf_{x\in\CO_{K,v}^\times}\ord_p(\lambda_v(x)-1);\]And the residual root number equals to \[-\prod_{v\in S_{\text{n-split}}}(-1)^{\varepsilon_v}.\]Let \[\mu_p(\CX)=\min\{\mu_1,\mu_2\},\] here \[\mu_1=\sum_{\substack{\text{$v\in S_{\text{n-split}}$ and }\\ \text{$v$ is inert in $K/F$}}}\mu_p(\lambda_v)+\inf_{\substack{\text{$v$ is inert in $K/F$}\\ \text{ and $\lambda_v$ is unramified }}}\ord_p(\lambda_v^*(\varpi_v)+1),\] and
\[\mu_2=\min_{v\in S_\text{n-split}}\left\{\sum_{v'\in S_\text{n-split}, v\neq v' }\mu_p(\lambda_{v'})+\inf_{\substack{\beta\in F_v,\\ \tau_{K_v/F_v}(\beta)\equiv\frac{\varepsilon(1/2,\lambda_v^*,\psi_{K_v})}{\lambda_v^*(2\theta)}\cdot (-1)^{\epsilon_v+1}}}\ord_p(A_\beta(\lambda_v^*,\theta))\right\},\] here $\theta\in K_v^\times$ is any element such that $\ov{\theta}=-\theta$.
\end{itemize}
\begin{remark}
Since $p>2$, we have $\mu_p(\lambda_v)=0$ whenever $v$ is ramified in $K/F$.\end{remark}
\begin{prop}For each $a\in\BA_{K}^{\times,(S)}$, the restriction of $p^{-\mu_p(\CX)}\BE_{\lambda,k}$ to the geometric irreducible component that contain the cusp $\left(\infty,\begin{pmatrix}1 & 0\\0 & \fc(a)^{-1}\end{pmatrix}\right)$ is defined over $\ov{\BZ}_p$.
\end{prop}
\begin{proof}
For each $\beta\in F_+^\times$, by $q$-expansion principle (cf.~Proposition~\ref{q-exp}), need to show the $\beta$-th Fourier coefficient $\alpha_\beta\left(p^{-\mu_p(\CX)}\BE_{\lambda,k},\begin{pmatrix}1 & 0 \\ 0 & \fc(a)^{-1} \end{pmatrix}\right)$ of $p^{-\mu_p(\CX)}\BE_{\lambda,k}$ at the cusp $\left(\infty,\begin{pmatrix} 1 & 0 \\0 & \fc(a)^{-1}\end{pmatrix}\right)$ lies in $\ov{\BZ}_p$.

If $\mu_p(\CX)=\sum_{v\in S_{\text{n-split}}}\mu_p(\lambda_v)$, the result follows from Corollary \ref{pe01}.
So now we assume the following
\begin{itemize}
\item [(1)] Any $v\in S_{\text{n-split}}$ such that inert in $K/F$ with $\mu_p(\lambda_v)>0$, we have\[\ord_p (\lambda_v^*(\varpi_v)+1)\neq\inf_{x\in\CO_{K,v}^\times}\ord_p(\lambda_v(x)-1);\]
\item[(2)] The residual root number equals to \[-\prod_{v\in S_{\text{n-split}}}(-1)^{\varepsilon_v},\] here $\varepsilon_v$ as in the definition of $\mu_p(\CX)$.
\end{itemize}
If exists a place $v'\in S_{\text{n-split}}$ such that $\frac{\epsilon(1/2,\lambda_{v'}^*,\psi_{K_{v'}})}{\lambda_v^*(2\vartheta)}\equiv -(-1)^{\varepsilon_{v'}}\tau_{K_{v'}/F_{v'}}(\beta)$, then \[\alpha_\beta\left(p^{-\mu_2}\BE_{\lambda,k},\begin{pmatrix}1 & 0 \\0 & \fc(a)^{-1}\end{pmatrix}\right)\in\ov{\BZ}_p.\]

If for each $v\in S_{\text{n-split}}$ \[\frac{\epsilon(1/2,\lambda_v^*,\psi_{K_v})}{\lambda_v^*(2\vartheta)}\equiv(-1)^{\varepsilon_v}\tau_{K_v/F_v}(\beta)\pmod{\fm_p},\] then by the assumption that the residual root number equals to \[-\prod_{v\in S_{\text{n-split}}}(-1)^{\varepsilon_v},\] there exists an inert place $v'$ prime to $S_{\text{n-split}}$ such that $\epsilon(1/2,\lambda_{v'}^*,\psi_{K_{v'}})\equiv-\lambda_v^*(2\vartheta)\tau_{K_{v'}/F_{v'}}(\beta)\pmod{\fm_p}$.
By Proposition \ref{pe5}, \[\epsilon(1/2,\lambda_{v'}^*,\psi_{K_{v'}})\equiv(-1)^{\ord_{v'} \fc(1)_{v'}}\lambda_{v'}^*(2\vartheta)\pmod\fm_p,\]we have $\ord_{v'}(\beta\fc(1)_{v'})$ is odd.
Thus \[\alpha_\beta\left(p^{-\mu_1}\BE_{\lambda,k},\begin{pmatrix} 1 & 0 \\0 & \fc(a)^{-1}\end{pmatrix}\right)\in\ov{\BZ}_p.\]
\end{proof}
\begin{proof}[Proof of Proposition \ref{mod}]Given $r>0$ and $u\in\CO_{F,\fl}$.
If we are in the case $(i)$ of the definition of $\mu_p(\CX)$. Let $\eta\in\BA_{F,f}^\times$ such that
\begin{enumerate}
\item [(a)] $p^{-\mu_p(\lambda_v)}A_{\eta_v N(u_v)}(\lambda_v^*,\vartheta)$ is $p$-adic unit for $v\in S_{\text{n-split}}$, where $u_v=\frac{2\vartheta}{\vartheta+\varpi_v^{-\delta_v}a_v}$ as in \S\ref{add};
\item [(b)] $\eta_\fl\equiv u\pmod{\fl^r}$;
\item [(c)] $\epsilon(1/2, \lambda_v^*,\psi_{K_v})\equiv(-1)^{\varepsilon_v}\tau_{K/F}(\eta_v)\lambda_v^*(2\vartheta)\pmod{\fm_p} $ for all $v$ finite and nonsplit;
\item [(d)] For $v\in S_{\text{split}}$, $\eta_v$ is a $v$-adic unit.
\end{enumerate}This is possible by Proposition \ref{pe4}. May further choose $\eta_v$ at those places prime to $\fl$, $ S_{\text{n-split}}$ to be $\fc(1)_v^{-1}$. Furthermore, we can choose $\eta$ such that $\tau_{K/F}(\eta)=1$. The left proof is the same as the self-dual case.

Assume we are in the case $(ii)$ of the definition of $\mu_p(\CX)$, we have $\tau_{K/F}(\eta)=-1$. If $\mu_p(\CX)=\mu_1$, let $v$ inert and $v_0$ prime to $S_{\text{n-split}}$ such that $\ord_p(\lambda_v^*(\varpi_{v_0})+1)=\inf_{\substack{\text{$v$ is inert in $K/F$}\\ \text{ and $\lambda_v$ is unramified }}}\ord_p(\lambda_v^*(\varpi_v)+1)$. Let $\eta$ be the same as in the case $(1)$ and $\eta'=\eta\varpi_v$, then $\tau_{K/F}(\eta')=1$, the left proof is the same as the self-dual case.

If $\mu_p(\CX)=\mu_2$, let $v_0\in S_{\text{n-split}}$ and $\eta_{v_0}\in F_{v_0}$ such that \[\varepsilon(1/2,\lambda_v^*,\psi_{K_v})\equiv(-1)^{\epsilon_v+1} \tau_{K_v/F_v}(\beta)\lambda_v^*(2\vartheta)\pmod{\fm_p} \]and \[\mu_2=\sum_{v'\in S_{\text{n-split}}, v'\neq v_0}\mu_p(\lambda_{v'})+\ord_p(A_{\eta_{v_0}}(\lambda_{v_0},\vartheta)).\]
By Proposition \ref{pe4}, we can extend $\eta_{v_0}$ to $\eta\in\BA_{F,f}^\times$ such that \begin{enumerate}
\item [(a)] $p^{-\mu_p(\lambda_v)}A_{\eta_v N(u_v)}(\lambda_v^*,\vartheta)$ is $p$-adic unit for $v\in S_{\text{n-split}}$ not equal to $v_0$, where $u_v=\frac{2\vartheta}{\vartheta+\varpi_v^{-\delta_v}a_v}$ as in \S\ref{add};
 \item [(b)] $\eta_\fl\equiv u\pmod{\fl^r}$;
\item [(c)] $\epsilon(1/2, \lambda_v^*,\psi_{K_v})\equiv(-1)^{\varepsilon_v}\tau_{K/F}(\eta_v)\lambda_v^*(2\vartheta)\pmod{\fm_p} $ for all $v$ finite and nonsplit;
\item [(d)] For $v\in S_{\text{split}}$, $\eta_v$ is a $v$-adic unit.
\end{enumerate}
We have $\tau_{K/F}(\eta)=1$. The left proof is the same as the self-dual case.
\end{proof}

\bigskip
{\bf The case $\fl$ is inert}
\bigskip

The conjectural $\mu_p(\CX^\pm)$ is given by the following:
\begin{itemize}
\item[(i)]\[\mu_p(\CX^+)=\mu_p(\CX^-)=\sum_{v\in S_{\text{n-split}}}\mu_p(\lambda_v)\] if exists a $v\in S_{\text{n-split}}$ inert in $K/F$ with $\mu_p(\lambda_v)>0$ and \[\inf_{x\in\CO_{K,v}^\times}\ord_p(\lambda_v(x)-1)=\ord_p (\lambda_v^*(\varpi_v)+1),\]here $\varpi_v$ is any uniformizer $F_v$;
\item[(ii)] Any $v\in S_{\text{n-split}}$ that is inert in $K/F$ with $\mu_p(\lambda_v)>0$, we have\[\ord_p (\lambda_v^*(\varpi_v)+1)\neq\inf_{x\in\CO_{K,v}^\times}\ord_p(\lambda_v(x)-1).\]
\item[] Let \[\mu_1=\sum_{\substack{\text{$v\in S_{\text{n-split}}$ and }\\ \text{$v$ is inert in $K/F$}}}\mu_p(\lambda_v)+\inf_{\substack{\text{$v\nmid\fl$ is inert in $K/F$}\\ \text{ and $\lambda_v$ is unramified }}}\ord_p(\lambda_v^*(\varpi_v)+1),\] and
\[\mu_2=\min_{v\in S_\text{n-split}}\left\{\sum_{v'\in S_\text{n-split}, v\neq v' }\mu_p(\lambda_{v'})+\inf_{\substack{\beta\in F_v,\\ \tau_{K_v/F_v}(\beta)\equiv\frac{\varepsilon(1/2,\lambda_v^*,\psi_{K_v})}{\lambda_v^*(2\theta)}\cdot (-1)^{\epsilon_v+1}}}\ord_p(A_\beta(\lambda_v^*,\theta))\right\},\] here $\theta\in K_v^\times$ is any element such that $\ov{\theta}=-\theta$.
\begin{itemize}
\item[(a)]\[\mu_p(\CX^+)=\sum_{v\in S_{\text{n-split}}}\mu_p(\lambda_v)\] and \[\mu_p(\CX^-)=\min\{\mu_1,\mu_2\}\] if any $v\in S_{\text{n-split}}$ that is inert in $K/F$ with $\mu_p(\lambda_v)>0$, we have\[\ord_p (\lambda_v^*(\varpi_v)+1)\neq\inf_{x\in\CO_{K,v}^\times}\ord_p(\lambda_v(x)-1);\]And the residual root number equals to \[(-1)^{\delta^+}\frac{\varepsilon(1/2,\lambda_\fl^*,\psi_{K_\fl})}{\lambda_\fl^*(2\vartheta)}\prod_{v\in S_{\text{n-split}}}(-1)^{\varepsilon_v},\]here $\delta^+\in\{0,1\}$ is such that for all $\lambda'\in\CX^+$, the conductor of $\lambda'_\fl$ is congruent to $\delta^+$ modulo $2$.
\item[(b)]\[\mu_p(\CX^-)=\sum_{v\in S_{\text{n-split}}}\mu_p(\lambda_v)\] and \[\mu_p(\CX^+)=\sum_{v\in S_{\text{n-split}}}\mu_p(\lambda_v)=\min\{\mu_1,\mu_2\}\] if any $v\in S_{\text{n-split}}$ that is inert in $K/F$ with $\mu_p(\lambda_v)>0$, we have\[\ord_p (\lambda_v^*(\varpi_v)+1)\neq\inf_{x\in\CO_{K,v}^\times}\ord_p(\lambda_v(x)-1);\]And the residual root number equals to \[(-1)^{\delta^-}\frac{\varepsilon(1/2,\lambda_\fl^*,\psi_{K_\fl})}{\lambda_\fl^*(2\vartheta)}\prod_{v\in S_{\text{n-split}}}(-1)^{\varepsilon_v},\]$\delta^-\in\{0,1\}$ is such that for all $\lambda'\in\CX^-$, the conductor of $\lambda'_\fl$ is congruent to $\delta^-$ modulo $2$. Here $\mu_1, \mu_2$ same as the split case.
\end{itemize}
\end{itemize}

\begin{proof}
Note that $\mu_p(\CX^m)=\min \{\mu_p(\CX^+),\mu_p(\CX^-)\}$ is $\sum_{v\in S_{\text{n-split}}}\mu_p(\lambda_v)$, thus by Corollary \ref{pe01} and $q$-expansion principle (cf.~Proposition~\ref{q-exp}), $p^{-\mu_p(\CX^m)}\BE_{\lambda,k}$ is defined over $\ov{\BZ}_p$.

Let $u\in\fl^{\delta^*}\CO_{F,\fl}^\times$ and $r>0$. Consider $\eta\in \BA_{F,f}^\times$ such that
\begin{enumerate}
\item [(a)] $p^{-\mu_p(\lambda_v)}A_{\eta_v N(u_v)}(\lambda_v^*,\vartheta)$ is $p$-adic unit for $v\in S_{\text{n-split}}$, where $u_v=\frac{2\vartheta}{\vartheta+\varpi_v^{-\delta_v}a_v}$ as in \S\ref{add};
\item[(b)] $\varpi_\fl^{-\delta^*}\eta_v\equiv \varpi_\fl^{-\delta^*}u\pmod{\fl^r}$;
\item [(c)] $\epsilon(1/2,\lambda_v^*,\psi_{K_v})\equiv(-1)^{\varepsilon_v}\tau_{K_v/F_v}(\eta_v)\lambda_v^*(2\vartheta)\pmod{\fm_p}$ for all $v\neq \fl$ and $v$ non split. Here if $v\notin S_{\text{n-split}}$, we define $\varepsilon_v=0$;
 \item [(d)] For $v\in S_{\text{split}}$, $\eta_v$ is a $v$-adic unit.
\end{enumerate}
By Proposition \ref{pe4}, this is possible, and we may further choose $\eta_v=\fc(1)_v^{-1}$ for $v$ prime to $\fl$ and $S_{\text{n-split}}$.

If in the case $(i)$ of the definition of $\mu_p(\CX^\pm)$, we can choose $\eta$ such that $\tau_{K/F}(\eta)=1$. If in the case $(ii)$ of the definition of $\mu_p(\CX^\pm)$, we have
\[\ov{\epsilon}(\lambda)=(-1)^{\delta^*}\cdot\frac{\varepsilon(1/2,\lambda_\fl^*,\psi_{K_\fl})}{\lambda_\fl^*(2\vartheta)}\prod_{v\in S_{\text{n-split}}}(-1)^{\varepsilon_v},\] by Proposition \ref{pe4}, we still have $\tau_{K/F}(\eta)=1$. The left proof is the same as the self-dual case.
\end{proof}
\section{$p$-stability of Hecke L-values}
\label{S6}
The goal of this section is to prove the following main result:
Let $\CX$ be an $\wh{\Gamma}$ orbit of Hecke characters with infinity type $k\Sigma+\kappa(1-c)$, $k\Sigma+\kappa\in\BZ_{>0}[\Sigma]$, $\kappa\in \BZ_{\geq 0}[\Sigma]$.
For $\lambda\in \CX$, let $\CL(\lambda)$ be the algebraic part of $L_f(0,\lambda)$ defined as in \S\ref{mainre}. 
\begin{thm}\label{main}
Suppose that
\begin{itemize}
\item[(i)] $p$ is ordinary, $(p,2 D_F)=1$,
\item[(ii)] $(\fl,p)=1$ and $\lambda|_{F_\fl^\times}$ is unramified for $\lambda\in \CX$,
\item[(iii)] $p$ is odd if $\CX$ is residually self-dual but not self-dual,
\item[(iv)] if $\lambda$ is residually self-dual, $\fl$ inert, then $\rank_{\BZ_\ell}\Gamma=1$.
\end{itemize} Then the following hold:
\item[(1)] If $\CX$ is not residually self-dual, then for Zariski dense $\lambda\in\CX$, \[\ord_p(\CL(\lambda))=\mu_p(\CX).\]\item[(2)] If $\CX$ is residually self-dual, then for \[\mu_p(\CX^m):=\min\{\mu_p(\CX^+),\mu_{p}(\CX^-)\},\] we have \[\ord_p(\CL(\lambda))=\mu_p(\CX^m)\]
 for for Zariski dense $\lambda\in \CX^m$.
\end{thm} 

The invariants $\mu_p(\CX), \mu_p(\CX^\pm)$ are defined in \S\ref{pprop}. 

In \S\ref{6.1}, we reduce the main theorem \ref{main} to the case $k\in \BZ_{> 0}$. In \S\ref{6.2}, we prove the $p$-indivisibility of the measure on a single character is preserved under Galois and Hecke twist. In \S\ref{6.3}, we finish the proof of the main result.
 \subsection{Reduction} \label{6.1}

let $\lambda^\vee(x)=\lambda(x^{c,-1})|x|_{\BA_K^\times}$. Then $\lambda$ is of infinity type $k\Sigma+\kappa(1-c)$ with $k\in \BZ_{> 0}$ and $\kappa\in \BZ_{\geq 0}[\Sigma]$ if and only if $\lambda^\vee$ is of infinity type $k'\Sigma+\kappa'(1-c)$ with $k'\in \BZ_{> 0}$ and $k'\Sigma+\kappa'\in \BZ_{> 0}[\Sigma]$. In fact, $k'=2-k$ and $\kappa'=1-k\Sigma-\kappa$. By the following Proposition \ref{vee}, it is sufficient to show the main theorem \ref{main} for $\lambda$ with infinity type $k\Sigma+\kappa(1-c)$ for $k\in \BZ_{> 0}$ and $\kappa\in\BZ_{\geq 0}[\Sigma]$.
\begin{lem}\label{mvee}For each $v\in S_{\text{n-split}}$, $\mu_p(\lambda_v)=\mu_p(\lambda^\vee_v)$.
\end{lem}
\begin{proof}Let $w$ be the place of $K$ above $v$.
If exists $u\in \CO_{K,w}^\times$ such that $1-\lambda_w(u)$ is a $p$-adic unit, then $\mu_p(\lambda_w)=\mu_p(\lambda^\vee_w)=0$. Now suppose that $\inf_{x\in\CO_{K,w}^\times}\ord_{p}(1-\lambda_w(x))>0$, then we must have the conductor of $\lambda_w$ is equal to $1$, otherwise, there exists a unit $u_0\in \CO_{K,w}^\times$ such that $\lambda_w(u_0)$ is a $q$-th primitive root of unity, where $q$ is the rational prime blow $w$. As $q\neq p$, $1-\lambda_w(u_0)$ is a $p$-adic unit, contradiction.

Hence for $\inf_{x\in\CO_{K,w}^\times}\ord_{p}(1-\lambda_w(x))>0$, we have $p|q^n-1$, here $q^n$ is the cardinality of residue field of $K_w$. Thus \[\inf_{x\in\CO_{K,w}^\times}\ord_{p}(1-\lambda_w(x))\leq \ord_{p}(1-\zeta_p)=\frac{1}{p-1}\leq \ord_p(|y|^{-1}_w-1)\] for any $y\in K_{w}^\times\bs\CO_{K,w} ^\times$. Thus\[\mu_p(\lambda_w^\vee)=\min\left\{\inf_{x\in\CO_{K,w}^\times}\ord_p(1-\lambda_w(x)),\inf_{y\in K_v^\times\bs \CO_{K,w}^\times}\ord_p((|y|^{-1}_w-1)+(1-\lambda_w(y)))\right\}=\mu_p(\CX).\]
\end{proof}
Denote $\CX^\vee$ the $\wh{\Gamma}$ orbit of $\lambda^\vee$. Simply let \[X^\circ:=\begin{cases}
\CX,&\quad \text{if $\CX$ is not residually self-dual},\\
\CX^+,&\quad \text{if $\CX$ is self-dual},\\
 \CX^m,&\quad \text{if $\CX$ is residually self-dual}.
\end{cases}\] Combine the Lemma \ref{mvee}, Proposition \ref{pe1} and the definition of $\mu_p(\CX^\circ)$, we have
\begin{cor}\label{mu2}The following equality holds:
 \[ \mu_p(\CX^\circ)=\mu_p((\CX^\vee)^\circ).\]
\end{cor}
\begin{lem}\label{mu1}For $\lambda\in \CX$, we have
\[\prod_{w\in \Sigma_p}G(\lambda_w)\frac{\pi^\kappa\Gamma_{\Sigma}(k\Sigma+\kappa)L_f(0,\lambda)}{\sqrt{|D_F|_\BR}\Im(\vartheta)^{\kappa}\Omega_\infty^{k\Sigma+2\kappa}}\sim_{\ov{\BZ}_p^\times}\prod_{w\in \Sigma_p}G(\lambda_w^\vee)\frac{\pi^{\kappa'} \Gamma_{\Sigma}(k' \Sigma+\kappa')L_f(0,\lambda^\vee)}{\sqrt{|D_F|_\BR}\Im(\vartheta)^{\kappa' }\Omega_\infty^{k'\Sigma+2\kappa' }},\] here $\vartheta\in K$ is any purely imaginary element such that $2\vartheta D_{F}^{-1}$ is prime to $p$ and $\sim_{\ov{\BZ}_p^\times}$ means up to a $p$-adic unit.
\end{lem}
\begin{proof}
This follows from the complex functional equation \[L(0,\lambda)=\epsilon(0,\lambda)L(0,\lambda^\vee).\]
\end{proof}

 \begin{prop}\label{vee}
Under the same assumption in the main theorem, the $p$-stability property for $\CX^\circ$ is equivalent to the $p$-stability property for $(\CX^\vee)^\circ$.
\end{prop}
\begin{proof}
By Corollary \ref{mu2} and Lemma \ref{mu1}, for $\lambda\in \CX^\circ$, \[p^{-\mu_p(\CX^\circ)}\prod_{w\in \Sigma_p}G(\lambda_w)\frac{\pi^\kappa\Gamma_{\Sigma}(k\Sigma+\kappa)L_f^{}(0,\lambda)}{\sqrt{|D_F|_\BR}\Im(\vartheta)^\kappa\Omega_\infty^{k\Sigma+2\kappa}}\sim_{\ov{\BZ}_p^\times}p^{-\mu_p((\CX^\vee)^\circ)}\prod_{w\in \Sigma_p}G(\lambda_w^\vee)\frac{\pi^{\kappa'}\Gamma_{\Sigma}(k'\Sigma+\kappa')L_f^{}(0,\lambda^\vee)}{\sqrt{|D_F|_\BR}\Im(\vartheta)^{\kappa'}\Omega_\infty^{k'\Sigma+2\kappa'}}.\]
\end{proof}
Recall that $\Gamma$ is the maximal $\BZ_\ell$-free quotient of $\Cl_\infty'$.
In the left of the section, let $\CX$ be an $\wh{\Gamma}$ orbit of Hecke characters with infinity type $k\Sigma+\kappa(1-c)$, $k\in \BZ_{> 0}$ and $\kappa\in \BZ_{\geq 0}[\Sigma]$. Let $\lambda\in \CX^\circ$ be fixed. 
\subsection{Modulo $p$ non-vanishing under Galois action}\label{6.2}
Fix a decomposition \[\Cl_\infty'=\Delta\times\Gamma\] of the group $\Cl_\infty'$, here $\Delta$ is the maximal torsion subgroup of $\Cl_\infty'$.

Let $E$ be a number field sufficiently large such that
 \begin{itemize}
\item [(i)] The geometric modular form $\BE_{\lambda,k}$ is defined over the localization of $\CO_E$ at the prime $\wp$ above $p$ induced by $\iota_p$;
\item [(ii)] $E$ contains the reflex field $K^*$\label{E};
\item [(iii)] $E_\wp$ contains images of $\wh{\chi}$ and $\wh{\lambda}$;
\item [(iv)] $E$ contains all $|\Delta|$-th root of unity.
\end{itemize}
Let $E'$ be the maximal abelian extension of $E$ that is unramified above $\wp$, and let $\wp'$ be the prime of $E'$ induced by $\iota_p$. Then the CM test objects $x_n(t)$ for $t\in \BA_{K}^{\times,(p)}\cdot\CO_{K,p}^\times$ are defined over $\CO_{E',(\wp')}$. Denote $\CW'$ be the completion of the ring of $\CO_{E',(\wp')}$. Let $\wp'$ be the maximal ideal of $\CW$. Denote $\BF=\CW/\wp'$ and $q=\# \CO_{E}/\wp$.

Denote $\varphi_\lambda$ to be the $\ov{\BZ}_p$ valued measure on $\Cl_\infty'$ given by
\[\int_{\Cl_\infty'}\varepsilon\varphi_\lambda:=\int_{\Cl_\infty'}\varepsilon d\phi_\lambda\cdot p^{-\mu_p(\CX^\circ)}\]
here $\CX^\circ$ is defined in \S\ref{6.1}.

For any $b\in \BA_{K,f}^{\times,(p)}\cdot\CO_{K,p}^\times$, denote $\phi_\lambda|[b]$ the $p$-adic measure on $\Cl_\infty'$ defined by \[\int_{\Cl_\infty'}\varepsilon d\varphi_\lambda|[b]:=p^{-\mu_p(\CX^\circ)}a_\fl^{-n}\sum_{[a]_n\in\Cl_n'}\varepsilon(a)\wh{\chi}^{(p)}(a)(\wh{\lambda}_{\Sigma_p^c}\wh{\lambda}_{\Sigma_p}^{-1})(a_{\Sigma_p^c})\theta^\kappa\wh{\BE}_{\lambda,k}(x_n(ab))\in\ov{\BZ}_p\]for $n\gg 0$, here $a\in \BA_{K,f}^{\times,(p)}\cdot\CO_{K,p}^\times$.

Let $D(\wp'|\wp)$ be the decomposition group of $\wp'$ in $\Gal(E'/E)$.
\begin{prop}\label{Gal}
If $\int_{\Cl_\infty'}\varepsilon d\varphi_\lambda\equiv 0 \pmod {\fm_p}$, then for all $\sigma\in D(\wp'|\wp)$ and all $b\in \BA_{K,f}^{\times,(p)}\cdot\CO_{K,p}^\times$,
\[\int_{\Cl_\infty'}\varepsilon^\sigma d\varphi_\lambda|[b]\equiv 0\pmod {\fm_p}.\]
\end{prop}
\begin{proof}
Let $s\in\BA_{E,f}^\times$ be correspond to $\sigma'$ under reciprocity map such that $s_v$ is unit at all $v|p$.
Write $x_n(a)=(y_n(a),j_{a})$ as in \S\ref{CMp}, we first consider the Galois action on $x_n(a)$. By the galois action on CM points, we have $y_n(a)^\sigma=(y_n(N_{\Sigma^*}\circ N_{E/K^*}(s)^{-1,c}a))$.

We claim that there exists $u\in\CO_{K,\Sigma_p}^\times$ does not depend on $a$ such that \[\sigma(y_n(a),j_{a})=(y_n(N_{\Sigma^*}\circ N_{E/K^*}(s)^{-1,c}a),uj_{N_{\Sigma^*}\circ N_{E/K^*}(s)^{-1,c}a}).\]
Denote $A_n(a)$ the abelian variety in $x_n(a)$ and $\BC^\Sigma/\CL_{n, a}$ with $\CL_{n, a}\subset K$ the complex abelian variety corresponding to $[\vartheta,\rho(a)\varsigma_n]$ under the complex uniformization. Fix an isomorphism $\alpha:A_{n}(a)_{\BC}\simeq\BC^\Sigma/\CL_{n, a}$. By the main theorem of complex multiplication (cf.~\cite[Thm.~6.1]{Lang}) there exists a unique isomorphism $\alpha':A_{n}(a)^{\sigma}\simeq \BC^\Sigma/N_{\Sigma^*}\circ N_{E/K^*}(s)^{-1,c}\CL_{n, a}$ over $\BC$ such that the following diagram commutes:
\[\xymatrix{
 K/\CL_{n, a}\ar[d]_{N_{\Sigma^*}\circ N_{E/K^*}(s)^{-1,c}}\ar[rr]^{\alpha}&&A_{n}(a)_{\text{tor}}\ar[d]^{\sigma(\cdot)}\\
 K/N_{\Sigma^*}\circ N_{E/K^*}(s)^{-1,c}\CL_{n, a}\ar[rr]^{\alpha'}&&A_{n}(a)_{\text{tor}}^{\sigma}.
 }\]
Now let $j_a$ be the $\mu_{p^\times}\otimes\CO_F^*$ level structure in $x_n(a)$ and $j_{N_{\Sigma^*}\circ N_{E/K^*}(s)^{-1,c}a}$ the $\mu_{p^\times}\otimes\CO_F^*$ level structure in $x_n(N_{\Sigma^*}\circ N_{E/K^*}(s)^{-1,c}a)$. By construction of CM points, there exists $u_1\in \CO_{K,\Sigma_p}^\times$ only depends on $s$ such that $\alpha^{-1}\circ j_a=u_1\cdot(\alpha')^{-1}\circ j_{N_{\Sigma^*}\circ N_{E/K^*}(s)^{-1,c}a}$. By definition, \[j_a^{\sigma}(x):=\sigma j_a(\sigma^{-1}(x))\\
=(N_{E/K^*}(s)^{-1,c})_{\Sigma_p}u_1\cdot j_{N_{\Sigma^*}\circ N_{E/K^*}(s)^{-1,c}a}(\sigma^{-1}x)=u\cdot j_{N_{\Sigma^*}\circ N_{E/K^*}(s)^{-1,c}a}(x)\] for some $u\in\CO_{K,\Sigma_p}^\times $ only depends on $\sigma$.
Combine with the fact that $\sigma$ fixes images of $\wh{\chi}$, $\wh{\lambda}$ and the definition field of $\wh{\BE}_{\lambda,k}$, we have
\[\begin{aligned}
\sigma\left(\int_{\Cl_\infty'}\varepsilon d\varphi_\lambda|[b]\right)&\equiv \sigma(p^{-\mu_p(\CX^\circ)})\sigma\left(a_\fl^{-n}\sum_{[a]_n\in\Cl_n'}\varepsilon(a)\wh{\chi}^{(p)}(a)(\wh{\lambda}_{\Sigma_p^c}\wh{\lambda}_{\Sigma_p}^{-1})(a_{\Sigma_p^c})\theta^\kappa\wh{\BE}_{\lambda,k}(x_n(ab))\right)\\
&\equiv \sigma(p^{-\mu_p(\CX^\circ)} a_\fl^{-n}\sum_{[a]_n\in\Cl_n'}\sigma(\varepsilon(a))\wh{\chi}^{(p)}(a)(\wh{\lambda}_{\Sigma_p^c}\wh{\lambda}_{\Sigma_p}^{-1})(a_{\Sigma_p^c})\sigma(\theta^\kappa\wh{\BE}_{\lambda,k}(x_n(ab)))\\
&\equiv \sigma(p^{-\mu_p(\CX^\circ)}) a_\fl^{-n}\sum_{[a]_n\in\Cl_n'}\sigma(\varepsilon(a))\wh{\chi}^{(p)}(a)(\wh{\lambda}_{\Sigma_p^c}\wh{\lambda}_{\Sigma_p}^{-1})(a_{\Sigma_p^c})\theta^\kappa\wh{\BE}_{\lambda,k}(\sigma(x_n(ab)))\\
&\equiv u'p^{-\mu_p(\CX^\circ)} a_\fl^{-n}\sum_{[a]_n\in\Cl_n'}\sigma(\varepsilon(a))\wh{\chi}^{(p)}(a)(\wh{\lambda}_{\Sigma_p^c}\wh{\lambda}_{\Sigma_p}^{-1})(a_{\Sigma_p^c})\theta^\kappa\wh{\BE}_{\lambda,k} (x_n(a))
\end{aligned}\] here denote $b'=N_{\Sigma^*}\circ N_{E/K^*}(s)^{-1,c}b$, then \[u'=\frac{\sigma(p^{-\mu_p(\CX^\circ)})}{p^{-\mu_p(\CX^\circ)}}\cdot u^{-(k\Sigma+2\kappa)}\cdot (\sigma(\varepsilon(b'))\wh{\chi}^{(p)}(b')(\wh{\lambda}_{\Sigma_p^c}\wh{\lambda}_{\Sigma_p}^{-1})(b'_{\Sigma_p^c}))^{-1}\]which is a $p$-adic unit.
\end{proof}
\subsection{Proof of main theorem}\label{6.3}
The goal of this subsection is to show that if the $p$-stability result does not hold, then there exists a certain linear combination of $p$-adic modular forms of some level that vanish on certain twist orbits of CM points. This contradicts with the density of CM points and the modulo $p$ non-vanishing of the non-constant Fourier coefficient of these $p$-adic modular forms.
\subsubsection{Preliminaries}
Exists $n_0$ such that $n\geq n_0$, the natural maps induce embedding from $\Gamma^{(n)}:=1+\varpi_\fl^n\CO_{K,\fl}/(1+\varpi_\fl^n\CO_{F,\fl})$ to $\Cl_\infty'$ and $\Gamma$. Let $\Gamma_n=\Gamma/\Gamma^{(n)}$, we have $\Cl_n'=\Delta\times\Gamma_n$, $n\gg0$.
Let $\varepsilon:\Cl_\infty'\ra \ov{\BQ}^\times$ be a finite order character, supposing that $\varepsilon=\nu\times\phi$, $\phi$ is a character on $\Gamma_n$ for some $n$ sufficiently large. We have

\[\begin{aligned}
\int_{\Cl_\infty'}\varepsilon \varphi_\lambda&=p^{-\mu_p(\CX^\circ)}a_\fl(0)^{-n}\sum_{[a]_n\in\Cl_n'}\varepsilon(a)\wh{\chi}^{(p)}(a)(\wh{\lambda}_{\Sigma_p^c}\wh{\lambda}_{\Sigma_p}^{-1})(a_{\Sigma_p^c})\theta^\kappa\wh{\BE}_{\lambda,k}(x_n(a)),
\end{aligned}\]here $a\in\BA_{K,f}^{\times,(p)}\cdot \CO_{K,p}^\times$ and $[a]_n$ its image in $\Cl_n$.

 Define $\Cl^{\text{alg}}$ the subgroup of $\Cl_\infty'$ generated by $\BA_{K,f}^{\times,(\fl)}$ and $\Delta^{\text{alg}}=\Delta\cap \Cl^{\text{alg}}$.

Recall that (cf.~\cite[Lem.~8.23]{Hia}):
 \begin{fact}\
\begin{enumerate}
\item Each nontrivial element of $\Delta^{\alg}$ can be represented by the square free product of primes of $K$ that are ramified in $K/F$ and prime to $\fl$, $S_{\text{n-split}}$. In particular, $\Delta^{\text{alg}}$ is an elementary abelian group of type $(2,\cdots,2)$.
\item Each nontrivial element of $\Cl_\infty'/\Delta^{\text{alg}}\Gamma$ is represented by a split prime of $K$ that is prime to $S$, $\fl\cdot 2\vartheta\CD_{F}^{-1}\CD_{K/F}^{-1}$.
\end{enumerate}
\end{fact}

Let $\CR$ be a set of representatives of $\Delta^{\text{alg}}$ and assume each $\fr\in\CR$ is a square free product of $\varpi_w$, where $\varpi_w$ is a uniformizer of $K_w$ with $K_w/F_v$ ramified and $(v,\fl)=1$. Let $\CS$ be a set of representatives of $\Cl_\infty'/\Delta^{\text{alg}}\Gamma$ and assume each $\fs\in \CS$ equals to some uniformizer $\varpi_w$ of $K_w$ for some split prime $v=w\ov{w}$ of $F$ satisfies the above condition. We may assume that each projection of class of $\fs$ in $\Cl_\infty'$ to $\Gamma$ is represented by element $\fs_{\Gamma}\in \Gamma^{(n_0)}$ for $n_0$ sufficiently large and fixed. (For example, we may choose for each class $\Cl_{n_0}'/\Delta^{\text{alg}}\times\Gamma_{n_0}$ a $\fs=\varpi_w$ as in the fact such that the class of $\fs$ in $\Cl_{n_0}'=\Delta\times \Gamma_{n_0}$ is in $\Delta$). Let $\fs_\Delta=\fs\fs_{\Gamma}^{-1}$.

Thus we know that $\fs_{\Gamma}$ are distinct in $\Gamma^{(n_0)}$ and pairwise independent module the rational subgroup $\Gamma^{(n_0)}\cap K_{(\fl)}^\times/F_{(\fl)}^\times$.

In the following, we write the measure as Hecke twist sum of $p$-integral $p$-adic modular forms (cf.~\eqref{measum}).
Suppose that $\varepsilon=\nu\times\phi$ with branch character $\nu$, then for any $a\in\BA_{K,f}^{\times,(p)}\cdot\CO_{K,p}^\times$ and $n\geq n_0$, \[\begin{aligned}
\int_{\Cl_\infty'}\varepsilon d\varphi_\lambda|[a]&=p^{-\mu_p(\CX^\circ)}a_\fl^{-n}\sum_{\fs\in \CS}\nu\wh{\chi}(\fs_\Delta)\sum_{\fr\in \CR}\nu\wh{\chi}(\fr)\sum_{[t]_n\in \Gamma_n}\phi(t\fs_{\Gamma})\wh{\chi}^{(p)}(t\fs_{\Gamma})(\wh{\lambda}_{\Sigma_p^c}\wh{\lambda}_{\Sigma_p}^{-1})(t_{\Sigma_p^c})\theta^{\kappa}\wh{\BE}_{\lambda,k}(x_n(\fs\fr ta))\\
 &=a_\fl^{-n}\sum_{\fs\in\CS}\nu\wh{\chi}(\fs_{\Delta})\sum_{[t]_n\in\Gamma_n}\phi\wh{\chi}(t)(\wh{\lambda})_{\Sigma_p^c}(\wh{\lambda}_{\Sigma_p}^{-1})(t_{\Sigma_p^c})\BH_\lambda|[\varsigma_{0,f}^{-1}\rho(\fs)\varsigma_{0,f}](x_n(t\fs_{\Gamma}^{-1}a))
 \end{aligned}\] here $t\in \BA_{K,f}^{\times,(p)}\cdot\CO_{K,p}^\times$,\[\BH_{\lambda}=\sum_{\fr\in\CR}\nu\wh{\chi}(\fr)(\theta^\kappa{\BE})|[\varsigma_{0,f}^{-1}\rho(\fr)\varsigma_{0,f}]\] with $\BE=p^{-\mu_p(\CX^\circ)}\wh{\BE}_{\lambda,k}$ $p$-integral.
For $v|\CD_{K/F}$, we may choose the uniformizer to be $a_v-\varpi^{\delta_v}\vartheta$ as in \S\ref{CM}, then
\begin{equation}\label{Hlam}\displaystyle \BH_{\lambda}=\sum_{\fr\in\CR}\nu\wh{\chi}(\fr)(\theta^\kappa{\BE})\left|\left[\begin{pmatrix}1 &0 \\0& N(\fr)\end{pmatrix}\right]\right.,\end{equation}
here $N$ is the norm from $\BA_{K,f}^\times$ to $\BA_{F,f}^\times$.
 Furthermore, by choice of $\fs$ and choice of $\varsigma_v$ in \S\ref{CM}, we have
\[\BH_\lambda|[\varsigma_{0,f}^{-1}\rho(\fs)\varsigma_{0,f}]]=\BH_\lambda\left|
 \left[\begin{pmatrix}1 & 0 \\ 0 & N(\fs)
\end{pmatrix}\right]\right..\]

Now we can write:
\begin{equation}\label{measum}
\int_{\Cl_\infty'}\varepsilon d\varphi_\lambda|[a]=a_\fl^{-n}\sum_{\fs\in\CS}\nu\wh{\chi}(\fs_{\Delta})\sum_{[t]_n\in\Gamma_n}\phi\wh{\chi}(t)(\wh{\lambda})_{\Sigma_p^c}(\wh{\lambda}_{\Sigma_p}^{-1})(t_{\Sigma_p^c})\BH_\lambda\left|
\left[\begin{pmatrix}1 & 0 \\ 0 & N(\fs)
\end{pmatrix}\right]\right.(x_n(t\fs_{\Gamma}^{-1}a)). \end{equation}
Let $g\in G(\BA_{F,f}^{(p)})$ and $U^{(p)}\subset G(\BA_{F,f}^{(p)})$ a compact open subgroup. Recall that $|[g]$ induces maps from $p$-adic modular form with prime to $p$ level $U^{(p)}$ to $gU^{(p)}g^{-1}$, thus for any $p$-adic modular form $f$ with prime to $p$-level $U^{(p)}$, denote $\alpha_\beta(f,g_f)$ its $\beta$-th Fourier coefficient at the cusp $[\infty, g_f]$, we have
 \[\alpha_\beta(f|[g],g_f)=\alpha_{\beta}(f,g_fg).\]

 We need the following fact for later argument.
By the same analysis as above and the construction of $\BE_{\lambda,k}$.
\begin{fact}\label{coeff} For any cusp of the form $\left[\infty,\begin{pmatrix} 1 & 0 \\0 & \fc^{-1}\end{pmatrix}\right]$ with $\fc\in\BA_{K,f}^{\times,(S)}$,
\begin{enumerate}
\item For each $\fr\in\CR$, the Fourier coefficients between $\BE_{\lambda,k}$ and $\BE_{\lambda,k}|[\varsigma_{0,f}^{-1}\rho(\fr)\varsigma_{0,f}]$ has the following relation:
\[\alpha_\beta\left( \BE_{\lambda,k}|[\varsigma_{0,f}^{-1}\rho(\fr)\varsigma_{0,f}],\begin{pmatrix}1 & 0 \\ 0 & \fc^{-1} \end{pmatrix}\right)=\alpha_\beta\left(\BE_{\lambda,k},\begin{pmatrix}1 & 0 \\ 0 & N(\fr)\fc^{-1}\end{pmatrix}\right).\]
\item For each $\fs=\varpi_w\in\CS$, the Fourier coefficients between $\BE_{\lambda,k}$ and $\BE_{\lambda,k}|[\varsigma_{0,f}^{-1}\rho(\fs)\varsigma_{0,f}]$ have the following relation:
\[\alpha_\beta\left(\BE_{\lambda,k}|[\varsigma_{0,f}^{-1}\rho(\fs)\varsigma_{0,f}],\begin{pmatrix} 1 & 0 \\ 0 & \fc^{-1} \end{pmatrix}\right)=\alpha_\beta\left(\BE_{\lambda,k},\begin{pmatrix} 1 & 0 \\ 0 & N(\fs)\fc^{-1} \end{pmatrix}\right).\]
\item For each $u\in \CO_{F,\fl}$, the Fourier coefficients between $\BE_{\lambda,k}$ and
 $\BE_{\lambda,k}\left|\left[\begin{pmatrix}
 1 & \frac{u\varpi_\fl^{\delta_\fl}}{\varpi_\fl^{mr}} \\
 0 & 1
 \end{pmatrix}\right]\right.$ have the following relation:
\[\alpha_\beta\left(\BE_{\lambda,k}\left|\left[\begin{pmatrix}
 1 & \frac{u\varpi_\fl^{\delta_\fl}}{\varpi_\fl^{mr}} \\
 0 & 1
\end{pmatrix}\right]\right.,\begin{pmatrix}
 1 & 0 \\
 0 & \fc^{-1}
 \end{pmatrix}\right)=\psi_\fl\left(\frac{u\varpi_\fl^{\delta_\fl}}{\varpi_\fl^{mr}}\beta\right)\alpha_\beta\left(\BE_{\lambda,k},\begin{pmatrix}
 1 & 0 \\
0 & \fc^{-1}
\end{pmatrix}\right).\] Furthermore, if $\ord_{\fl}(\beta)\geq 1$,
\[\displaystyle \begin{aligned}
&\alpha_\beta\left(\BE_{\lambda,k}\left|\left[\begin{pmatrix}
1 & \frac{u\varpi_\fl^{\delta_\fl}}{\varpi_\fl^{mr}} \\
0 & 1
\end{pmatrix}\begin{pmatrix}
1 & 0 \\
 0 & \varpi_\fl
 \end{pmatrix}\right]\right.,\begin{pmatrix}
 1 & 0 \\
0 & \fc^{-1}
 \end{pmatrix}\right)\\
=&\chi_{1,\fl}(\varpi_\fl)|\varpi_\fl^{-1}|_\fl\psi_\fl\left(\frac{u\varpi_\fl^{\delta_\fl}}{\varpi_\fl^{mr+1}}\beta\right)\alpha_\beta\left(\BE_{\lambda,k},\begin{pmatrix}
 1 & 0 \\
 0 & \fc^{-1}
 \end{pmatrix}\right).\end{aligned}\]
\end{enumerate}
\end{fact}
Given a branch character $\nu:\Delta\ra \ov{\BQ}^\times$. Then $\varphi_\lambda$ induces a $p$-adic measure $\varphi^\nu_\lambda$ on $\Gamma$ by \[\int_{\Gamma}\phi d\varphi_\lambda^\nu:=\int_{\Cl_\infty'}\nu\phi d\varphi_\lambda.\]
From now on we consider $\nu=1$, we may also denote $\varphi_\lambda$ for $\varphi_\lambda^\nu$.

\subsubsection{Proof of the case $\rank_{\BZ_\ell}\Gamma=1$}
For each $\fs\in \CS$ and $\zeta\in\mu_{\ell^r}$ a $\ell^r$-th primitive root of unity, let
 \[\BH_{\fs}=\sum_{u\in\CO_{F,\fl}/(\varpi_\fl^r)}\zeta^u \BH_\lambda\left|\left[\begin{pmatrix}
1 & 0 \\
0 & N(\fs)
\end{pmatrix}\begin{pmatrix}
1 & \frac{u\varpi_\fl^{\delta_\fl}}{\varpi_\fl^r} \\
0 & 1
\end{pmatrix}\right]\right..\] 
 \begin{prop}\label{inf}Assume for almost all $\varepsilon\in \wh{\Gamma}$ with $\lambda\varepsilon\in \CX^\circ$ one has $\int_{\Gamma}\varepsilon d\varphi_\lambda\equiv 0\pmod{\fm_p}$. Then there exists a positive integer $r$ and $\ell^{r}$-th primitive roots of unity $\zeta$ such that exists an arithmetic progression $\{n_i\}_i$ such that, \[\sum_{\fs\in\CS}\wh{\chi}(\fs_{\Delta})\BH_{\fs}(x_{n_i}(\fs_{\Gamma}^{-1}a))\equiv 0\pmod{\fm_p}\] for all $a\in\BA_{K,f}^{\times,(p)}\cdot\CO_{K,p}^\times$ and all $i$.
\end{prop}
\begin{proof}
Let $r$ be the maximal integer such that $\mu_{\ell^r}\subset \BF_q$. We may assume $r\geq 1$ by enlarging $E$ in \S\ref{E}. Note that for any $n\geq r$ and $\zeta\in \mu_{\ell^n}$,
\[\tr_{\BF_q(\mu_{\ell^n})/\BF_q}(\zeta)=\begin{cases}
0, & \text{if $\zeta\notin \BF_q$},\\
[\BF_q(\mu_{\ell^n}):\BF_q]\zeta,& \text{if $\zeta\in\BF_q$}.
 \end{cases}\]
Let $\{n_i\}_i$ be an arithmetic progression such that for $\varepsilon_i\in \wh{\Gamma}$ with conductor $n_i$ at $\fl$, $\int_{\Gamma}\varepsilon_i d\varphi_\lambda\equiv 0\pmod{\fm_p}$. We may assume $\chi$ is trivial on $\Gamma^{n_i-r}$ and $n_i-r\geq n_0$.
Together with Proposition \ref{Gal}, \eqref{measum} and the fact that $a_\fl$ is a $p$-adic unit, for any $a\in\BA_{K,f}^{\times,(p)}\cdot\CO_{K,p}^\times$,
 \begin{equation}\label{aver}\sum_{\fs\in\CS}\wh{\chi}(\fs_{\Delta})\sum_{t\in\Gamma^{n_i-r}/\Gamma^{n_i}}\varepsilon_i(t)
\BH_\lambda\left|\left[\begin{pmatrix}
1 & 0 \\
0 & N(\fs)
\end{pmatrix}\right]\right.(x_n(t\fs_{\Gamma}^{-1}a))\equiv0 \pmod{\fm_p}.\end{equation}
For $n_i$ sufficiently large, we have an isomorphism
\[t:\CO_{F,\fl}/\varpi_\fl^r\simeq\Gamma^{n_i-r}/\Gamma^{n_i},\quad u\mapsto t_{i,u}:=1+\varpi_{\fl}^{n_i-r}ue_\fl,\] where $1,e_\fl\in\CO_{K,\fl}$ such that $1$, $e_\fl$ ia a relative integral basis of $\CO_{K,\fl}$. Furthermore, we take $e_\fl=\vartheta\varpi_\fl^{\delta_\fl}+a_\fl$ as in \S\ref{el}, then
\begin{equation}\label{unip}\BH_\lambda\left|
\left[\begin{pmatrix}
1 & 0 \\
0 & N(\fs)
\end{pmatrix}\right]\right.(x_{n_i}(t_{i,u}\fs_{\Gamma}^{-1}a))=\BH_\lambda\left|
\left[\begin{pmatrix}
1 & 0 \\
0 & N(\fs)
\end{pmatrix}\begin{pmatrix}
1 & \frac{u\varpi_\fl^{\delta_\fl}}{\varpi_\fl^r} \\
0 & 0
\end{pmatrix}\right]\right.(x_{n_i}(\fs_{\Gamma}^{-1}a)).\end{equation}
By replacing $\{\varepsilon_i\}$ by a smaller subset, we can further assume that $\varepsilon_i(t_u)$ does not depend on $i$, only depends on $u$. Thus it follows from \eqref{aver} and \eqref{unip} that there exists a primitive $\ell^r$-th root of unity $\zeta$ such that for any $a\in\BA_{K,f}^{\times,(p)}\cdot\CO_{K,p}^\times$ and any $n_i$
\[\sum_{\fs\in\CS}\wh{\chi}(\fs_{\Delta})\sum_{u\in\CO_{F,\fl}/(\varpi_\fl^r)}\zeta^u
\BH_\lambda\left|
\left[\begin{pmatrix}
1 & 0 \\
0 & N(\fs)
\end{pmatrix}\begin{pmatrix}
1 & \frac{u\varpi_\fl^{\delta_\fl}}{\varpi_\fl^r} \\
0 & 1
\end{pmatrix}\right]\right.(x_{n_i}(\fs_{\Gamma}^{-1}a))\quad \equiv0 \pmod{\fm_p}.\]
\end{proof}
\begin{proof}[Proof of Theorem \ref{main}]
Assume the condition of Proposition \ref{inf} holds and may assume all $n_i$ have the same parity. In the case $\CX$ is residually self-dual and $\fl$ inert, note that we must have all $n_i$ have the same parity since the root numbers of $\lambda\varepsilon_i$ are constant in $\CX^\circ$. By density of CM point (cf. ~Theorem~\ref{de}), we know $\Xi_V$ is dense in $V^{\# \CS}$. For fixed $a\in\BA_{K,f}^{\times,(S\cup \CD_{K/F})}$, by translating $\Xi_V$ by $\varsigma_{0,f}^{-1}\rho(a)\varsigma_{0,f}$, the CM points are still dense in the $\# \CS$-th power of the geometric irreducible component that contains $\CX_{n_1}(a)$. Note that if $\fl$ is not inert or $\fl$ is inert and all $n_i$ are even, then $X_{n_1}(a)$ and $x_0(a)$ lie in the same geometric irreducible component. If $\fl$ is inert and all $n_i$ are odd, then $X_{n_1}(a)$ and $x_1(a)$ lie in the same geometric irreducible component.
Thus for any $a\in\BA_{K,f}^{\times,(S)}$:
\begin{itemize}
\item [(a)] If $\fl$ is not inert, or $\fl$ is inert and all $n_i$ are even, we have $\BH_{1}$ is zero modulo $\fm_p$ at the irreducible component contains $x_0(a)$, in particular is zero modulo $\fm_p$ at the cusp \[\left[\infty,\begin{pmatrix}
1 & 0 \\
0 & \fc^{-1}(a)
\end{pmatrix}\right];\]
\item [(b)] If $\fl$ is inert, and all $n_i$ are odd. As $\wh{\BE}_{\lambda,k}(\CX_{n_1}(a))=\wh{\BE}_{\lambda,k}\left|\left[\begin{pmatrix}
 1 & 0 \\
 0 & \varpi_\fl
\end{pmatrix}\right]\right.(\CX_{n_1+1}(a))$ and $\CX_{n_1+1}$ lies in the irreducible component that contains the cusp \[\left[\infty,\begin{pmatrix}
 1 & 0 \\
 0 & \fc(a)^{-1}
 \end{pmatrix}\right],\] we have \[\BH_{1}':= \sum_{u\in\CO_{F,\fl}/(\varpi_\fl^r)}\zeta^u\sum_{\fr\in\CR}\wh{\chi}(\fr)\theta^{\kappa}\BE\left|\left[\begin{pmatrix}
1 & 0 \\
0 & N(\fr)
\end{pmatrix}\begin{pmatrix}
1 & \frac{u\varpi_\fl^{\delta_\fl}}{\varpi_\fl^r} \\
0 & 1
\end{pmatrix}\begin{pmatrix}
1 &0 \\
0& \varpi_\fl
\end{pmatrix}\right]\right.\] is zero modulo $\fm_p$ at the cusp \[\left[\infty,\begin{pmatrix}
 1 & 0 \\
 0 & \fc(a)^{-1}
 \end{pmatrix}\right].\]
\end{itemize}
On the other hand, if we are in the first case, note that when $\beta$ prime to $\{v|\CD_{K/F}, v\nmid \fl, \text{and}\ v\notin S_{\text{n-split}}\}$, it follows from Fact \ref{coeff} that the $\beta$-th Fourier coefficient of $\BH_1$ at the cusp $\left[\infty,\begin{pmatrix}
1& 0 \\
0 & \fc(a)^{-1}
\end{pmatrix}\right]$ is given by \[\sum_{u\in\CO_{F,\fl}/(\varpi_\fl^r)}\zeta^u\psi_\fl\left(\frac{\beta u\varpi_\fl^{\delta_\fl}}{\varpi_\fl^r}\right)\beta^\kappa \alpha_\beta\left(\BE,\begin{pmatrix}
1& 0 \\
0 & \fc(a)^{-1}
\end{pmatrix}\right).\] Now choose $\beta\in\CO_{F,(p\fl)}^\times$ such that $\zeta^u\psi_\fl\left(\frac{\beta u\varpi_\fl^{\delta_\fl}}{\varpi_\fl^r}\right)=1$ for all $u$, then the $\beta$-th Fourier coefficient of $\BH_{1}$ at the cusp $\left[\infty,\begin{pmatrix}
 1& 0 \\
 0 & \fc(a)^{-1}
 \end{pmatrix}\right]$ is \[\ell^r\beta^\kappa \alpha_\beta\left(\BE,\begin{pmatrix}
1& 0 \\
0 & \fc(a)^{-1}
\end{pmatrix}\right).\] 

If we are in the second case, when $\beta$ prime to $\{v|\CD_{K/F}, v\nmid \fl, \text{and}\ v\notin S_{\text{n-split}}\}$ and $\ord_\fl(\beta)\geq 1$, from Fact \ref{coeff} we know that the $\beta$-th Fourier coefficient of $\BH_1'$ at the cusp $\left[\infty,\begin{pmatrix}
1& 0 \\
0 & \fc(a)^{-1}
\end{pmatrix}\right]$ is given by
\[\chi_{1,\fl}(\varpi_\fl)|\varpi_\fl^{-1}|_\fl\sum_{u\in\CO_{F,\fl}/(\varpi_\fl^r)}\zeta^u\psi_\fl\left(\frac{\beta u\varpi_\fl^{\delta_\fl}}{\varpi_\fl^{r+1}}\right)\beta^\kappa \alpha_\beta\left(\BE,\begin{pmatrix}
 1& 0 \\
 0 & \fc(a)^{-1}
 \end{pmatrix}\right).\]We may further choose $\beta$ such that $\zeta^u\psi_\fl\left(\frac{\beta u\varpi_\fl^{\delta_\fl}}{\varpi_\fl^{r+1}}\right)=1$ for all $u$. 

 However, by Proposition \ref{mod}, we can further choose $a, \beta$ such that $(\beta,p)=1$ and \[\alpha_\beta\left(\BE,\begin{pmatrix}
1& 0 \\
0 & \fc(a)^{-1}
\end{pmatrix}\right)\nequiv 0\pmod{\fm_p},\] contradiction. 
\end{proof}
\begin{remark}\label{sden}
If without the condition that $\{n_i\}_i$ contain an arithmetic progression in the Theorem \ref{de},
we may call strong density. If the strong density holds, from the same analysis, one could replace the
\enquote{Zariski dense} by \enquote{except finitely many} in the case $\rank_{\BZ_\ell}\Gamma=1$.
\end{remark}
\subsubsection{The case $\rank_{\BZ_\ell}\Gamma>1$}\ \label{rg}
 Identify $\Gamma$ with $\BZ_\ell^{g}$ by choosing a $\BZ_\ell$ basis $\{e_i\}$ of $\Gamma$. Then a finite order character $\varepsilon$ of $\Gamma$ can be viewed as a point in $\mu_{\ell^\infty}^{g}\subset \BG_{m,\ov{\BQ}_\ell}^{g}$ via $\varepsilon\mapsto (\varepsilon(e_i))_i$. Denote $X$ the Zariski closure of those points corresponding to $\varepsilon\in \wh{\Gamma}$ such that \[\int_{\Gamma}\varepsilon d\varphi_\lambda\nequiv 0\pmod {\fm_p}.\] The following Manin-Mumford conjecture for torus over a field of characteristic $0$ is known (cf.~\cite{CMM} \cite{La}).
 \begin{prop}
Let $\BG_m^n$ be a torus over a field of characteristic $0$ and $Y$ a subset of torsion points. Then the Zariski closure of $Y$ is a finite union of torsion cosets. Here a torsion coset means that an algebraic subgroup translated by a torsion point.
\end{prop}
For a torus $\BG_{m,\ov{\BQ}_\ell}^n$, denote $T_\ell(\BG_{m,\ov{\BQ}_\ell}^{n})$ its $\ell$-adic Tate module $\varprojlim_{k}\BG_{m,\ov{\BQ}_\ell}^n[\ell^k]$ and identify $\BG_{m,\ov{\BQ}_\ell}^n[\ell^\infty]$ with $T_\ell(\BG_{m,\ov{\BQ}_\ell}^{n})\otimes\BQ_\ell/T_\ell(\BG_{m,\ov{\BQ}_\ell}^{n})$ via \[\BG_{m,\ov{\BQ}_\ell}^n[\ell^k]\simeq T_\ell(\BG_{m,\ov{\BQ}_\ell}^{n})/\ell^kT_\ell(\BG_{m,\ov{\BQ}_\ell}^{n})\xrightarrow{\cdot\frac{1}{\ell^k}}\frac{1}{\ell^k}T_\ell(\BG_{m,\ov{\BQ}_\ell}^{n})/T_\ell(\BG_{m,\ov{\BQ}_\ell}^{n}),\quad k\in\BZ_{\geq 1}.\]
Let $p$ be a prime distinct from $\ell$ and $q$ is a power of $p$.
\begin{cor}\label{4.3}
Let $X\subset \BG_{m,\ov{\BQ}_\ell}^{n}$ be the Zariski closure of a subset of torsion points such that $X$ is a proper subset, then there exists a basis $e_1,\cdots,e_n$ of $T_\ell(\BG_{m,\ov{\BQ}_\ell}^{n})$, a large power $P$ of $q$, an infinite sequence $0\ll n_1<n_2<\cdots<n_k<\cdots$ containing an arithmetic progression such that the following torsion points are disjoint with $X$:
\[\left\{\sum_{i}P^{m_i}\frac{e_i}{\ell^{n_k}}\in T_\ell(\BG_{m,\ov{\BQ}_\ell}^{n})\otimes \BQ_\ell/T_\ell(\BG_{m,\ov{\BQ}_\ell}^{n})\ \Big|\ m_i, k\in \BZ\right\}.\]
\end{cor}
\begin{proof}By above proposition, there exists $N\in\BZ_{>0}$ such that $NX=\cup_{j} G_j$, where $G_j$ is a proper sub-torus of $\BG_{m,\ov{\BQ}}^n$, we may assume each $G_j$ is connected.
By counting points modulo $\ell^k$, we have the following fact:
\begin{fact}
Let $M_j\subset \BZ_{\ell}^n, j=1,\cdots s$ be sub-$\BZ_\ell$-modules with rank less than $n$. Then there exists a basis $e_1,\cdots, e_n$ of $\BZ_\ell^n$ such that all $e_i$ and $\sum e_i$ are not in $\cup M_j$.
\end{fact} Now let $M_j=T_\ell(G_j)$, we thus know there exists basis $e_i',\cdots, e_n'$ of $T_\ell(\BG_{m,\ov{\BQ}_\ell}^{n})$ such that $\sum e_i'$ is not in $\cup M_j$.

Take $P$ to be a large power of $q$ that is close enough to $1$ in $\BZ_\ell$ such that $\{\sum_{i}P^{m_i}e_i'\ |\ m_i\in\BZ_\ell\}\subset T_\ell(\BG_{m,\ov{\BQ}_\ell}^{n})$ disjoint with $\cup M_j$. As each $G_j$ is a direct factor of $\BG_{m,\ov{\BQ}_\ell}^n$, we have for any $k> 0$, the torsion points $\{\sum_{i}P^{m_i}\frac{e_i'}{\ell^k}\ |\ m_i\in\BZ_\ell\}$ are disjoint with $NX$. Write $N=\ell^{n_0}\cdot a$ with $(a,\ell)=1$ and let $e_i=a^{-1}e_i'$, thus for $k> n_0$, $\{\sum_{i}P^{m_i}\frac{e_i'}{\ell^k}\ |\ m_i\in\BZ_\ell\}$ are disjoint with $X$.
\end{proof}
Now we complete proof in the case $g:=\rank_{\BZ_\ell}\Gamma>1$. Identify $\Hom(\Gamma,\mu_{\ell^\infty})$ with $\BG_{m,\ov{\BQ}_\ell}^{g}[\ell^\infty]$. Recall that $X\subset \BG^{g}_{m,\ov{\BQ}_\ell}$ is the Zariski closure of those images of $\varepsilon$ such that $\int_{\Gamma}\varepsilon\varphi_{\lambda}\nequiv 0\pmod {\fm_p}$.

 \begin{prop}\label{propg}If $X$ is a proper closed subset of $\BG_{m,\ov{\BQ}_\ell}^{g}$, then there exists an infinite sequence of $d$-tuple of characters $(\varepsilon_{1,k},\cdots,\varepsilon_{d,k})_{k}$ on $\Gamma$ , a large power $P$ of $q$ and $r$ such that
\begin{itemize}\item[(1)] $\cap \ker \varepsilon_{i,k}=\Gamma^{\ell^{n_k}}$ for some increasing sequence $\{n_k\}_k$ contains an arithmetic progression; 
 \item[(2)] We have\[\sum_{\fs\in\CS}\wh{\chi}(\fs_{\Delta})\sum_{\substack{[t]_{c_k}\in\Gamma_{c_k}\\ \varepsilon_{i,k}(t)\in \mu_{\ell^r}}}\prod_{i}\varepsilon_{i,k}(t)\wh{\chi}(t)(\wh{\lambda}_{\Sigma_p^c}\wh{\lambda}_{\Sigma_p}^{-1})(t_{\Sigma_p^c})\BH_\lambda\left|\left[\begin{pmatrix}1 & 0 \\ 0 & N(\fs)\end{pmatrix}\right]\right.(x_{c_k}(t\fs_{\Gamma}^{-1}a))\equiv 0\pmod{\fm_p}\] for all $a\in\BA_{K,f}^{\times,(p)}\cdot\CO_{K,p}^\times$. Here $t$ is chosen in $\BA_{K,f}^{\times, (p)}\cdot\CO_{K,p}^\times$ and $c_k$ is maximal such that $\Gamma^{(c_k)}\subset \Gamma^{\ell^{n_k}}$.
\end{itemize}
\end{prop}
\begin{proof}
By Corollary \ref{4.3}, there exists a large power $P$ of $q$, an infinite sequence of $g$-tuple characters $(\varepsilon_{1,k},\cdots,\varepsilon_{g,k})_{k}$ on $\Gamma$ such that $\cap \ker \chi_{i,k}=\Gamma^{\ell^{n_k}}$ for some sequence $\{n_k\}$ that containing an arithmetic progression and \[\int_{\Gamma}\prod_i \varepsilon_{i,k}^{P^{m_i}}d\varphi_\lambda\equiv 0\pmod {\fm_p}\] for all $(m_1,m_2,\cdots,m_g)\in \BZ^{g}$. Let $S_{i,k}\subset P^{\BZ_\ell}$ be a finite subset such that $\varepsilon_{i,k}^{S_{i,k}}$ is the set of all Galois conjugates of $\Frob_P$ of $\varepsilon_{i,k}$. Let $c_k$ such that $\Gamma^{(c_k)}\subset \Gamma^{\ell^{n_k}}$, we have
\[0\equiv \sum_{\fs\in\CS}\wh{\chi}(\fs_{\Delta})\sum_{[t]_{c_k}\in\Gamma_{c_k}}\prod_{i}\sum_{s\in S_{i,k}}\varepsilon_{i,k}^{s}(t) \wh{\chi}(t)(\wh{\lambda}_{\Sigma_p^c}\wh{\lambda}_{\Sigma_p}^{-1})(t_{\Sigma_p^c})\BH_\lambda\left|\left[\begin{pmatrix}1 & 0 \\ 0 & N(\fs)\end{pmatrix}\right]\right.(x_{c_k}(t\fs_{\Gamma}^{-1}a)).\] Now let $r'$ be the integer such that $\BF_P^\times[\ell^\infty]=\mu_{\ell^{r'}}$.
\end{proof}
Suppose $(\ell)=\fl^m$ in $F_\fl$. There exists $n_0$ such that for $n$ sufficiently large, $\Gamma^{(mn+n_0)}\subset \Gamma^{n}\subset \Gamma^{(mn-n_0)}$. Then for $k$ sufficiently large, the set $\{[t]\in\Gamma_{mn_k+n_0}|\ \varepsilon_{i,k}(t)\in \mu_\ell^{r'}\}$ is the same as the set $\Gamma^{\ell^{n_k -r'}}/\Gamma^{(mn_k+n_0)}$. Let $r$ be an integer such that for $k$ sufficiently large, $\Gamma^{\ell^{n_k -r'}}/\Gamma^{(mn_k+n_0)}\hookrightarrow \Gamma^{(mn_k+n_0-mr)}/\Gamma^{(mn_k+n_0)}$.
Note that for $n$ sufficiently large, the logarithm map induces an isomorphism of abelian groups \[1+\varpi_\fl^n\CO_{K,\fl}\simeq \varpi_\fl^n\CO_{K,\fl}.\] Recall that $\CO_{K,n}=\CO_{F,\fl}\oplus\CO_{F,\fl}\varpi_{\fl}^ne_{\fl}$, $e_{\fl}=\vartheta\varpi_{\fl}^{\delta_\fl}+\fa_\fl$ as in \ref{el}. Fix a $\BZ_\ell$-basis $e_1,\cdots,e_g$ of $\CO_{F,\fl}e_{\fl}$.
\begin{fact}
For $n$ sufficiently large, the following is a complete set of representatives of $\Gamma^{(n-rm)}/\Gamma^{(n)}$ :\[t_{n,i,j}=1+\varpi_\fl^{n-mr}je_i, 0\leq j\leq \ell^{r}-1\] and the logarithm map induces an isomorphism of group:\[\Gamma^{(n-mr)}/\Gamma^{(n)}\simeq \frac{\varpi_\fl^{n-mr}\CO_{F,\fl}e_{\fl}}{\varpi_\fl^{n}\CO_{F,\fl}e_{\fl}},\quad t_{n,i,j}\mapsto 1-t_{n,i,j} .\]
\end{fact}

Now we know that there exists a subgroup $U\subset \frac{\CO_{F_\fl}}{(\varpi_\fl^{mr})}$ does not depend on $k$ represented by elements in $\CO_{F,\fl}$ such that $\Gamma^{\ell^{n_k -r'}}/\Gamma^{(mn_k+n_0)}$ can be identified with $\{t_{k,u}:=1+\varpi_\fl^{mn_k+n_0-mr}ue_{\fl}\pmod {1+\varpi_\fl^{mn_k+n_0}}|\ u\in U\}$.
\begin{proof}[Proof of Theorem \ref{main}]
If $X$ is a proper closed subset of $\BG_{m,\ov{\BQ}_\ell}^{g}$.

We may suppose further that
for any $u\in U$, $\prod_{i=1}^g \varepsilon_{i,k}(t_{k,u})$ does not depend on $k$. Denote it by $\zeta_u\in \mu_{\ell^{r'}}$.

We thus have for $k\geq 0$,
\[\displaystyle \begin{aligned}
&\sum_{\fs\in\CS}\wh{\chi}(\fs_\Delta)\sum_{u\in U}\zeta_u\wh{\chi}(t_{k,u})\BH_\lambda\left|\left[\begin{pmatrix}1 & 0 \\ 0 & N(\fs)\end{pmatrix}\right]\right.(x_{mn_k+n_0}(at_{k,u}\fs_{\Gamma}^{-1}))\\
&\equiv \sum_{\fs\in\CS}\wh{\chi}(\fs_\Delta)\sum_{u\in U}\zeta_u\BH_\lambda\left|\left[\begin{pmatrix} 1 & 0 \\0 & N(\fs)\end{pmatrix}\begin{pmatrix}1 & \frac{u\varpi_\fl^{\delta_\fl}}{\varpi_\fl^{mr}} \\ 0 & 1\end{pmatrix}\right]\right.(x_{mn_k+n_0}(a\fs_{\Gamma}^{-1}))\\
&\equiv 0\pmod{\fm_p}
\end{aligned}\] for all $a\in \BA_{K,f}^{\times,(p)}\cdot\CO_{K,p}^\times$.

Here we take $k$ sufficiently large such that $\wh{\chi}(t_{k,u})=1$ for all $u$. The first equality follows from the fact that $n$ sufficiently large, $\varsigma_{\fl,n}^{-1}\rho(1+u\varpi_\fl^{n-mr}e_{\fl})\varsigma_{\fl,n}\in\begin{pmatrix}
1 & \frac{u\varpi_\fl^{\delta_\fl}}{\varpi_\fl^{mr}} \\0 & 1\end{pmatrix}\Gamma_0(\fl)$. Recall here that $\Gamma_0(\fl)=\begin{pmatrix}\CO_{F,\fl}^\times & \fp_\fl^{\delta_\fl} \\ \fp_\fl^{1-\delta_\fl}& \CO_{F,\fl}^\times\end{pmatrix}$. As \[\Gamma(\fl^{2mr}):=\begin{pmatrix}1+\fp_\fl^{2mr} & \fp_\fl^{\delta_\fl+2mr} \\\fp_\fl^{-\delta_\fl+2mr} & 1+\fp_\fl^{2mr} \end{pmatrix}\subset \begin{pmatrix} 1 & \frac{u\varpi_\fl^{\delta_\fl}}{\varpi_\fl^{mr}} \\ 0 & 1 \end{pmatrix}\Gamma_0(\fl)\begin{pmatrix}1 & -\frac{u\varpi_\fl^{\delta_\fl}}{\varpi_\fl^{mr}} \\ 0 & 1 \end{pmatrix},\] we may view the level of $\BH_\lambda\left|\left[\begin{pmatrix} 1 & 0 \\0 & N(\fs) \end{pmatrix}\begin{pmatrix}1 & \frac{u\varpi_\fl^{\delta_\fl}}{\varpi_\fl^{mr}} \\0 & 1\end{pmatrix}\right]\right.$ at $\fl$ is $\Gamma(\fl^{2mr})$.

Applying density of CM points (cf.~Theorem~\ref{de}), we have
\[\displaystyle\BH_{\fs}:=\sum_{u\in U}\zeta_u\BH_\lambda\left|\left[\begin{pmatrix}1 & 0 \\0 & N(\fs) \end{pmatrix}\begin{pmatrix}1 &\frac{u\varpi_\fl^{\delta_\fl}}{\varpi_\fl^{mr}} \\ 0 & 1\end{pmatrix}\right]\right.\] has zero Fourier coefficient at the cusps \[\left[\infty,\begin{pmatrix} 1 & 0 \\ 0 & \fc^{-1}(\fa)\end{pmatrix}\right],\] $a\in\BA_{K,f}^{\times,(S)}$.

On the other hand, note that for any cusp of the form $\left[\infty,\begin{pmatrix} 1 & 0 \\ 0 & \fc^{-1}\end{pmatrix}\right]$ for $\fc\in\BA_{K,f}^{\times,(S)}$ \[
\begin{aligned}
&\alpha_\beta\left(\BH_\lambda\left|\left[\begin{pmatrix} 1 & 0 \\ 0 & N(\fs)\end{pmatrix}\begin{pmatrix} 1 &\frac{u\varpi_\fl^{\delta_\fl}}{\varpi_\fl^{mr}} \\ 0 & 1 \end{pmatrix}\right]\right.,\begin{pmatrix} 1 & 0 \\ 0 & \fc^{-1} \end{pmatrix}\right)\\
&=\psi_\fl\left(\frac{u}{\varpi_\fl^{mr}}\beta\varpi_\fl^{\delta_\fl}\right)\alpha_\beta\left(\BH_\lambda,\begin{pmatrix} 1 & 0 \\0 & N(\fs) \fc^{-1}\end{pmatrix}\right).\\
\end{aligned}\]
 Thus the Fourier expansion of $\BH_{\fs}$ at the cusp $\begin{pmatrix}1 & 0 \\0 & \fc^{-1}\end{pmatrix}$ is equal to
\[\begin{aligned}
&\sum_{\beta}\sum_{u\in U}\zeta_u\psi_\fl\left(\frac{u}{\varpi_\fl^{mr}}\beta\varpi_\fl^{\delta_\fl}\right)\alpha_{\beta}\left(\BH_\lambda,\begin{pmatrix}
 1 & 0 \\ 0 & N(\fs)\fc^{-1} \end{pmatrix}\right)q^\beta\\
&=\sum_{\beta}\beta^\kappa\sum_{u\in U}\zeta_u\psi_\fl\left(\frac{u}{\varpi_\fl^{mr}}\beta\varpi_\fl^{\delta_\fl}\right)\sum_{\fr\in\CR}\wh{\chi}(\fr)\alpha_{\beta}(\BE,\begin{pmatrix}1 &0 \\0& N(\fr\fs)\fc^{-1}
\end{pmatrix})q^\beta.
\end{aligned}\]
By the Proposition \ref{mod}, we can choose $\beta\in F$ such that
\begin{enumerate}
\item [(i)] $\beta$ is prime to $p$ and $\{v|\CD_{K/F}, v\nmid \fl, \text{and}\ v\notin S_{\text{n-split}}\}$;
\item [(ii)] $\zeta_u\psi_\fl\left(\frac{u}{\varpi_\fl^{mr}}\beta\varpi_\fl^{\delta_\fl}\right)=1$ for all $u\in U$;
\item [(iii)] There exists a cusp of the form $\fc:=\fc(\fa)$, $a\in\BA_{K,f}^{\times,(S)}$ such that $\alpha_\beta\left(\BE,\begin{pmatrix} 1 & 0 \\ 0 & \fc^{-1} \end{pmatrix}\right)\neq 0$ modulo $\fm_p$.
\end{enumerate}
For the reason of (ii), since $\zeta_{\cdot}:U\ra \mu_{\ell^{r'}}, u\mapsto \zeta_u$ is a group homomorphism, we can lift it to a group homomorphism $\CO_{F,\fl}/\varpi_\fl^{mr}\ra \mu_{\ell^r}$. On the other hand, $\psi_\fl(\frac{\cdot \beta}{\varpi_\fl^{mr}})$ runs over all characters on $\CO_{F,\fl}/\varpi_\fl^{rm}$ as $\beta$ runs over a set of representatives of $\CO_{F,\fl}/\varpi_\fl^{mr}$. Thus (ii) holds.

Thus for $\fc=\fc(\fa)$, \[\begin{aligned}\alpha_\beta\left(\BH_{1},\begin{pmatrix} 1 & 0 \\ 0 & \fc^{-1} \end{pmatrix}\right)=\beta^\kappa\#U\cdot \alpha_{\beta}\left(\BE,\begin{pmatrix} 1 &0 \\1& \fc^{-1} \end{pmatrix}\right)\neq 0 \pmod{\fm_p}\end{aligned}.\]
Here the first equality follows from property (2) of $\beta$ and definition of $\BH_{\lambda,1}$, the second inequality follows from property (1), (3) of $\beta$ and the fact that $\# U$ is a power of $\ell$. Thus we get contradiction.
\end{proof}
 \begin{remark}\label{ul}
This remark talks about necessity for choice of $U(\fl)$-eigen vector at $\fl$. As seen in the argument, we need some $\beta\in \CO_{F,\fl}$ such that $\psi\left(\frac{\cdot \beta}{\varpi_\fl^{mr-\delta_\fl}}\right)$ could be any given character on $\CO_{F,\fl}/(\varpi_\fl^{mr})$ that surjectivly maps to $\mu_{\ell^r}$ and $\beta$-th Fourier coefficients of the Eisenstein series to be a $p$-adic unit. Such a $\beta$ need not be a unit in $\CO_{F,\fl}$ whenever $\rank_{\BZ_\ell}\Gamma>1$. It seems that $U(\fl)$-eigen test vector is needed compared with spherical vector if $\rank_{\BZ_\ell}\Gamma>1$ and $\lambda_\fl$ unramified. The local Fourier coefficient of $U(\fl)$-eigen vector at $\fl$ is of the form $1_{\CO_{F,\fl}}$, however, the spherical vector at $\fl$ is of the form $1_{\CO_{F,\fl}^\times}+h$, where $h$ is some $p$-integral function supported on $\CO_{F,\fl}\bs\CO_{F,\fl}^\times$ not necessary to take $p$-primitive value.
\end{remark}
\begin{remark}The assumption $\rank_{\BZ_\ell}\Gamma=1$ in the residually self-dual and $\fl$ inert case is due to the fact that in general, the parity of conductor does not behave well for primitive characters on $\Gamma/\Gamma^{\ell^n}$.

\end{remark}

{\bf Data Availability:}

Data sharing not applicable to this article as no datasets were generated or analysed during the current study.

{\bf Conflict of interest statement:}

The author states that there is no conflict of interest.

\end{document}